\documentclass{amsart}

\usepackage{amsmath,amsfonts,amssymb,amsthm,enumerate,tikz-cd,color}
\usetikzlibrary{matrix,arrows}

\theoremstyle{plain}
\newtheorem{Theorem}{Theorem}[section]
\newtheorem*{Theorem2*}{Theorem 1.2}
\newtheorem{Lemma}{Lemma}[section]
\newtheorem{Proposition}[Theorem]{Proposition}
\newtheorem{question}{Question}

\newtheorem{Corollary}[Theorem]{Corollary}
\newtheorem{Conjecture}[Theorem]{Conjecture}

\newtheorem*{theorem*}{Theorem}
\theoremstyle{definition}
\newtheorem{Definition}{Definition}

\theoremstyle{remark}
\newtheorem{Remark}{Remark}
\DeclareMathOperator{\spec}{Spec}
\DeclareMathOperator{\red}{red}

\DeclareMathOperator{\ad}{ad}
\DeclareMathOperator{\End}{End}

\DeclareMathOperator{\Jac}{Jac}

\DeclareMathOperator{\Sym}{Sym}
\DeclareMathOperator{\coker}{coker}
\DeclareMathOperator{\Sel}{Sel}

\DeclareMathOperator{\cry}{cr}

\DeclareMathOperator{\ch}{ch}
\DeclareMathOperator{\Hom}{Hom}
\DeclareMathOperator{\can}{can}
\DeclareMathOperator{\ext}{Ext}

\DeclareMathOperator{\loc}{loc}

\DeclareMathOperator{\Gal}{Gal}
\DeclareMathOperator{\gr}{gr}

\DeclareMathOperator{\Mat}{Mat}

\DeclareMathOperator{\rk}{rk}
\DeclareMathOperator{\fil}{fil}
\DeclareMathOperator{\Aut}{Aut}

\DeclareMathOperator{\Ker}{Ker}

\DeclareMathOperator{\res}{Res}
\DeclareMathOperator{\Div}{Div}

\newcommand{\Q}{\mathbb{Q}}
\newcommand{\Z}{\mathbb{Z}}

\newcommand{\dr}{\textrm{dR}}
\newcommand{\F}{\mathbb{F}}

\DeclareMathOperator{\NS}{NS}
\DeclareMathOperator{\codim}{codim}
\DeclareMathOperator{\Ind}{Ind}

\begin{document}

\begin{abstract}
We give new instances where Chabauty--Kim sets can be proved to be finite, by developing a notion of ``generalised height functions'' on Selmer varieties. We also explain how to compute these generalised heights in terms of iterated integrals and give  the first explicit nonabelian Chabauty result for a curve $X/\Q$ whose Jacobian has Mordell-Weil rank \emph{larger than} its genus.
\end{abstract}

\title[Quadratic Chabauty and Rational Points II]{Quadratic Chabauty and Rational Points II: Generalised Height Functions on Selmer Varieties}

\author{Jennifer S. Balakrishnan}
\address{Jennifer S. Balakrishnan, Department of Mathematics and Statistics, Boston University, 111 Cummington Mall, Boston, MA 02215, USA}
\email{jbala@bu.edu}
\author{Netan Dogra}
\address{Netan Dogra, Department of Mathematics, Imperial College London, London SW7 2AZ, UK }
\email{n.dogra@imperial.ac.uk}

\date{\today}
\maketitle

\tableofcontents


\date{\today}
\maketitle


\section{Introduction}\label{sec:intro}

Given a smooth projective curve $X$ of genus $g\geq 2$ over a number field $K$, it is known by Faltings' theorem that the set $X(K)$ of its $K$-rational points is finite, but in general there is no known method to determine this set explicitly. When the Mordell--Weil rank of the Jacobian $J$ of $X$ is less than $g$, the method of Chabauty \cite{chabauty:1941}, made effective by Coleman \cite{coleman:1985}, can determine explicit finite sets of $\mathfrak{p}$-adic points containing the set $X(K)$. In many cases, this can give a computationally feasible approach to determine the set of rational points \cite{mccallum:2012}. 

In a series of papers  \cite{kim:2005, kim:2009, kim:2008}, Kim proposed a generalisation of the Chabauty--Coleman method, which gives a nested sequence 
$$X(K_{\mathfrak{p}} )_1 \supset X(K_{\mathfrak{p}} )_2 \supset \cdots \supset X(K)$$ of sets $X(K_{\mathfrak{p}})_n$ of $\mathfrak{p}$-adic points, each containing the set $X(K)$, such that the ``depth 1'' set $X(K_{\mathfrak{p}} )_1$ is exactly the one arising from the Chabauty--Coleman method. Here $\mathfrak{p}$ is a prime of $K$ lying above a prime $p$ which splits completely and for which $X$ has good reduction.  When $K=\Q $, Kim  \cite{kim:2009} showed that  the Bloch--Kato conjectures imply the finiteness of $X(\mathbb{Q}_p)_n$ for $n$ sufficiently large.  Coates and Kim \cite{kim:2010} proved this eventual finiteness (again for $K=\Q $) in the case when $J$ has complex multiplication.  Recently, Ellenberg and Hast \cite{EH:2017} extended this result to give a new proof of Faltings' theorem for curves $X/\Q $ which are solvable covers of $\mathbb{P}^1$.

In this paper, we consider two questions about the depth 2 set $X(K_{\mathfrak{p}})_2$, continuing our previous investigation \cite{balakrishnan:2016}:

\begin{question} When can $X(K_{\mathfrak{p}})_2$ be proved to be finite?\end{question}
\begin{question} When can $X(K_{\mathfrak{p}})_2$ be computed explicitly? \end{question}

The key technical construction which we use to study these question is presented in Section \ref{sec:ht}. We define the notion of \textit{equivariant generalised $p$-adic heights}, inspired by Nekov\'a\v r's construction of $p$-adic height functions \cite{Nek:1993}. We give a brief explanation of Nekov\'a\v r's construction for divisors on $X$.
Recall that the local height on $X$ is usually defined to be a pairing on divisors of degree zero with disjoint support, and the global height is given by the sum of local heights, which only depends on the class of the divisors in the Picard group of $X$. In Nekov\'a\v r's construction, local and global heights are constructed as functions on isomorphism classes of \textit{mixed extensions}. Recall that the $\Q _p $-Kummer map allows us to associate to a divisor $D$ in $\Div ^0 (X)$ a Galois cohomology class $\kappa (D)\in H^1 (G_K ,V)$, where $V:=H^1 _{\acute{e}t}(X_{\overline{\Q }},\Q _p (1))$. Equivalently, we may think of $\kappa (D)$ as an isomorphism class of Galois representations of the form
 \[
 \rho =\left( \begin{array}{cc}
 1 & 0 \\
 * & \rho _V \\
 \end{array} \right) ,
 \]
 where $\rho _V $ is the Galois representation associated to $V$. Nekov\'a\v r associates to a pair of divisors $D_1 ,D_2 $ with disjoint support a Galois representation of the form 
 \[
 \rho =\left( \begin{array}{ccc}
 1 & 0 & 0 \\
 * & \rho _V & 0 \\
 *  & * & \chi \\ \end{array}
 \right)
 \]
 where $\chi $ is the cyclotomic character. A Galois representation of this form is referred to as a \textit{mixed extension} with graded pieces $\Q _p ,V,\Q _p (1)$. Nekov\'a\v r's $p$-adic heights are functions on isomorphism classes of such mixed extensions (with some conditions at primes above $p$). For each prime $v$, Nekov\'a\v r defines a local height function $h_v$ on mixed extensions of $G_v$-representations. The global height is then the sum of the local heights, and class field theory implies this global height is bilinear in the two off-diagonal $H^1 (G_K ,V)$-classes.

From the point of view of the Chabauty--Kim method, the interesting feature of the $p$-adic height is that this bilinear structure gives an algebraic criterion for a collection $M_v $ of mixed extensions of $G_v$-representations to come from a global $G_K$-representations. More precisely, in our previous work, we showed that if the Picard number of the Jacobian is bigger than 1, then the Chabauty--Kim method can be used to associate to each point $x$ (over any extension $L|\mathbb{Q}$) a 
$G_L$-representation $A_Z (x)$ which is a mixed extension with graded pieces $\Q _p ,V,\Q _p (1)$. We then obtain an obstruction to an adelic point $(x_v )\in \prod X(K _v )$ coming from a global point in $X(K)$: the associated mixed extensions $A_Z (x_v )$ must come from a global mixed extension, and hence there must be a `bilinear relation' between the three $*$ entries (as the contributions from primes away from $p$ are small, this can essentially be thought of as an obstruction to an element of $\prod _{v|p}X(K_v )$ coming from $X(K)$). This obstruction defines a subset intermediate between $X(K_{\mathfrak{p}})_1 $ and $X(K_{\mathfrak{p}})_2 $. Furthermore, by relating the mixed extensions $A_Z (x)$ to the ones arising in Nekov\'a\v r's theory, we gave a formula for $h_v (A_Z (x))$ as a local height pairing $h_v (A_Z (x-b,D_Z (x-b))$ between divisors. This was inspired by earlier uses of $p$-adic heights to obtain quadratic Chabauty formulae for integral points on elliptic and hyperelliptic curves  in papers of Kim \cite{kimrank1:2010} and of the first author with Kedlaya and Kim \cite{BKK:2011} and Besser and M\"{u}ller \cite{balakrishnan:2013}.

To recover $X(K_{\mathfrak{p}})_2$, we need to consider more general mixed extensions (with graded pieces $\Q _p ,V$ and $W$ for $W$ a quotient of $\wedge ^2 V$). The key technical construction of this paper is the definition of a \textit{generalised height} for such mixed extensions. As in the classical case, generalised heights gives a simple algebraic criterion for a collection of local mixed extensions to come from a global mixed extensions (see Lemma \ref{heights_prove_global} for a precise formulation). Via a twisting construction explained in Section \ref{subsec:twist}, one may associate to each point $z$ of $X(K)$ a mixed extension $A(b,z)$. This gives an explicit equation for $X(K_{\mathfrak{p}})_2$ (see Lemma \label{lemma13}), and in particular gives a necessary condition for an adelic point to come from a rational point. 
The relation between the approach of this paper (which we refer to below as ``QC2'') and previous related papers (``QC0'' \cite{balakrishnan:2013} and ``QC1'' \cite{balakrishnan:2016}) may be summarised as follows:

\begin{center}
 \begin{tabular}{r c   c   c }
 \hline
    & QC0 & QC1 & QC2 \\
    \hline
\begin{tabular}{@{}r@{}}Scope (proof \\of finiteness\\ 
of a superset)\end{tabular}     &  \begin{tabular}{@{}c@{}} $X(\mathbb{Z})$ for $X/\mathbb{Q}$\\ hyperelliptic  \\
    with $r =g$\end{tabular}&  \begin{tabular}{@{}c@{}}$X(K)$ for $X/K$ with  \\  $r<g+\rho (J)-1$, \\$K=\mathbb{Q}$ or im. quad.\end{tabular} &  \begin{tabular}{@{}c@{}}$X(K)$ for $X/K$  \\ satisfying hypotheses \\ of Theorems \ref{thm1} or \ref{thm2} \end{tabular}\\
   \hline
\begin{tabular}{@{}r@{}}Bilinear\\ structure \\used \end{tabular} & \begin{tabular}{@{}c@{}}Coleman-Gross\\
    $p$-adic height \cite{coleman:1989}\end{tabular} & \begin{tabular}{@{}c@{}} Nekov\'a\v r $p$-adic  \\ height \cite{Nek:1993} on  \\ $M_{f,T_0 } (G_{K,T} ; \Q _p ,V,\Q _p (1))$  \end{tabular}  &  \begin{tabular}{@{}c@{}}Generalised height \\ functions (\S 3) on \\ $M_{f,T_0 } (G_{K,T}; \Q_p, V ,W) $,\\  inspired by Nekov\'a\v r \end{tabular}  \\
    \hline
           \begin{tabular}{@{}r@{}}Local \\computation \\  \end{tabular} & \begin{tabular}{@{}c@{}} $h_v (z-\infty ,z-\infty )$ \end{tabular} &  \begin{tabular}{@{}c@{}} $h_v (z-b,D(b,z))$  \end{tabular} & \begin{tabular}{@{}c@{}}$h_v (A(b,z))$   \end{tabular} \\ \hline
\\
    \end{tabular}
 \end{center}
 
 Here $M_{f,T_0 } (G_{K,T}; \mathbb{Q}_p ,V,W)$ denotes the set of isomorphism classes of mixed extensions of $G_{K,T}$ representations with graded pieces $\Q _p ,V,W$ which are crystalline at all primes above $p$. See Section \ref{sec:ht} for a precise definition.

\subsection{Main results}

To address Question 1, in Section 2, we begin by recalling when, for $K=\mathbb{Q}$, finiteness of $X(\mathbb{Q}_p )_2 $ is implied by the Bloch--Kato conjectures. We also note some elementary extensions of our previous results \cite{balakrishnan:2016} on finiteness of $X(\mathbb{Q}_p )_2$ when the N\'eron--Severi group of its Jacobian is large.  We then use generalised heights to prove new finiteness results when the curve $X$ is hyperelliptic and satisfies ``Manin--Demjanenko''-type conditions, i.e., that there are isogeny factors occurring in the Jacobian with large multiplicity. To state the first main theorem, let $K=\mathbb{Q}$ or an imaginary quadratic field. We introduce the notational convention that, for an abelian variety $A$ over $K$, $$
\rho _f (A):=\left\{ \begin{array}{cl} 
\rk \NS (A)+\rk (\NS (A_{\overline{\mathbb{Q}}})^{c=-1}) & \mathrm{if}\; K=\mathbb{Q}, \\
\rk \NS (A) & \mathrm{else}, \\ \end{array} \right. $$
and 
$$
e(A):= \left\{ \begin{array}{cl}
\rk \End ^0 (A) +\rk (\End (A_{\overline{\mathbb{Q}}})^{c=-1}) & \mathrm{if}\; K=\mathbb{Q}, \\
\rk \End ^0 (A) & \mathrm{else}. \\
\end{array} \right. $$ 
Here $\NS (A)$ denotes the N\'eron--Severi group of $A$ and $\NS (A_{\overline{\mathbb{Q}}})^{c=-1}$ the subspace of $\NS (A_{\overline{\mathbb{Q}}})\otimes \mathbb{Q}$ on which complex conjugation acts by $-1$. As usual, $\End^0(A):=\mathbb{Q}\otimes \End (A)$, where $\End (A)$ denotes endomorphisms of $A$ defined over $K$.
\begin{Theorem}\label{thm1}
Let $X/K$ be a hyperelliptic curve and suppose $J$ is isogenous to $A^d \times B$, where $A$ is an abelian variety of rank $r$. If 
$$
\rho _f (A)d +d(d-1)e(A)/2-1 >\min \{d(r-\dim (A)),r^2 -\dim (A)^2 \},
$$
then $X(K_{\mathfrak{p}} )_2 $ is finite.
\end{Theorem}

In Section \ref{example_shaska}, we give an  example of a genus 5 curve $X/\Q(i)$ which satisfies the hypotheses of the theorem but does not satisfy the Chabauty--Coleman bound.

Theorem \ref{thm1} is somewhat reminiscent of the following result, due to Demjanenko when $A$ is an elliptic curve and Manin in general \cite{manin:1969}, \cite[\S 5.2]{serre:1997}.
\begin{theorem*}[Manin--Demjanenko]
Let $A$ be a simple abelian variety of rank $r$, with $\dim \End ^0 (A)=h$. If $J$ is isogenous to $A^d \times B$, with $d>r/h$, then $X(\mathbb{Q})$ is finite and may be computed effectively.
\end{theorem*}

To address Question 2, we use generalised heights to obtain equations for Selmer varieties at depth 2, and hence for the set $X(K_{\mathfrak{p}} )_2$. The equations are given in terms of height functions on Selmer varieties in Proposition \ref{exactformula1}. To get from this proposition to an explicit computation, we need a way to compute the local generalised heights of the mixed extensions $A(b,z)$ arising from the twisting construction of section \ref{sec:ht}. In this paper we focuse on the problem of describing the local heights at primes above $p$. This is done in section \ref{sec:hodgefiltration} in three stages. The definition of the local heights is in terms of certain associated filtered $\phi $-modules $D_{\cry }(A(b,z))$. First, one uses a $p$-adic comparison theorem due to Olsson \cite{olsson:2011} to relate this to a more tractable filtered $\phi $-module $A^{\dr }(b,z)$, which is the fibre at $z$ of a flat connection $\mathcal{A}^{\dr }$. The filtration on $A^{\dr }(b,z)$ is then computed in section  \ref{computehodge} by computing a filtration by sub-bundles on $\mathcal{A}^{\dr }$. Finally, the $\phi $-action is computed in section \ref{frobenius_structure}, when $X$ is a hyperelliptic curve, in terms of iterated integrals. This is used to render the equations for $X(K_{\mathfrak{p}})_2$ in terms of $p$-adic heights explicit (see Proposition \ref{prop5} for a more general result). We use this to give the first explicit nonabelian Chabauty result for a curve $X/\Q$ which has Mordell--Weil rank \emph{larger than} its genus.

As an example, we  consider the family of genus 2 curves
\begin{equation}\label{Xa}
X=X_a :y^2 =x^6 +ax^4 +ax^2 +1,
\end{equation}
which was previously studied by Kulesz, Matera, and Schost \cite{Kulesz:2004}. We prove results controlling the set of $K$-rational points of $X_a $ for $a\in K_0 $, where $K_0 =\mathbb{Q}$ or a real quadratic field and $K$ is a totally real extension of $K_0$.
We consider the case where the Mordell--Weil rank over $K$ of the associated elliptic curve 
\begin{equation}\label{Ea}
E=E_a :y^2 =x^3 +ax^2 +ax+1
\end{equation}
is two. Consider the maps $X\to E$ given by $f_1: (x,y)\mapsto (x^2 ,y)$ and $f_2: (x,y) \mapsto (x^{-2},yx^{-3})$. As the rank of $E_t$ over the function field $\mathbb{Q}(t)$ is 1, generated by the point $b=(0,1)$ \cite[Prop. 1]{Kulesz:2004}, for all but finitely many values of $a$, the specialisation $E_a$ over $K_0 $ has the point $b$ of infinite order. By the conjectured equidistribution of parity, one expects to find many values of $a$ for which $E_a(K )$ has rank 2.

Note that the Jacobian of $X$ is isogenous to $E\times E$, and hence, when the rank of $E$ is 2, the Chabauty--Coleman method  does not apply.  When the rank of $E$ is 2, we show that $X(K_{\mathfrak{p}})_2$ is finite and give equations for a finite set containing it. 

To state the theorem, let  $\omega _i =\frac{x^i}{2y}dx$ and let $w$ denote the hyperelliptic involution. 
Following Liu, we say that a genus 2 curve has \emph{potential type V reduction} at $v$ if, in an extension $L_w |K_v $ over which the curve acquires stable reduction, the special fibre of its stable model is isomorphic to two genus 1 curves meeting at a point.  For simplicity, in the introduction we state a special case of the theorem, under a simplifying assumption on the reduction type of $X$. The general statement may be found in Section \ref{sec:algorithms}.

\begin{Theorem}[Special case]\label{thm2} Let $K_0$ be $\mathbb{Q}$ or a real quadratic field. Let $K|K_0 $ be a totally real extension. Let $X/K_0 $ be a genus 2 curve in the family $y^2 =x^6 +ax^4 +ax^2 +1$ whose Jacobian has Mordell--Weil rank 4 over $K$. Suppose $p$ is a prime of $\mathbb{Q}$ such that 
\begin{itemize}
\item The prime $p$ splits completely in $K|\mathbb{Q}.$
\item The curve $X$ has good reduction at all primes above $p$, and the action of $G_K$ on $E[p]$ is absolutely irreducible.
\item If $E$ has complex multiplication by a CM extension $L$, then $L$ is not contained in $K(\mu  _p )$.
\end{itemize}
Suppose that $X$ has no primes of potential type $V$ reduction.
Suppose $z_0 $ is a point in $X(K)$ such that $f_1 (z_0 )\wedge f_2 (z_0 )$ is of infinite order in $\wedge ^2 E(K)$. 
Then $X(K)$ is contained in the finite set
of $z$ in $X(K_{\mathfrak{p}})$ satisfying $G(z) = 0$, where
\begin{align*}
G(z) &= F_1 (z)F_2 (z_0 )-F_1 (z_0 )F_2 (z), 
\end{align*}
with  $b=(0,1)$ and 
\begin{align*}
F_1 (z) & = \int ^z _b (\omega _0 \omega _1 -\omega _1 \omega _0 )+\frac{1}{2}\int ^z _b \omega_0 \int ^b _{w(b)}\omega _1, \\
F_2 (z) &=2\int ^z _{b }(-\omega _0 \omega _3 +a \omega _1 \omega _2 +2\omega _1 \omega _4 )-\frac{1}{2}x(z) -\int ^Fz _{b }\omega _0 \int ^{b }_{w(b)}\omega _3.\end{align*}
\end{Theorem}
We briefly indicate the techniques used in the proof of the theorem (precise definitions may be found in subsequent sections).
The isogeny gives an isomorphism $$V=T_p \Jac (X)\otimes \mathbb{Q}_p \simeq V_E \oplus V_E,$$ where $V_E =T_p E\otimes \mathbb{Q}_p $. The quotient of the fundamental group of $X$ used is an extension
$$
1\to \Sym ^2 V_E \to U \to V\to 1 .
$$
The first step of the proof is to prove non-density of the localisation map 
$$
\loc _{\mathfrak{p}}:\Sel (U)\to H^1 _f (G_{\mathfrak{p}} ,U).
$$
from the Selmer variety of $U$ to the local cohomology variety $H^1 _f (G_{\mathfrak{p}},U)$.
In the case $K=\Q$ and $p>3,$ we know $H^1 _f (G_{K,T} ,\Sym ^2 V_E )=0$ by Flach \cite{flach:1992}.
In general, by Freitas, Le Hung, and Siksek \cite{freitas:2015} we know that $E_a /K_0$ is modular. Under our assumptions, the vanishing of the Selmer group of $\Sym ^2 V_E $ follows from modularity lifting results \cite{allen:2016}. This implies that the dimension of the global Selmer variety is 4. By $p$-adic Hodge theory, the local Selmer variety has the same dimension. Hence non-density cannot be proved by a dimension argument. Instead, it is deduced using the notion of a generalised height function which is equivariant with respect to the the action of $\Mat _2 (\mathbb{Q}_p )$ on $V\simeq V_E \oplus V_E $. 

\subsection{Notation}\label{subsec:notation}
We follow slightly different notational conventions to those used in \cite{balakrishnan:2016}, to make our notation more compatible with standard references such as \cite{Bloch-Kato:1990}. $X$ is a smooth projective curve with good reduction outside a set of primes $T_0$, and $p$ is a rational prime that splits completely in $K$ and such that $X$ has good reduction at all primes above $p$. We fix a prime $\mathfrak{p}$ above $p$, and define $T:=T_0 \cup \{v|p \}$.

For $v$ a prime not above $p$, define \begin{align*}
H^1 _f (G_v ,W) &  :=\Ker (H^1  (G_v ,W)\to H^1  (I_v ,W)), \\
H^1 _g (G_v ,W) & := H^1  (G_v,W). \\
\end{align*}

For $\mathfrak{p}$ a prime above $p$, define \begin{align*}
H^1 _f (G_{\mathfrak{p}} ,W) &  :=\Ker (H^1  (G_{\mathfrak{p}} ,W)\to H^1  (G_{\mathfrak{p}} ,W\otimes B_{\cry })), \\
H^1 _g (G_{\mathfrak{p}} ,W) & := \Ker (H^1  (G_{\mathfrak{p}} ,W)\to H^1  (G_{\mathfrak{p}} ,W\otimes B_{\dr })), \\
\end{align*}

We define the global versions
\begin{align*}
H^1 _f (G_{K,T} ,W) & :=\{c \in H^1 (G_{K,T} ,W) : \prod _{v\in T} \loc _v (c)\in \prod _{v\in T} H^1 _f (G_v ,W)\}, \\
H^1 _g (G_{K,T} ,W) & :=\{c \in H^1 (G_{K,T} ,W) : \prod _{v\in T} \loc _v (c)\in \prod _{v\in T} H^1 _f (G_v ,W)\}. \\
\end{align*}
More generally, for $S\subset T$ we may define global Galois cohomology groups with conditions intermediate between $H^1 _f $ and $H^1 _g$: 
\[
H^1 _{f,S}(G_{K,T},W):=\{ c\in H^1 (G_{K,T},W): \prod _{v\in T}\loc _v (c) \in \prod _{v\in S}H^1 _g (G_v ,W)\times \prod _{v\in T-S}H^1 _f (G_v ,W)\}.
\]
The reason for introducing these different conditions is that in the theory of Selmer varieties we use cohomology classes which may be ramified at primes of bad reduction---and hence may not lie in $H^1 _f$---but the dimensions of the Selmer varieties (which are of central importance in proving finiteness results) will in some sense be governed by $H^1 _f$.

For finite-dimensional continuous $\mathbb{Q}_p $-representations $W_1 ,W_2 $ of a topological group $G$, we identify the vector spaces 
$H^1 (G,W_1 ^* \otimes W_2 )$ and $\ext ^1 (W_1 ,W_2 )$ in the usual way. Via this identification, we define subspaces such as $\ext ^1 _f (W_1 ,W_2 )$.

If $U$ is a unipotent group over $\mathbb{Q}_p $ with a continuous action of $G_{K,T}$ which is crystalline at all primes above $p$, we similarly define $H^1 _f (G_{K,T} ,U)$ as the set of isomorphism classes of $G_{K,T}$-equivariant $U$-torsors which are crystalline at all $v$ above $p$ and unramified at all $v$ prime to $p$ (and analogously for $H^1 _g $ and $H^1_{f,S}$).

We make repeated use of the twisting construction in nonabelian cohomology, as in \cite[I.5.3]{serregc:1997}. For topological groups $U$ and $W$, equipped with a continuous homomorphism $U\to \Aut (W)$, and a continuous left $U$-torsor 
$P$, we shall denote by $W^{(P)}$ the group obtained by twisting $W$ by the $U$-torsor $P$:
\[
W^{(P)}:=W\times _U P .
\]
Similarly if $P$ is a continuous right $U$-torsor we define
$
^{(P)}W:=P\times _U W .
$
For $U$ a group with a continuous action of a topological group $\Gamma $, $H^1 (\Gamma ,U)$ will denote the set of isomorphism classes of $\Gamma $-equivariant left $U$-torsors.

\section{The Chabauty--Kim method}\label{sec:chabautykim}

Let $K$ be a number field, and $X,T,T_0 $ as in section \ref{subsec:notation}.
Given a rational point $b$ in $X(K)$, let $\pi _1 ^{\acute{e}t,\mathbb{Q}_p}(\overline{X},b)$ denote the unipotent $\mathbb{Q}_p $-\'etale fundamental group of $\overline{X}:=X\times _{\Q }\overline{\Q }$ with basepoint $b$. Let 
$$
\pi _1 ^{\acute{e}t,\mathbb{Q}_p }(\overline{X},b) \supset U^{(1)}=[\pi _1 ^{\acute{e}t,\mathbb{Q}_p }(\overline{X},b) ,
\pi _1 ^{\acute{e}t,\mathbb{Q}_p }(\overline{X},b) ]\supset U^{(2)}=[U^{(1)},\pi _1 ^{\acute{e}t,\mathbb{Q}_p }(\overline{X},b)] \supset \cdots
$$
denote the central series filtration of $\pi _1 ^{\acute{e}t,\mathbb{Q}_p }(\overline{X},b)$. Associated to this filtration we have the groups 
$$U_n:=U_n (b):=\pi _1 ^{\acute{e}t,\mathbb{Q}_p }(\overline{X},b)/U^{(n)}, \qquad U[n] :=\Ker (U_n \to U_{n-1}),$$
and the $U_n $-torsor
$$
P_n (b,z):=\pi _1 ^{\acute{e}t}(\overline{X};b,z)\times _{\pi _1 ^{\acute{e}t}(\overline{X},b)}U_n (b).
$$
Then the assignment $z\mapsto [P_n (b,z)]$ defines a map
$$
j_n :X(K)\to H^1 (G_{K,T} ,U_n (b)).
$$
One of the fundamental insights of the theory of Selmer varieties is that the cohomology spaces $H^1 (G,U_n (b))$ carry a much 
richer structure than merely that of a pointed set, and that this extra structure has Diophantine applications.
For the following theorem we take $G$ to be either $G_v $ or $G_{K,T}$:
\begin{Theorem}[Kim \cite{kim:2005}]
Let $U$ be a finite-dimensional unipotent group over $\mathbb{Q}_p $, admitting a continuous action of $G$. Let 
\[
U=U^{(0)} \supset U^{(1)}\supset \ldots
\] 
denote the central series filtration of $U$.
Suppose 
$H^0 (G,U^{(i)} /U^{(i+1)})(\mathbb{Q}_p )=0$ for all $i$. Then the functors
\begin{equation}\nonumber
R\mapsto H^j (G,U(R)) , j=0,1,
\end{equation}
is represented by an affine algebraic variety over $\mathbb{Q}_p $, such that, for all $i$, the exact sequence
\[
H^1 (G,U^{(i)}/U^{(i+1)})\to H^1 (G,U/U^{(i+1)})\to H^1 (G,U/U^{(i)})
\]
is a diagram of schemes over $\mathbb{Q}_p $. 
\end{Theorem}
In this paper, we will never distinguish between a cohomology variety and its $\mathbb{Q}_p$-points. We henceforth let $U$ denote a Galois stable quotient of $U_n $, whose abelianisation equals $U_1 $.
Since the abelianisation of $U(\mathbb{Q}_p )$ has weight $-1$, it satisfies the hypotheses of the theorem, and hence 
$H^1 (G ,U)$ has the structure of the $\mathbb{Q}_p $-points of an algebraic variety over $\mathbb{Q}_p$.

To go from the cohomology varieties $H^1 (G_{K,T} ,U)$ to Selmer varieties, one must add local conditions.
 Let $P(z)$ denote the pushout of $\pi _1 (\overline{X};b,z)$ along $\pi _1 (\overline{X},b)\to U$.
Then for each $v$ prime to $p$, there is a 
\textit{local unipotent Kummer map}
\[
\begin{array}{cc}
j_v :X(K_v ) \to H^1 (G_v ,U) ; & x \mapsto [P(x)]
\end{array}
\]
which is trivial when $v$ is a prime of good reduction and has finite image in general, by work of Kim and Tamagawa \cite{kim:2008}.
For $v$ above $p$, and $x$ in $X(K_v)$, the torsor $P(x)$ is crystalline by Olsson \cite{olsson:2011}, and we define $j_v$ to be the map
\[
\begin{array}{cc}
j_v :X(K_v ) \to H^1 _f (G_v ,U) ; & x \mapsto [P(x)].
\end{array}
\]
There is then a commutative diagram
$$
\begin{tikzpicture}
\matrix (m) [matrix of math nodes, row sep=3em,
column sep=3em, text height=1.5ex, text depth=0.25ex]
{X(K) & H^1 (G_{K,T} ,U) \\
\prod _{v\in T} X(K_v ) & \prod _{v\in T} H^1 (G_v ,U) \\};
\path[->]
(m-1-1) edge[auto] node[auto]{} (m-2-1)
edge[auto] node[auto] { $j$ } (m-1-2)
(m-2-1) edge[auto] node[auto] {$\prod j_v $ } (m-2-2)
(m-1-2) edge[auto] node[auto] {$\prod \loc _v $} (m-2-2);
\end{tikzpicture}$$
Kim \cite{kim:2005} also showed that the
localisation morphisms are morphisms of varieties, and the set of crystalline cohomology classes has the structure of the $\mathbb{Q}_p$-points of a variety. Since, at any $v$ prime to $p$, the image of $X(K_v  )$ in $H^1 (G_v ,U)$ is finite \cite{kim:2008}, we may define a subvariety $\Sel (U)$ 
of $H^1 (G_{K,T} ,U)$ to be the set of cohomology classes $c$ satisfying the following conditions:
\begin{itemize}
 \item {$\loc _v (c)$ comes from an element of $X(K_v )$ for all $v$ prime to $p$},
 \item {$\loc _v (c)$ is crystalline for all $v$ above $p$}, and
 \item {the projection of $c$ to $H^1 (G_{K,T} ,V)$ lies in the image of $\Jac(X)(K)\otimes \mathbb{Q}_p $}.
\end{itemize}
For a prime $\mathfrak{p}$ above $p$, we define $X(K_{\mathfrak{p}})_U :=j_{\mathfrak{p}}^{-1}\loc _{\mathfrak{p}}\Sel (U)$.
We shall refer to this variety as the \textit{Selmer variety} associated to $U $. We include the last condition, which is somewhat non-standard and perhaps in conflict with the ``Selmer'' prefix, so as to be able to make statements about relations between the set of weakly global points $X(K_{\mathfrak{p}})_U $ and the rank of the Jacobian of $X$ which are not conditional on the finiteness of the $p$-part of the Shafarevich--Tate group. When $U=U_n $, we write $\Sel _n $ for $\Sel (U_n )$ and $X(K_{\mathfrak{p}})_n $ for $X(K_{\mathfrak{p}})_{U_n }$. 

It is often convenient to break up the Selmer variety by first fixing an element $\alpha =(\alpha _v )\in \prod _{v\in T_0 }j_v (X(K_v ))$, and defining $\Sel (U)_{\alpha }$ to be the subvariety of $\Sel (U)$ consisting of cohomology classes whose localisation at $v\in T_0 $ is equal to $\alpha _v $. We similarly write $X(K_{\mathfrak{p}})_\alpha $. We call the tuple $\alpha $ a \emph{collection of local conditions}.

\begin{Lemma}[{\cite[Lemma 2.6]{balakrishnan:2016}}]\label{old_lemma}
Let $\alpha _1 ,\ldots ,\alpha _N \in \Sel (U)$ be a set of representatives for the image of $\Sel (U)$ in $\prod _{v\in T_0 }j_v (X(K_v ))$. Then
\[
\Sel (U) \simeq H^1 _f (G_{K,T},U^{(\alpha _i )})' ,
\]
where $H^1 _f (G_{K,T}U^{(\alpha _i )})'$ denotes the subvariety of $H^1 _f (G_{K,T},U^{(\alpha _i )})$ consisting of crystalline torsors whose image in $H^1 _f (G_{K,T},V)$ lies in the image of $J(K)\otimes \Q _p $.
\end{Lemma}

\begin{Lemma}
Suppose
\begin{equation}\label{eqn:key_inequality}
\dim D_{\dr}(W)/F^0 -\dim H^1 _f (G_{K,T},W) > r-g,
\end{equation}
then $X(K_{\mathfrak{p}})_2 $ is finite.
\end{Lemma}
\begin{proof}
By \cite{kim:2009}, it is enough to prove that equation \eqref{eqn:key_inequality} implies
\[
\dim D_{\dr}(U)/F^0 > \dim \Sel (U).
\]
Since 
\[
\dim D_{\dr}(U)/F^0 =\dim D_{\dr}(V)/F^0 +\dim D_{\dr}(W)/F^0 ,
\]
and $\dim (D_{\dr}(V)/F^0  )=g$, to prove the lemma it will be enough to prove $\dim \Sel (U)\leq r+\dim H^1 _f (G_{K,T},W)$. By Lemma \ref{old_lemma}, it is enough to prove that, for all $i$,
\[
\dim H^1 _f (G_{K,T},U^{(\alpha _i )})' \leq r+\dim H^1 _f (G_{K,T},W).
\]
Since the action of $U$ on itself by conjugation is unipotent, we have a Galois-equivariant short exact sequence
\[
1\to W \to U^{(\alpha _i )}\to V \to 1,
\]
inducing an exact sequence of pointed varieties
\[
H^1 _f (G_{K,T},W)\to \Sel (U^{(\alpha _i )}) \to J(K)\otimes \Q _p .
\]
Hence $\dim \Sel (U^{(\alpha _i )})\leq \dim H^1 _f (G_{K,T},W)+\dim ( J(K)\otimes \Q _p )$, as required.
\end{proof}

\subsection{The context of the present work}
We will always take $U$ to be an intermediate quotient
$$
U_2 \to U \stackrel{\pi }{\longrightarrow} V.
$$
The group $U_2 $ is an extension of $V$ by 
$$
\overline{\wedge ^2 V}:=\coker (\Q _p (1) \stackrel{\cup ^* }{\longrightarrow }\wedge ^2 V ),
$$
hence such quotients are in correspondence with Galois-stable summands of $\wedge ^2 V /\mathbb{Q}_p (1)$.
This paper is concerned with the commutative diagram
$$
\begin{tikzpicture}
\matrix (m) [matrix of math nodes, row sep=3em,
column sep=3em, text height=1.5ex, text depth=0.25ex]
{X(K) & \Sel (U)  \\
 X(K_{\mathfrak{p}}) & H^1 _f (G_{\mathfrak{p}} ,U) & D_{\dr}(U)/F^0 \\ };
\path[->]
(m-1-1) edge[auto] node[auto] {$j$} (m-1-2)
edge[auto] node[auto] {} (m-2-1)
(m-1-2) edge[auto] node[auto] {$\loc _{\mathfrak{p}} $} (m-2-2)
edge[auto] node[auto] {$j_{\mathfrak{p}}$} (m-2-3)
(m-2-1) edge[auto] node[auto] { } (m-2-2)
(m-2-2) edge[auto] node[auto] {$\simeq $} (m-2-3);
\end{tikzpicture} $$
and in particular with identifying situations under which $\loc _p $ is not dominant and describing what $X(K_{\mathfrak{p}} )_U$ looks like in this case.

\subsection{Provable finiteness via the geometric N\'eron--Severi group}

One piece of the weight $-2$ representation $[U_2 ,U_2 ]$ whose Selmer group we can understand unconditionally is the Artin--Tate part, equivalently the part coming from the geometric N\'eron--Severi group of $J$. In this subsection we restrict to the case $K=\mathbb{Q}$.
\begin{Lemma}\label{artin-tate-bound}
For any representation of $G_{\Q ,T}$ on a finite-dimensional vector space $V$ over a field $F\subset \mathbb{Q}_p $, which factors through a finite quotient $\Gal (L|\mathbb{Q})$ of $\Gal (\mathbb{Q})$, where $L|\mathbb{Q} $ is unramified at $p$, we have an isomorphism
$$
H^1 _f (G_{\mathbb{Q},T} ,V\otimes _F \mathbb{Q}_p (1) )\simeq (V\otimes \mathbb{Q}_p )^{c=1}/ (V\otimes \mathbb{Q}_p )^{\Gal (\mathbb{Q})},
$$
where  $c\in \Gal (\mathbb{Q})$ denotes complex conjugation.
\end{Lemma}
\begin{proof}
The crucial point is that, since $H^0 (G_{L,T},V\otimes \Q _p (1))=0$, the inflation-restriction exact sequence induces an isomorphism
\[
H^1 (G_{\mathbb{Q},T} ,V\otimes _F \mathbb{Q}_p (1))\simeq H^1 (G_{L,T},V\otimes _F \mathbb{Q}_p (1))^{\Gal (L|\mathbb{Q})},
\]
and similarly we have isomorphisms
\[
H^1 (G_p ,V\otimes _F \Q _p (1))\simeq \oplus _{v|p}H^1 (G_v ,V\otimes _F \Q _p (1))^{\Gal (L|\Q )},
\]
which induce isomorphisms
\[
H^1 _f (G_p ,V\otimes _F \Q _p (1))\simeq \oplus _{v|p}H^1 _f (G_v ,V\otimes F\Q _p (1))^{\Gal (L|\Q )}.
\]
This induces an isomorphism
$$
H^1 _f (G_{\mathbb{Q},T} ,V\otimes _F \mathbb{Q}_p (1))\simeq H^1 _f (G_{L,T},V\otimes _F \mathbb{Q}_p (1))^{\Gal (L|\mathbb{Q})}.
$$
Given this, we observe
\begin{align*}
H^1 _f (G_{L,T},V\otimes _F \mathbb{Q}_p (1))^{\Gal (L|\mathbb{Q})} \simeq & (H^1 _f (G_{L,T}, \mathbb{Q}_p (1))\otimes _F V)^{\Gal (L|\mathbb{Q})} \\
\simeq & ((\mathcal{O}_L ^\times \otimes _{\mathbb{Z}}\mathbb{Q}_p )\otimes _F V)^{\Gal (L|\mathbb{Q})}.
\end{align*}
Now we use the description of $\mathcal{O}_L ^\times \otimes \mathbb{Q}_p $ as a Galois module \cite[$\S$8.7.2]{neukirch:2000}:
$$
\mathcal{O}_L ^\times \otimes _{\mathbb{Z}}\mathbb{Q}_p \simeq \Ind ^{\Gal (L|\mathbb{Q})}_{\langle c \rangle }\mathbb{Q}_p /\mathbb{Q}_p,
$$and finally, we have
\begin{align*}
((\mathcal{O}_L ^\times \otimes _{\mathbb{Z}}\mathbb{Q}_p )\otimes _F V)^{\Gal (L|\mathbb{Q})} 
\simeq & \Hom _{\Gal (L|\mathbb{Q})}((\mathcal{O}_L ^\times \otimes _{\mathbb{Z}}\mathbb{Q}_p ),V\otimes _F \mathbb{Q}_p ) \\
\simeq & \Ker(\Hom _{\Gal (L|\mathbb{Q})}(\Ind ^{\Gal (L|\mathbb{Q})}_{\langle c \rangle } \mathbb{Q}_p ),V\otimes _F \mathbb{Q}_p )\to \Hom _{\Gal (L|\mathbb{Q})}(\mathbb{Q}_p ),V\otimes _F \mathbb{Q}_p ))\\
\simeq & \Ker (\Hom _{\langle c \rangle }(\mathbb{Q}_p ,V\otimes _F \mathbb{Q}_p )\to \Hom _{\Gal (L|\mathbb{Q})}(\mathbb{Q}_p ,V\otimes \mathbb{Q}_p )) \\
\simeq & (V\otimes \mathbb{Q}_p )^{c=1}/ (V\otimes \mathbb{Q}_p )^{\Gal (\mathbb{Q})}.
\end{align*}\end{proof}
We deduce the following corollary:
\begin{Proposition}
Let $K=\mathbb{Q}$, and define $\rho _f (J)=\rk \NS(J)+\rk (\NS (J_{\overline{\mathbb{Q}}})^{c=-1})$ as in the introduction.
If $$
\rk J <g-1+\rho _f (J),$$ 
then $X(\mathbb{Q}_p )_2 $ is finite.
\end{Proposition}
\begin{proof}
Let $$
W:=\varinjlim _L \bigg[ \mathbb{Q}_p (1) \otimes \Hom _{G_L }(\mathbb{Q}_p (1),\overline{\wedge ^2 V}) \bigg] \subset \overline{\wedge ^2 V}
$$ be the Artin--Tate part of $[U_2 ,U_2 ]$. Then we know that $W$ contains (and is equal to by Faltings) the Artin--Tate representation $(\NS(J_{\overline{\mathbb{Q}}})\otimes \mathbb{Q}_p )/\mathbb{Q}_p (1)$. The proof that $X(\mathbb{Q}_p )_U$ is finite is as in \cite[Lemma 2.6]{balakrishnan:2016}, with the only difference being a more general choice of $W$. To recall, we use the fact that it is enough to prove that $$\dim \Sel (U)<\dim H^1 _f (G_p ,U).$$ 
It is enough to prove that $\dim \Sel (U)_{\alpha }<\dim H^1 _f (G_p ,U)$ for any collection of local conditions. By Lemma \ref{old_lemma}, we have
\[
\dim \Sel (U)_{\alpha }\leq \dim H^1 _f (G_{K,T},W)+\dim H^1 _f (G_{K,T}(V).
\]
At $p$, we claim the sequence 
$$
1 \to H^1 _f (G_p ,W)\to H^1 _f (G_p ,U) \to H^1 _f (G_p ,V) \to 1 ,
$$
is exact. One way to see this is that the non-abelian Dieudonn\'e functor induces an isomorphism of schemes
\[
H^1 _f (G_p ,U) \simeq D_{\dr}(U)/F^0 .
\]
In \cite[\S 1]{kim:2009}, Kim proves that this map is algebraic. The map is given by sending a torsor $P$ to a $D_{\cry }(U)$-torsor object $D_{\cry }(P)$ in the category of filtered $\phi $-modules, and by proving that the set of isomorphism classes of such torsors is represented by $D_{\dr}(U)/F^0 $. Although it is not explicitly stated in loc. cit. that this map is bijective, one can deduce it from the fact that the map has an inverse given by sending a $D_{\cry}(U)$-torsor $P$ to the crystalline $U$-torsor 
$\spec (F^0 (\mathcal{O}(P)\otimes B_{\cry })^{\phi =1})$.
Hence exactness follows from exactness of 
\[
1\to D_{\dr} (W)/F^0 \to D_{\dr}(U)/F^0 \to D_{\dr} (V) /F^0 \to 1.
\]
We deduce that $X(\mathbb{Q}_p )_2 $ is finite whenever 
$$
\dim H^1 _f (G_{\mathbb{Q},T} ,W)+\rk J < \dim H^1 _f (G_p ,W) +g.
$$
The proposition now follows from Lemma \ref{artin-tate-bound}, since this implies
$$
\dim H^1 _f (G_p ,W)-\dim H^1 _f (G_{\mathbb{Q},T} ,W)=\dim \NS(J)+\dim \NS(J_{\overline{\mathbb{Q}}})^{c=-1}-1.
$$
\end{proof}
\subsection{Finiteness assuming the Bloch--Kato conjectures}
Here we describe situations when finiteness of $X(\mathbb{Q}_p )_2 $ is implied by the Bloch--Kato conjectures. The Bloch--Kato conjectures relate the dimension of $H^1 _f (G_{K,T},W)$ to the rank of certain graded pieces of $K$-groups of algebraic varieties. 
Let $Z$ be a smooth projective variety over $\mathbb{Q}$. For $i\in \mathbb{Z}$, let $K_i (Z)$ denote the $i$th algebraic $K$-group of $Z$ in the sense of Quillen. The only fact we will use about $K_i (Z)$ is that it is zero when $i<0$, and the action of Adams operators enables one to define a grading $K_i (Z)\otimes \mathbb{Q}=\oplus _{j\in \mathbb{Z}}K_i ^{(j)}(Z)$ on the group tensored with $\mathbb{Q}$. The following is a special case of their conjectures.
\begin{Conjecture}[Bloch--Kato {\cite[Conjecture 5.3 (i)]{Bloch-Kato:1990}}]\label{BKconj}
Let $Z$ be a smooth projective variety over $\mathbb{Q}$. Then for any $n>0$ and $2r-1\neq n$, the map
$$
\ch _{n,r}:K_{2r-1-n}^{(r)}(Z)\otimes \mathbb{Q}_p \to H^1 _f (G_{\mathbb{Q}},H^n (Z_{\overline{\mathbb{Q}}},\mathbb{Q}_p (r))
$$
is an isomorphism.
\end{Conjecture}
Kim \cite{kim:2009} showed that this conjecture implies that $X(\mathbb{Q}_p )_n $ is finite for all $n$ sufficiently large, with no hypotheses on the rank of $J$. As we are interested in $X(\mathbb{Q}_p )_2 $, we now work out the exact conditions on $X$ for which Kim's argument can be used to show that Conjecture \ref{BKconj} implies finiteness of $X(\mathbb{Q}_p )_2 $.
\begin{Lemma}\label{BKimplies}
Conjecture \ref{BKconj} implies $H^1 _f (G_{\mathbb{Q},T} ,\overline{\wedge ^2 V}^*(1))=0$.
\end{Lemma}
\begin{proof}
As $\overline{\wedge ^2 V}^* (1)$ is a direct summand of $H^2 _{\acute{e}t}(X\times X_{\overline{\mathbb{Q}}},\mathbb{Q}_p (1))$,  it suffices to prove that 
$$
H^1 _f (G_{\mathbb{Q},T} ,H^2 _{\acute{e}t}(X\times X_{\overline{\mathbb{Q}}},\mathbb{Q}_p (1) ))=0.
$$
This follows from Conjecture \ref{BKconj}, since that implies
\[
\dim H^1 _f (G_{\mathbb{Q},T} ,H^2 _{\acute{e}t}(X\times X_{\overline{\mathbb{Q}}},\mathbb{Q}_p (1) ))
\leq \dim H^1 _g (G_{\mathbb{Q},T} ,H^2 _{\acute{e}t}(X\times X_{\overline{\mathbb{Q}}},\mathbb{Q}_p (1) )) 
\leq \dim K_{-1}(X\times X)\otimes \mathbb{Q}=0.
\]
\end{proof}
\begin{Lemma}
Conjecture \ref{BKconj} implies 
$$\dim H^1 _f (G_p ,\overline{\wedge ^2 V})-\dim H^1 _f (G_{\mathbb{Q},T} ,\overline{\wedge ^2 V})\geq g(g-1).$$
\end{Lemma}
\begin{proof}
Recall the following corollary of Poitou--Tate duality  \cite[Remark 2.2.2]{fontaine:1994}:
\begin{align*}
& \dim _{\mathbb{Q}_p }(H^0 (G_{\mathbb{Q},T},W))-\dim _{\mathbb{Q}_p }(H^1 _f (G_{\mathbb{Q},T} ,W)) +\dim _{\mathbb{Q}_p }(H^1 _f (G_{\mathbb{Q},T},W^* (1))) \\
& -\dim _{\mathbb{Q}_p }(H^0 (G_{\mathbb{Q},T} ,W^* (1)))= -\dim _{\mathbb{Q}_p }(D_{\dr }(W)/F^0 )+\dim _{\mathbb{Q}_p }(H^0 (G_{\mathbb{R}},W)).
\end{align*} 
In this case of $W=\overline{\wedge ^2 V}$, we have 
\begin{align*}
& \dim _{\mathbb{Q}_p }(H^0 (G_{\mathbb{Q},T} ,W^* (1)))=\rho (J_X )-1, \;\qquad D_{\dr }(W)/F^0 =H^1 _f (G_p ,W), \\
& \dim _{\mathbb{Q}_p }(H^0 (G_{\mathbb{R}},W))=g(g-1), \qquad\qquad\qquad H^0 (G_{\mathbb{Q},T} ,W)=0,
\end{align*}
hence the claim follows from Lemma \ref{BKimplies}. 
\end{proof}
Hence we deduce the following simple criterion for conjectural finiteness of $X(\mathbb{Q}_p )_2 $.
\begin{Lemma}\label{BKimpliesfinite}
Suppose Conjecture \ref{BKconj}. Let $X/\Q$ be a curve of genus $g \geq 2$. If $r= \rk J(\mathbb{Q})<g^2  $, then $X(\mathbb{Q}_p )_2 $ is finite. \\
\end{Lemma}
\begin{proof}
By the previous lemma, we have 
\begin{align*}
\dim H^1 _f (G_{\mathfrak{p}},U_2 )-\dim \Sel (U_2 ) & \geq \dim H^1 _f (G_{\mathfrak{p}},W)-\dim H^1 _f (G_{K,T},W)+g-\rk J(K) \\
& >g(g-1) +g -g^2 =0.
\end{align*}
\end{proof}

\section{Generalised height functions}\label{sec:ht}

We now return to considering a general $K|\mathbb{Q}$, with $p$ a prime splitting completely in $K$ and $\mathfrak{p}$ a prime of $K$ lying above $p$.
In \cite{balakrishnan:2016}, we used Nekov\'a\v r's formalism of $p$-adic height functions on mixed extensions to describe Chabauty--Kim sets in terms of $p$-adic height pairings of cycles on $X$. Given a choice of global character $\chi \in H^1 (G_{K,T} ,\mathbb{Q}_p )$, Nekov\'a\v r's $p$-adic height functions associate to certain filtered Galois representations with graded pieces $\mathbb{Q}_p $, $V$, and $\mathbb{Q}_p (1)$, a collection of local cohomology classes with values in $\mathbb{Q}_p (1)$. We obtain a $\mathbb{Q}_p$-valued function by summing the cup products of these local classes with $\chi $.

In this section, we describe a natural generalisation of Nekov\'a\v r's formulation of the $p$-adic height pairing, resulting in a notion of generalised $p$-adic height functions. To do this, we essentially mimic his construction at every step, occasionally rephrasing some constructions in terms of nonabelian cohomology. 
\subsection{Twisting the enveloping algebra}\label{subsec:twist}
We recall some core ideas from \cite{balakrishnan:2016} about mixed extensions and nonabelian cohomology. In what follows, $X$ is a smooth projective curve over $K$, $V=H^1 _{\acute{e}t}(\overline{X},\mathbb{Q}_p )^* $ and $W$ is a quotient of $\overline{\wedge ^2 V}$. Before describing generalities on mixed extensions, we first give the main examples of the filtered Galois representations in this paper.

Let $\mathbb{Z}_p [\! [\pi _1 ^{\acute{e}t}(\overline{X},b)]\! ]:=\varprojlim \mathbb{Z}_p [\pi _1 ^{\acute{e}t}(\overline{X},b)/N ]$, where the limit is over normal subgroups $N$ such that $\pi _1 ^{\acute{e}t}(\overline{X},b)/N$ is a finite $p$-group. Let 
$I$ denote the kernel of the natural map
\[
\Q _p \otimes \mathbb{Z}_p [\! [\pi _1 ^{\acute{e}t}(\overline{X},b)]\! ] \to \Q _p .
\]
Then define $A_2 (b):=\Q _p\otimes \mathbb{Z}_p [\! [\pi _1 ^{\acute{e}t}(\overline{X},b)]\! ]/I^3 $ (see \cite[\S 2]{kim:2010}).

The Galois representation $\mathbb{Q}_p [\pi _1 (\overline{X};b,z)]$ has the structure of an equivariant \\
$(\mathbb{Q}_p [\pi _1 (\overline{X},b)],\mathbb{Q}_p [\pi _1 (\overline{X},z)])$-bimodule, allowing one to define a finite-dimensional $(A_2 (b),A_2 (z))$-bimodule
\begin{align*}
A_2 (b,z) & :=\mathbb{Q}_p [\pi _1 ^{\acute{e}t}(\overline{X};b,z)]\otimes _{\mathbb{Q}_p [\pi _1 ^{\acute{e}t}(\overline{X},b)]}A_2 (b) \\
& \simeq A_2 (z) \otimes _{\mathbb{Q}_p [\pi _1 ^{\acute{e}t}(\overline{X},z)]}\mathbb{Q}_p [\pi _1 ^{\acute{e}t}(\overline{X};b,z)].
\end{align*}

As in \cite{balakrishnan:2016}, we define $A_2 (b)\to A(b)$ to be the quotient of $A_2 $ by the kernel of the composite 
$$
I^2 /I^3 \simeq \overline{\wedge ^2 V}\oplus \Sym ^2 V \to \overline{\wedge ^2 V} \to W. 
$$
$A(b)$ is then an algebra with a faithful left action of $U(b)\subset A(b)^\times $. 
Given a $U$-torsor $P$, the induced twist of $A$ by $P$, denoted $A(b)^{(P)}$, is an element of $\mathcal{M} (G_{K,T} ;\mathbb{Q}_p ,V,W)$. If $P$ is crystalline above $p$ and unramified outside $T$, then $A(b)^{(P)}$ will also have these properties, inducing a morphism of varieties
$$
H^1 _{f,T_0 } (G_{K,T} ,U)\to M_{f,T_0 } (G_{K,T} ;\mathbb{Q}_p ,V,W).
$$
The multiplication map $\wedge :IA(b)\times IA(b)\to I^2 A(b)$ factors through $V\times V$, which enables us to make the following definition.
\begin{Definition}\label{defn:tau} 
$\tau :V\to \Hom (V,W)$ denote the map $v_1 \mapsto ( v_2 \mapsto v_1 \wedge v_2 )$. 
\end{Definition}
\begin{Lemma}\label{symmbit}
$A^{(P)}$ is a mixed extension of $\pi _* P$ by $[IA(b)]+\tau _* \pi _* P$.
\end{Lemma}
In the case when $P$ is the $U$-torsor of paths from $b$ to $z$, we shall denote the corresponding element of $\mathcal{M}_{f,T_0 } (G_{K,T} ;\mathbb{Q}_p ,V,W)$ by $A(b,z)$.
We obtain a map 
$$
H^1 _{f,T_0 } (G_{K,T} ,U)\to M_{f,T_0 } (G_{K,T} ;\mathbb{Q}_p ,V,W).
$$
We define $A(b,z):=A(b)^{(P(b,z))}$; we have an isomorphism of mixed extensions
$$
A(b)^{(P(b,z))}=^{(P(z,b))}A(z).
$$
\begin{Lemma}\label{leftrighttwist}
For any $x_1 ,x_2 ,z_1 ,z_2$ in $X(K)$,
\begin{equation}
[IA(x_1 ,x_2 )]=[IA(z_1 ,z_2 )]+\tau _* (\kappa (x_1 +x_2 -z_1 -z_2 )) \in \ext ^1 (V,W).
\end{equation}
\end{Lemma}
\begin{proof}
First suppose $x_1 =x_2$. Then
$
[IA(x_1 ,z_1 )]=[IA(x_1 ,z_2 )^{\kappa (z_1 -z_2 )}].
$
By definition of the twisting construction,
$
[IA(x_1 ,z_2 )^{\kappa (z_1 -z_2 )}]=[IA(x_1 ,z_2 )]+\tau _* \kappa (z_1 -z_2 ).
$
Similarly,
$
[IA(x_1 ,z_1 )]=[IA(x_2 ,z_1 )^{\kappa (x_1 -x_2 )}]=[IA(x_2 ,z_1 )]+\tau _* \kappa (x_1 -x_2 ).
$
\end{proof}

Note that, in general, the extension class of $IA(b)$ in $H^1 (G,V^* \otimes W)$ will not lie in the image of $\tau _* $. More specifically, we know that its class in  $H^1 (G,V^* \otimes W)/H^1 (G,V)$ is related to 
the Abel--Jacobi class of the Gross--Kudla--Schoen cycle in $X\times X \times X$ corresponding to $b$ by \cite[Theorem 1]{darmon:2012}, which is generically nontrivial.
One situation where the class of $IA(b)$ does lie in the image of $\tau _*$, and furthermore 
can be described explicitly in terms of $b$, is when $X$ is hyperelliptic (the argument, given below, is a straightforward generalisation of Lemma 1.1 of \cite{kimrank1:2010}). This may be viewed as a special case of a slightly more general phenomenon where one reduces computations on $IA(b,z)$ to the case where $b=z$ is a Weierstrass point, at which point the computation becomes trivial. We refer to this as a \emph{hyperelliptic splitting principle}.  
\begin{Lemma}\label{hyperellipticsplitting}
Let $X$ be a hyperelliptic curve of genus $g$, with equation $y^2 =f(x),$
for $f$ a degree $2g+2$ polynomial.
Let $\alpha _1 ,\ldots ,\alpha _{2g+2}$ be the roots of $f$. Let $Z$ denote the $\mathbb{Q}$-divisor $\frac{1}{g+1}\sum _i (\alpha _i ,0)$. 
Then 
$$[IA(b,z)]=\tau _* (\kappa (b+z-Z)).$$ 
\end{Lemma}
\begin{proof}
First note that it will be enough to prove that the two classes are equal in $H^1 (G_{L,T},V^* \otimes W)$, 
for $L$ some finite extension of $K$, since the restriction map is injective. Let 
$L$ be an extension containing all roots of $f$. For any $i, j$, the divisor $(\alpha _i ,0)-(\alpha _j ,0)$ is torsion, and so in particular 
$$\kappa _L ((\alpha _i ,0) -(\alpha _j ,0) )=0.$$
Hence it is enough to show that the $H^1 (G_{L,T},V^* \otimes W)$ class obtained from $A_2 (b,z )$ agrees with that of $z+b -2(\alpha _i ,0)$ for some $i$. 
We prove this in three stages:\\
(i) Suppose $z=b =(\alpha _i ,0)$. Then the hyperelliptic involution gives an action of $\mathbb{Z}/2\mathbb{Z}$ on $A_2 (b )$. This acts on the $V$-graded piece as $-1$ and 
on the $W$-graded piece as the identity. Hence we obtain a splitting of $A_2 (b)$. \\
(ii) Now let $b $ be arbitrary. By Lemma \ref{leftrighttwist}, $IA_2 (b )$ is just the twist of $IA_2(\alpha _i )$ by the $H^1 (G_{K,T} ,V)$ class of 
$b -(\alpha _i ,0)$. Since this twist is via the conjugation action of $U$ on $A$, the corresponding extension class is $2\kappa (b -(\alpha _i ,0))$. \\
(iii) Now we consider the general case. Consider the right action of $V$ on $IA_2 (b )$. The representation $IA_2 (b ,z)$ is simply obtained by twisting 
by the $H^1 (G_{K,T} ,V)$-torsor associated to $z-b $. Hence the class is $z+b -2(\alpha _i ,0)$.
\end{proof}

\subsubsection{Recap of local Galois cohomology}
Let $v$ be a prime of $K$ not lying above $p$.
Let $I_v \subset G_v $ be the inertia subgroup and $F_v \in G_v /I_v $ a generator.
For any finite-dimensional $\mathbb{Q}_p$-representation of $W$, let $H^1 _f (G_v ,W):=W^{I_v }/(F_v -1)W^{I_v }$. Then for any such $W$, by Tate duality, there is a short exact sequence 
$$
0\to H^1 _f (G_v ,W)\to H^1 (G_v ,W)\to H^1 _f (G_v ,W^* (1))^* \to 0
$$
(see e.g., \cite[Lemma 1 and Theorem 1]{washington:1997}).
\begin{Lemma}\label{locallemma}
Let $V=H^1 (\overline{X},\mathbb{Q}_p )^*$, let $n\geq 0$, and let $W$ be a direct summand of $V^{\otimes (2n+1)}(-n)$. Then $H^1 (G_v ,W)=0$.
\end{Lemma}
\begin{proof}
As $W$ is a direct summand, it is enough to prove this for $W=V^{\otimes (2n+1)}(-n)$. Since this representation is its own Tate dual, it is enough to prove that $H^1 _f (G_v ,W)=0$, or equivalently 
$W^{G_v }=0$. This follows directly from the weight-monodromy conjecture for curves \cite{grothendieck:1972}: let $L$ be a finite extension of $\mathbb{Q}_v $ such that $I_L$ acts unipotently on $V$ (and hence $W$).
If $V[i]$ and $W[i]$ denote the graded pieces of $V$ and $W$ of weight $i$, resp., then weight-monodromy implies that we have an equality $(1-I_L )V[0]=V[-2]$, (and we know it is trivial on $V[-1]$), hence $(1-I_L )V^{\otimes (2n+1)}[2n]=V^{\otimes (2n+1)}[2n-2]$. Hence the kernel of $(1-I_L)$ on the weight zero part of $W$ is trivial, so $H^1 _f (G_v ,W)=H^1 (G_v ,W)=0$.
\end{proof}

\begin{Lemma}
The extension $[IA_2 (b)]$ in $\ext ^1 (V,W)$ is crystalline at all primes above $p$ and splits at all $v$ away from $p$.
\end{Lemma}
\begin{proof}
For $v \in T_0$, this follows from Lemma \ref{locallemma}.
For $v|p$, it follows from Olsson \cite{olsson:2011}. As the statement there is slightly different, we explain how to deduce the Lemma from it. In fact, we will explain how to deduce the more general result that $A_n (b,z)$ is crystalline for all $z$ in $X(\Q _p )$ and all $n\geq 1$. Let $\mathcal{O}(\pi _1 ^{\acute{e}t,\Q _p }(\overline{X};b,z))$ denote the co-ordinate ring of the $\Q _p $-unipotent \'etale torsor of paths from $b$ to $z$. By \cite[Theorem 1.11]{olsson:2011}, this is an ind-crystalline representation, and moreover
\[
D_{\cry }(\mathcal{O}(\pi _1 ^{\acute{e}t,\Q _p }(\overline{X};b,z)))\simeq \mathcal{O}(\pi _1 ^{\dr}(X;b,z)).
\]
To prove that $A_n (b,z)$ is crystalline it is enough to prove that $\varinjlim A_n (b,z)^* $ is ind-crystalline. This follows from Olsson's theorem via the Galois equivariant isomorphism 
\[
\varinjlim A_n (b,z)^* \simeq \mathcal{O}(\pi _1 ^{\acute{e}t,\Q _p }(\overline{X};b,z)),
\]
(see for example Hadian \cite[2.12]{hadian:2011} or Kim \cite[\S 2]{kim:2009}). 
\end{proof}
In the notation of \S \ref{sec:hodgefiltration}, we deduce furthermore that
\begin{equation}\label{eqn:olsn}
D_{\cry }(A_n (b,z))=A_n ^{\dr }(b,z).
\end{equation}
\subsection{Mixed extensions}

Following Nekov\'a\v r, we construct generalised height functions as functions on equivalence classes of mixed Galois representations with fixed graded pieces. The examples to bear in mind are the mixed extensions $A(b,z)$ constructed above.
\begin{Definition}
Define $\mathcal{M}_{f,T_0 } (G_{K,T} ;\mathbb{Q}_p ,V,W)$ to be the category whose objects are tuples $(M,(M_ i)_{i=0,1,2,3},(\psi _i  )_{i=0,1,2})$ where $M$ is a $G_{K,T}$ representation which is crystalline at all primes above $p$, $(M_i )$ is a Galois-stable filtration 
$$
M=M_0 \supset M_1 \supset M_2  \supset M_3 =0
$$
and the $\psi _i $ are isomorphisms 
$$\psi _0 :\mathbb{Q}_p \to M_0 /M_1, \qquad \psi _1 :V \to M_1 /M_2, \qquad \psi _2 :W \to M_2 /M_3$$
and whose morphisms are isomorphisms of Galois representations respecting the filtration and commuting with the $\psi _i $.
Define $M_{f,T_0 } (G_{K,T} ;\mathbb{Q}_p ,V,W)$ \index{$M_{f,T_0 } (G_{K,T} ;\mathbb{Q}_p ,V,W)$} to be  the set $\pi _0 ( \mathcal{M}_{f,T_0 } (G_{K,T} ;\mathbb{Q}_p ,V,W))$ of isomorphism classes of mixed extensions.
Similarly define $M(G_v ;\mathbb{Q}_p ,V,W)$ (resp. $M_f (G_v ;\mathbb{Q}_p , V,W)$ for $v$ above $p$) to be the set of isomorphism classes of corresponding categories of $G_v $ representations (resp. crystalline representations).
\end{Definition}
Given a mixed extension $M$ in $\mathcal{M}_{f,T_0 }(G_{K,T};\Q _p ,V,W)$, we obtain extensions $M/M_2 $ and $\Ker(M\to M_0 )$ of $\Q _p $ by  $V$ and of $V$ by $W$ respectively. By Lemma \ref{locallemma} these extensions are automatically unramified at all primes of $T_0 $, and hence lie in $H^1 _f (G_{K,T},V)$ and $\ext ^1 _f (V,W)$.
We denote by $\pi _{1*}$ and $\pi _{2*}$ the natural maps 
$$
\pi _{1*}: M_{f,T_0 } (G_{K,T} ;\mathbb{Q}_p ,V,W)\to H^1 _{f} (G_{K,T} ,V), \qquad \pi _{2*}: M_{f,T_0} (G_{K,T} ;\mathbb{Q}_p ,V,W)\to \ext ^1 _f (V,W).
$$
Define $U(\mathbb{Q}_p ,V,W)$ to be the unipotent group of vector space isomorphisms of unipotent isomorphisms of $\mathbb{Q}_p \oplus V  \oplus W$, i.e., those which respect the filtration
$$
\mathbb{Q}_p \oplus V \oplus W \supset V \oplus  W \supset W
$$
and are the identity on the associated graded. Recall from \cite[Lemma 4.7]{balakrishnan:2016} that we have an isomorphism
$$
M_{f,T_0 } (G_{K,T} ;\mathbb{Q}_p ,V,W)\simeq H^1 _{f,T_0 } (G_{K,T} ,U(\mathbb{Q}_p ,V,W)).
$$
The maps $\pi _{1*}$ and $\pi _{2*}$ are induced from the exact sequence
\begin{equation}\label{induced_exact}
0\to W \to U(\mathbb{Q}_p ,V,W)\stackrel{(\pi _1 ,\pi _2 )}{\longrightarrow } V\oplus V^* \otimes W \to 0.
\end{equation}

We now outline in broad strokes our generalisation of Nekov\'a\v r's formulation of the $p$-adic height pairing. Although we could work in somewhat greater generality, we restrict attention to our case of interest. We take as input a tuple $(V,W,j,s, \chi )$, where
$V$ and $W$ are as before. Let $s:D_{\dr }(V)\to F^0 D_{\dr }(V)$ be a splitting of the Hodge filtration. Finally $\chi $ is a non-crystalline element of $H^1 (G_{K,T},W^* (1))$.

Associated to this data, we will define, for each $v$ prime to $p$, a local \textit{pre-height} function
$$
\widetilde{h}_v :M(G_v ;\mathbb{Q}_p ,V,W)\to H^1 (G_v ,W),
$$
\index{$\widetilde{h}_v$}
as well as a local pre-height at primes $\mathfrak{p}$
$$
\widetilde{h}_{\mathfrak{p}} :M_f (G_{\mathfrak{p}} ;\mathbb{Q}_p ,V,W)\to H^1 _f (G_{\mathfrak{p}} ,W).
$$
Using $\chi$, we then define a global height 
$$
h :M_{f,T_0 } (G_{K,T} ;\mathbb{Q}_p ,V,W)\to \mathbb{Q}_p ,
$$
such that $h(M)$ only depends on the image of $M$ under
$$
M_{f,T_0 } (G_{K,T} ;\mathbb{Q}_p ,V,W)\to H^1 _f (G_{K,T} ,V)\otimes H^1 _f (G_{K,T} ,V^* \otimes W).
$$
As in the classical set-up, the global height will be a sum of local heights $h_v$ which are compositions of the map $\widetilde{h}_v $ with the character $\loc _v \chi $, thought of as an element of $H^1 (G_v ,W)^* $ via Tate duality. For the applications in this paper, we will only be interested in characters $\chi $ for which $h_{\mathfrak{q}}$ is trivial at all $\mathfrak{q}$ above $p$ except 
$\mathfrak{p}$.
\subsection{Definition of the local pre-height}\label{localheightdefn}
When $v$ is not in $T$, $\widetilde{h}_v$ is trivial. When $v$ is in $T_0 $, then by Lemma \ref{locallemma} we have $$H^0 (G_v ,V) =H^0 (G_v ,V^* \otimes W)=H^1 (G_v ,V)=H^1 (G_v ,V^* \otimes W)=0.$$ Hence via the exact sequence \eqref{induced_exact}, we get an isomorphism  
$$
M(G_v ;\mathbb{Q}_p ,V,W)\stackrel{\simeq }{\longrightarrow }H^1 (G_v ,W).
$$
We defined $\widetilde{h}_v $ to be this isomorphism and define $h_v $ as the composite 
$$
M(G_v ;\mathbb{Q}_p ,V,W) \stackrel{\widetilde{h}_v \cup \chi }{\longrightarrow} H^2 (G_v ,\mathbb{Q}_p (1))\simeq \mathbb{Q}_p .
$$
Finally for $v|p$, $\widetilde{h}_v $ and $h_v$ are defined following \cite[\S\S 3-4]{Nek:1993}. As we restrict to global heights for which the only contribution from primes $v|p$ is at $\mathfrak{p}$, we will only describe $h_{\mathfrak{p}}$, but the description carries over verbatim to other primes above $p$.

The local height above $p$ is described using Fontaine's functor $D_{\cry}$, which gives an equivalence of categories between $\mathcal{M}_f (G_{\mathfrak{p}};\Q _p ,V,W)$ and the category $\mathcal{M}_{\fil ,\phi }(\Q _p,D_{\cry }(V),D_{\cry }(W))$ of mixed extensions of filtered $\phi $-modules with graded pieces $\Q _p $,$D_{\cry }(V)$ ,$D_{\cry }(W)$. Similarly this induces a bijection between sets of isomorphism classes 
\[
M_f (G_{\mathfrak{p}};\Q _p ,V,W) \simeq M_{\fil ,\phi }(\Q _p ,D_{\cry }(V),D_{\cry }(W)).
\]
To ease notation we henceforth write $D_{\dr }(V)$ and $D_{\dr}(W)$ as $V_{\dr }$ and $W_{\dr }$ respectively. As $K_{\mathfrak{p}}$ is an unramified extension of $\Q _p $, and $V$ and $W$ are crystalline, we also identify these with $D_{\cry }(V)$ and $D_{\cry }(W)$.

We identify $M_{\fil ,\phi }(\Q _p ,V_{\dr},W_{\dr })$ with $F^0 \backslash U(\Q _p ,V_{\dr} ,W_{\dr })$ as follows. Given a mixed extension $M$, let $s^\phi ,s^H :\Q _p \oplus V_{\dr} \oplus W_{\dr }\stackrel{\simeq }{\longrightarrow }M$ be unipotent isomorphisms of filtered vector spaces which respect the Frobenius structure and Hodge filtration respectively. Then $(s^H )^{-1}\circ s^\phi $ defines an element  of $U(\Q _p ,V_{\dr} ,W_{\dr })$. The element $s^\phi $ is uniquely determined, and any different choice of the other isomorphism is of the form $s^H \circ u^H$, for some $u^H \in F^0 U(\Q _p ,V_{\dr} ,W_{\dr })$. This gives a bijective correspondence 
\[
M_{\fil ,\phi }(\Q _p ,V_{\dr},W_{\dr }) \to F^0 \backslash U(\Q _p ,V_{\dr} ,W_{\dr }),
\]
which is furthermore an isomorphism of algebraic varieties.

We first define a section $t$ of 
$$
M_f (G_{\mathfrak{p}};\mathbb{Q}_p ,V,W)\to H^1 _f (G_{\mathfrak{p}},V)\times H^1 _f (G_{\mathfrak{p}},V^* \otimes W)
$$
as follows: given exact sequences of crystalline $G_{\mathfrak{p}}$-representations
\begin{align*}
& 0 \to V\to E_1 \to \Q _p \to 0 \\
& 0 \to W \to E_2 \to V\to 0 ,\\
\end{align*}
we have a commutative diagram with exact rows
$$
\begin{tikzpicture}
\matrix (m) [matrix of math nodes, row sep=3em,
column sep=3em, text height=1.5ex, text depth=0.25ex]
{  0 & H^1 _f (G_{\mathfrak{p}},W) & H^1 _f (G_{\mathfrak{p}  },E_2 ) & H^1 _f (G_{\mathfrak{p} },V) & 0 \\
0 & D_{\dr }(W)/F^0 &  D_{\dr }(E_2 )/F^0  &  D_{\dr }(V)/F^0  & 0\\};
\path[->]
(m-1-1) edge[auto] node[auto]{ }(m-1-2)
(m-1-2) edge[auto] node[auto]{ }(m-1-3)
edge[auto] node[auto]{$\simeq $ }(m-2-2)
(m-1-3) edge[auto] node[auto]{ }(m-1-4)
edge[auto] node[auto]{ $\simeq $ }(m-2-3)
(m-1-4) edge[auto] node[auto]{ }(m-1-5)
edge[auto] node[auto]{ $\simeq $ }(m-2-4)
(m-2-1) edge[auto] node[auto]{ }(m-2-2)
(m-2-2) edge[auto] node[auto]{ }(m-2-3)
(m-2-3) edge[auto] node[auto]{ }(m-2-4)
(m-2-4) edge[auto] node[auto]{ }(m-2-5);
\end{tikzpicture} 
$$
(exactness of the top row follows from the isomorphism with the bottom row).
Define $\tau _{E_2 }: D_{\dr }(V)/F^0 \to D_{\dr }(E_2 )/F^0 $ to be the composite 
$$
D_{\dr }(V)/F^0 \stackrel{s}{\longrightarrow } D_{\cry }(V)\stackrel{r}{\longrightarrow} D_{\cry }(E_2 )\to D_{\dr }(E_2 )/F^0 
$$
where $s$ is the chosen splitting of the Hodge filtration, $r$ is the Frobenius-equivariant section of $D_{\cry }(E_2 )\to D_{\cry }(V)$, which exists and is unique by our assumptions on the Frobenius eigenvalues of $V$, and the third map is the projection. Then we define $$t(E_1 ,E_2 ):=\tau _{E_2 }(E_1 ).$$ 


For $M$ in $\mathcal{M}_f (G_{\mathfrak{p}} ;\mathbb{Q}_p ,V,W)$, let $E_1 $ and $E_2 $ be $M/M_2 $ and $\ker (M\to M_0 )$ respectively. Let $[M]$ denote the image of $M$ in $H^1 _f (G_{\mathfrak{p}},E_2 )$. Then we find that $[M]$ and $t(E_1 ,E_2 )$ have the same image in $H^1 _f (G_{\mathfrak{p}} ,V)$, hence by 
the diagram above, $[M]-t(E_1 ,E_2 )$ defines an element of $H^1 _f (G_{\mathfrak{p}},W)$, and we define 
$$
\widetilde{h}_{\mathfrak{p}} (M):=[M]-t(E_1 ,E_2 )\in H^1 _f (G_{\mathfrak{p}},W).
$$
The pre-height can be described explicitly as an algebraic function 
\[
F^0 \backslash U(\Q _p ,V_{\dr },W_{\dr })\to W_{\dr }/F^0 .
\]
\begin{Lemma}\label{explicit_formula}
Let $M$ be a mixed extension in $M_{\fil ,\phi }(\Q _p ,V_{\dr}, W_{\dr })$ given by 
$1+\alpha +\beta +\gamma \in U(\Q _p ,V_{\dr} ,W_{\dr} )$, where $\alpha \in V_{\dr} ,\beta \in V_{\dr }^* \otimes W_{\dr }$, $\gamma \in W_{\dr} $. In block matrix notation, $M$ is represented by 
\[
\left( \begin{array}{ccc}
1 & 0 & 0 \\
\alpha & 1 & 0 \\
\gamma & \beta & 1 \\
\end{array} \right) .
\]
Then 
\[
\widetilde{h}_p (M)=\gamma -\beta (s_1 (\alpha )),
\]
where
\[
s_1 :v\mapsto v-\iota \circ s(v)
\]
is the projection onto the $V_{\dr }/F^0 $ summand induces by the splitting $s$.
\end{Lemma}
\begin{proof}
The class of the extension $M$ in $M_1 /F^0 $ is given by $t^\phi -t^H$, where $t^\phi ,t^H$ are isomorphisms of filtered vector spaces 
\[
\Q _p \oplus M_1 \stackrel{\simeq }{\longrightarrow }M
\]
respecting the Frobenius action and Hodge filtration respectively. Hence, in terms of $s^\phi $ and $s^H$, this class is given by $s^H (\alpha +\gamma )$. Then the extension class $t([\alpha ])$ is given by $s^H (s_1 (\alpha )+\beta (s_1 (\alpha )))$. Hence 
the local height is given explicitly by 
\[
\widetilde{h}_p (M)=\gamma -\beta (s_1 (\alpha )).
\]
\end{proof}
\begin{Lemma}\label{itsonto}
 For any choice of splitting of the Hodge filtration, the composite map
$$
\pi _* \times \widetilde{h}_{\mathfrak{p}} :H^1 _f (G_{\mathfrak{p}} ,U)\to H^1 (G_{\mathfrak{p}} ,V)\times H^1 (G_{\mathfrak{p}} ,W)
$$
is an isomorphism of algebraic varieties.
\end{Lemma}
\begin{proof}
The fact that the pre-height is algebraic follows from the explicit formula in Lemma \ref{explicit_formula}.
It is enough to prove that the corresponding map
$$
D_{\dr }(U)/F^0 \to D_{\dr }(V)/F^0 \times D_{\dr }(W)/F^0
$$
is an isomorphism. 
We have a commutative diagram
$$
\begin{tikzpicture}
\matrix (m) [matrix of math nodes, row sep=3em,
column sep=3em, text height=1.5ex, text depth=0.25ex]
{ D_{\dr }(U)/F^0 & D_{\dr }(V)/F^0 \times D_{\dr }(W)/F^0 \\
 D_{\dr }(U(\mathbb{Q}_p ,V,W))/F^0 & D_{\dr }(V)/F^0 \times D_{\dr }(V^* \otimes W)/F^0 \times D_{\dr }(W)/F^0 \\};
\path[->]
(m-1-1) edge[auto] node[auto]{}(m-2-1)
edge[auto] node[auto]{ }(m-1-2)
(m-2-1) edge[auto] node[auto] {} (m-2-2)
(m-1-2) edge[auto] node[auto] {} (m-2-2);
\end{tikzpicture} $$
where the righthand map sends $(v,w)$ to $(v,[IA]+\tau _* v ,w)$, and the lefthand map sends $P$ to $A_{\dr }^{(P)}$. Both maps are closed immersions.
We first construct an inverse to the bottom map. Given $(v,\alpha ,w)$ in $D_{\dr }(V)/F^0 \times D_{\dr }(V^* \otimes W)/F^0 \times D_{\dr }(W)/F^0$, the mixed extension $t(v,\alpha )$ defines 
an element of $D_{\dr }(U(\mathbb{Q}_p ,V,W))/F^0 $, and we define $t(v,\alpha )^{(w)}$ to be the twist of $t(v,\alpha )$ by $w$. The map $(v,\alpha ,w)\mapsto t(v,\alpha )^{(w)}$ gives the desired inverse. When we restrict this map to $D_{\dr }(W)/F^0 $, it induces an inverse to the top map, as required.
\end{proof}
One may view the above lemma as saying that the fact that $H^1 _f (G_{\mathfrak{p}} ,U)$ is non-canonically isomorphic to $H^1 _f (G_{\mathfrak{p}} ,V)\times H^1 _f (G_{\mathfrak{p}} ,W)$ is an analogue of the fact that the $p$-adic height pairing depends on a choice of splitting of the Hodge filtration.

\subsection{Global height: definition and basic properties}\label{defnglobalheight}
Define $H^1 _s (G_{K,T},W ^* (1)):=H^1 (G_{K,T} ,W^* (1))/H^1 _f (G_{K,T} ,W^* (1))$.
Let $\chi $ be a nonzero element of $H^1 _s (G_{K,T} ,W^* (1))$, which is non-crystalline at $\mathfrak{p}$.
Given $\chi $ and a collection of local pre-heights $(\widetilde{h}_v )$ as above we define the associated
local height to be 
$$
h_v :=\chi _v \cup \widetilde{h}_v :M(G_v ;\mathbb{Q}_p ,V,W)\to H^2 (G_v ,\mathbb{Q}_p (1))\simeq \mathbb{Q}_p
$$
and the global height to be
$$
h=\sum _v h_v :M_{f,T_0 } (G_{K,T}; \mathbb{Q}_p ,V,W)\to \mathbb{Q}_p,
$$
where $v \in T_0 \cup \{ \mathfrak{p} \}$.  When we want to indicate the dependence on $\chi $, we write $h_{v,\chi }$ and $h_{\chi }$. Since $h$ and $h_v$ are linear in $\chi $, we may define a \textit{universal} height
$$
\widetilde{h}:M_{f,T_0 } (G_{K,T}; \mathbb{Q}_p ,V,W)\to (H^1 _s (G_{K,T} ,W^* (1)))^*
$$
\index{$\widetilde{h}$}
by setting $\widetilde{h}(M)$ to be the functional $\chi \mapsto h_{\chi }(M)$.

Note that by construction, $\widetilde{h}_v $ is bi-additive in the same way that usual local heights are bi-additive (see e.g. \cite[\S 4]{balakrishnan:2016}). Namely, for $i=1$ or 2, if $M$ and $N$ satisfy $\pi_{i*}(M)=\pi _{i*}(N)$, then we can form the sum $M+_{i,i}N$ in $M(G_v ;\mathbb{Q}_p ,V,W)$ (for example, when $i=1$, this is the Baer sum of the extensions $[M],[N]$ in $\ext ^1 (M/M_2 ,W)$), and its local pre-height will be equal to the sum of the local pre-heights of $M$ and $N$. If $P=(P_v )_{v\in T}\in \prod _{v|p} M_f (G_{v} ;\mathbb{Q}_p ,V,W)$ $ \times \prod _{v\in T_0 }M(G_v ;\mathbb{Q}_p ,V,W)$, then we similarly define $h(P)$ to be the sum of the local heights.
\begin{Lemma}\label{pairing}
The global height $h$ factors as 
$$
M_{f,T_0 } (G_{K,T};\mathbb{Q}_p ,V,W)\to H^1 _f (G_{K,T},V)\times H^1 _f (G_{K,T},V^* \otimes W) \to \mathbb{Q}_p,
$$
where the first map is the projection and the second is bilinear.
\end{Lemma}
\begin{proof}
As remarked above, the global height is additive, so it is enough to show that it is invariant under the action of $H^1 _f (G_{K,T} ,W)$ on $M_{f,T_0 } (G_{K,T} ;\mathbb{Q}_p,V,W)$. Invariance follows from Poitou--Tate duality: if a mixed extension $M$ is twisted by $c \in H^1 _f (G_{K,T} ,W)$, then this will change $h_v $ by $\chi _v \cup \loc _v c$, and $\sum \chi _v \cup \loc _v c =0$.
\end{proof}

\begin{Remark}
Note that, unlike classical $p$-adic heights, it is not clear that this construction defines a pairing 
$
H^1 _f (G_{K,T} ,V)\times H^1 _f (G_{K,T} ,V^* \otimes W)\to \mathbb{Q}_p, 
$
as we do not know that given $[E_1 ]$ in $H^1 _f (G_{K,T}, W)$ and $[E_2 ]$ in $\ext ^1 _f (V,W)$, $E_2 $ lifts to an element of $H^1 _f (G_{K,T} ,E_2 )$. The existence of such a lifting is equivalent to the vanishing of $[E_1 ]\cup [E_2 ]$ in $H^2 (G_{K,T} ,W)$, and hence would be implied by injectivity of the localisation map 
$
H^2 (G_{K,T} ,W)\to \oplus _{v\in T}H^2 (G_v ,W).
$
By Poitou-Tate duality, this would be implied by injectivity of $H^1 (G_{K,T} ,W^* (1))\to \oplus _{v\in T}H^1 (G_v ,W^* (1))$, and 
hence by Conjecture \ref{BKconj}, as in Lemma \ref{BKimpliesfinite}.
\end{Remark}

Given two different choices of splitting of the Hodge filtration $s^{(1)}$ and $s^{(2)}$, we obtain two different pre-heights $\widetilde{h}_p ^{(1)}$ and $\widetilde{h}_p ^{(2)}$. Their difference $\widetilde{h}_p ^{(1)}-\widetilde{h}_P ^{(2)}$ defines a map
$
M_f (G_{\mathfrak{p}},\mathbb{Q}_p ,V,W)\to D_{\dr}(W)/F^0 ,
$
which may easily be seen to factor as
$$
M_{f} (G_{\mathfrak{p}},\mathbb{Q}_p ,V,W)\to D_{\dr}(V)/F^0 \times D_{\dr}(V^* \otimes W)/F^0 \to D_{\dr}(W)/F^0 .
$$
The latter map may be defined as follows. The difference $s^{(1)}-s^{(2)}$ gives a homomorphism
$
\overline{s}:D_{\dr}(V)/F^0 \times  \to F^0 D_{\dr }(V).
$
Given $v\in D_{\dr}(V)/F^0 $ and $\alpha \in D_{\dr}(V^* \otimes W)/F^0 $, choose a lift of $\alpha $ to $\widetilde{\alpha }$ in $D_{\dr}(V^* \otimes W)$. The lift $\widetilde{\alpha }(\overline{s}(v))$ gives an element of $D_{\dr}(W)$, which is independent of the choice of $\widetilde{\alpha }$ modulo $F^0 D_{\dr}(W)$.

\begin{Lemma}\label{heights_prove_global}
Suppose $[P]=([P_v ])\in \prod _{v\in T_0 \cup \{ \mathfrak{p} \} }M (G_v ;\mathbb{Q}_p ,V,W)$ satisfies 
\begin{itemize}
\item $P_{\mathfrak{p}} $ is crystalline.
\item $\pi _{1*} P\in H^1 _f (G_{\mathfrak{p}} ,V)$ is in the image of $H^1 _f (G_{K,T} ,V)$,
\item $\pi _{2*} P\in H^1 _f (G_{\mathfrak{p}} ,V^* \otimes W)$ is in the image of $H^1 _f (G_{K,T} ,V^* \otimes W)$,
\item there exist $P_1 ,\ldots P_n $ in $M_{f,T_0 } (G_{K,T} ;\mathbb{Q}_p ,V,W)$, and $\lambda _i $ in $\mathbb{Q}_p $ such that 
$$
\pi _{1*}P \otimes \pi _{2*}P=\sum \lambda _i \pi _{1*}P_i \otimes \pi _{2*}P_i
$$
in $H^1 _f (G_{K,T} ,V)\otimes H^1 _f (G_{K,T} ,V^* \otimes W)$ and for all $\varphi $ in $H^1 _s (G_{K,T} ,\overline{\wedge ^2 V}^* (1))$, 
$$
h_{\varphi }(P)=\sum \lambda _i h_{\varphi }(P_i ).
$$
\end{itemize}
Then $P$ is in the image of $M_{f,T_0 } (G_{K,T} ;\mathbb{Q}_p ,V,W)$.
\end{Lemma}
\begin{proof}
We have an exact sequence of unipotent groups with $G_{K,T}$-action
$$
1\to W \to U(\mathbb{Q}_p ,V,W)\stackrel{\pi _1 \oplus \pi _2 }{\longrightarrow} V\oplus V^* \otimes W \to 1.
$$
The image of $H^1 (G_{K,T} ,U(\mathbb{Q}_p ,V,W))$ in $H^1 (G_{K,T} ,V\oplus V^* \otimes W)$ is  precisely equal to the kernel of the cup product map to $H^2 (G_{K,T} ,W)$. Note that
$$
\pi _{1*}P \cup \pi _{2*}P=\sum \lambda _i \pi _{1*}P_i \cup \pi _{2*}P_i =0,
$$
and thus we conclude that there is a mixed extension $P'$ whose image in $H^1 _f (G_{\mathfrak{p}} ,V) \otimes H^1 _f (G_{\mathfrak{p}} ,V^* \otimes W)$ is equal to that of $P$. Hence $P$ is the twist of $\loc _{\mathfrak{p}} P'$ by some $c$ in $H^1 (G_{\mathfrak{p}} ,W)$, and the claim of the lemma is exactly that this $c$ is in the image of $H^1 _f (G_{K,T} ,W)$. By Poitou--Tate duality this is true if and only if for all $\varphi $ in $H^1 (G_{K,T} ,W^* (1))$ which are crystalline at all primes above $p$ other than $\mathfrak{p}$, we have $\sum _v \varphi _v \loc _v c=0$. But, as in the proof of Lemma \ref{pairing},
$$h_{v,\varphi }(P)=h_{v,\varphi }(P')+\varphi _v \loc _v (c).$$
\end{proof}

\section{Equations for Selmer varieties}

In this section we use the bilinear structure of generalised heights to obtain formulas for $X(\mathbb{Q}_p )_U$. More precisely, generalised heights allow us to describe explicit trivialisations
\[
M_* (G_v ;\Q _p ,V,W)\simeq H^1 _* (G_v ,V)\times H^1 _* (G_v ,V^* \otimes W) \times H^1 _* (G_v ,W) 
\]
(where $*$ is $f$ or $g$ depending on whether or not $v|p$), and to describe the image of $M_{f,T_0 } (G_{K,T};\Q _p ,V,W)$ under the map
\[
M_{f,T_0 } (G_{K,T};\Q _p ,V,W)\stackrel{(\pi _{1*},\pi _{2*}, \loc )}{\longrightarrow } 
\begin{array}{c}H^1 _f (G_{K,T} ,V)\times H^1 _f (G_{K,T} ,V^* \otimes W) \\
\times \prod _{v|p}M_{f} (G_v ;\Q _p ,V,W) \times \prod _{v\in T_0 }M(G_v ;\Q _p ,V,W).
\end{array}
\]
In Lemma \ref{lemma13}, this is used to describe $X(K_{\mathfrak{p}})_{\alpha }$, by giving explicit quadratic relations between $\widetilde{h}_{\mathfrak{p}}(A(b,z))$ and $\kappa (b-z)$.

Fix a prime $\mathfrak{p}$ above $p$ and a set of local conditions $$\alpha \in \prod _{v\in T_0 }j_v (X(\mathbb{Q}_v ))\subset \prod _{v\in T_0 }H^1 (G_v ,U).$$ 
For $\alpha =(\alpha _v )_{v\in T_0 }$ in $\prod _{v\in T_0 }H^1 (G_v ,W)$, let $M_f (G_{K,T} ; \mathbb{Q}_p ,V,W)_{\alpha }$ denote the set of isomorphism classes of mixed extensions which are crystalline at $p$, and such that the localisation at $v\in T_0 $ corresponds to $\alpha _v \in H^1 (G_v ,W)$ via the isomorphism 
$
M(G_v ;\mathbb{Q}_p ,V,W) \simeq H^1 (G_v ,W).
$
Then the twisting construction defines a map 
$$
\Sel (U)_{\alpha }\to M_f (G_{K,T} ;\mathbb{Q}_p ,V,W)_{\alpha }.
$$

Let $m$ denote the codimension of $H^1 _f (G_{K,T},W)$ in $H^1 _f (G_{\mathfrak{p}},W)$.
Suppose $P\in H^1 _f (G_{\mathfrak{p}} ,U)$ comes from some $P'$ in $\Sel (U)_{\alpha }$, and let $Q$ denote the image of $P'$ in $H^1 _f (G_{K,T} ,V)$. Knowing $\pi _* P$ gives $g$ linear conditions on $Q$, and knowing $\widetilde{h}_p (P)$ gives $m$ quadratic conditions on $Q$. Finding exact formulas for the subspace of $H^1 _f (G_{\mathfrak{p}} ,U)$ where these $g+m$ equations have a solution is then a matter of elimination theory. Concretely, let $H_\alpha $ be the image of 
$\mathbb{Q}_p [\Sel (U)_{\alpha }]$ in $H^1 _f (G_{K,T} ,V)\otimes H^1 _f (G_{K,T} ,V^* \otimes W)$ under the map $P\mapsto \pi _{1*}(P)\otimes \pi _{2*}(P)$. Let
$S$ be a section of 
$
\mathbb{Q}_p [\Sel (U)_{\alpha }]\to H_\alpha .
$
Let $H$ be the image of $\Sel (U)_\alpha $ in $H^1 _f (G_{K,T} ,V)$, and let $T$ denote the map 
$
H\to H_\alpha 
$
sending $P$ to $P\otimes ([IA(b)]+\tau _* (P))$. Then by the multilinearity of generalised heights we have 
\begin{equation}\label{eqn:multilin}
\widetilde{h}(P)=\widetilde{h}(S \circ T ( \pi _* P))
\end{equation}
for all $P\in \Sel (U)_{\alpha }$.
To use this to write down equations for $\loc _{\mathfrak{p}} (\Sel (U)_{\alpha })$, we introduce some notation for resultants (see e.g. \cite[\S IX.3]{lang:2002}). Given finite-dimensional vector spaces $V_1 ,V_2 ,V_3 $ over a field $K$ and a morphism of algebraic varieties
$
F:V_1 \times V_2 \to V_3,
$
we define the resultant $R_{V_2 }(F)\subset \mathcal{O}(V_1 )$ to be the ideal defining the maximal subvariety of $V_1$ for which $F|_{R_{V_2 }(F)\times V_2 }$ is identically zero. By the fundamental theorem of elimination theory, this is of finite type over $K$. If $(\lambda _1 ,\ldots \lambda _n )$ is a basis for $V_2 $, we may also write this as $R_{\lambda _1 ,\ldots ,\lambda _n }(F)$, to indicate that the variables $\lambda _1 ,\ldots ,\lambda _n$ have been eliminated. In our case of interest, 
$$V_1=H^1 _f (G_{\mathfrak{p}} ,V)\oplus H^1 _f (G_{\mathfrak{p}} ,W), \qquad V_2 =H, \qquad V_3 =H^1 _f (G_\mathfrak{p},V)\oplus H^1 _s (G_{K,T} ,W^* (1)),$$ and the map
$$
F:V_1 \times V_2 \to V_3
$$
sends $(v_1 ,v_2 ,v_3 )$ in 
$H^1 _f (G_{\mathfrak{p}} ,V )\times H^1 _f (G_{\mathfrak{p}} ,W)\times H $ to \\
$$
(\loc _{\mathfrak{p}} (v_3 )-v_1 ,\widetilde{h}_{\mathfrak{p}} (v_2 )+\sum _{v\in T_0 }\widetilde{h}_v  ( \alpha _v )-\widetilde{h}(S\circ T(v_3 )) \in H^1 _f (G_{\mathfrak{p}} ,V)\oplus H^1 _s (G_{K,T} ,W^* (1)).
$$
\begin{Lemma}\label{lemma13}
The image of $\Sel (U)_{\alpha }$ in $H^1 _f (G_\mathfrak{p} ,V)\times H^1 _f (G_{\mathfrak{p}}, W)$ under the composite map $(\pi _* ,\widetilde{h}_p )\circ \loc _{\mathfrak{p}} $ is equal to the zero set of $R_H (F)$.
In particular 
$$
X(K_{\mathfrak{p}} )_{\alpha }=\{ z\in X(K_{\mathfrak{p}} ):\;\textrm{for all}\; G\in R_H (F),\; G(\kappa _p (z),\widetilde{h}_p (j_{\mathfrak{p}} (z)))=0  \}.
$$
\end{Lemma}
\begin{proof}
Whenever $P$ is in the image of $\Sel (U_2 )_{\alpha }$, it satisfies the equations above. Conversely, by Lemma \ref{mixed_injective}, 
there is a global $U$-torsor in $\Sel (U_2 )_{\alpha }$ whose localisation at $\mathfrak{p}$ is given by $P$ if and only if there is a mixed extension in $M_f (G_{K,T} ;\mathbb{Q}_p ,V,W)_{\alpha }$ whose localisation at $\mathfrak{p}$ is given by $A^{(P)}$. By Lemma \ref{heights_prove_global}, this happens if and only if there is an element $Q$ of $H^1 _f (G_{K,T} ,V)$ which is a simultaneous solution to 
\[ \left\{
\begin{array}{ll}
\loc _{\mathfrak{p}} (Q )=\pi _* P \\
h_{\varphi ,\mathfrak{p}}(P)+\sum _{v\in T_0 }h_{\varphi ,v}(\alpha _v )=h_{\varphi }(S\circ T(Q )).
\end{array}
\right.\]
\end{proof}

\begin{Lemma}\label{mixed_injective}
The map
\[
H^1 _{f,T_0 } (G_{K,T} ,U)\to M_{f,T_0 } (G_{K,T} ;\mathbb{Q}_p ,V,W); \qquad
P \mapsto A(b)^{(P)}
\]
is injective.
\end{Lemma}
\begin{proof}
As explained in \cite[\S 5.1]{balakrishnan:2016}, this map may be described as the composite 
\begin{align*}
H^1 _{f,T_0 } (G_{K,T} ,U)\to H^1 _{f,T_0 } (G_{K,T} ,\Aut (A(b))) \stackrel{\simeq }{\longrightarrow }M_{f,T_0 } (G_{K,T} ;\mathbb{Q}_p ,V,W),
\end{align*}
where the first map is induced from the group homomorphism $U\to \Aut (A(b))$ and the second map is induced from the isomorphism $$
M_{f,T_0 } (G_{K,T} ;\mathbb{Q}_p ,V,W)\simeq H^1 _{f,T_0 } (G_{K,T} ;U(\mathbb{Q}_p ,V,W))
$$
together with the structure of $A(b)$ as an $(\Aut (A(b)),U(\mathbb{Q}_p ,V,W))$-bitorsor.
The upshot is that it suffices to check the first map is injective. By definition of the map, this is implied by injectivity of 
$$
H^1 (G_{K,T} ,U)\to H^1 (G_{K,T} ,\Aut (A(b))).
$$
By the exact sequence 
$$
H^0 (G_{K,T} ,\Aut(A(b))/U) \to H^1 (G_{K,T} ,U)\to H^1 (G_{K,T} ,\Aut (A(b)))
$$ 
(see e.g., \cite[Proposition 36]{serregc:1997})
it is enough to show that the pointed $G_{K,T} $-set 
$\Aut (A(b))/U$ has no fixed points, which can be seen by noting that it is an extension of a weight $-1$ $G_{K,T}$-representation by a weight $-2$ $G_{K,T}$-representation.
\end{proof}

Let $P_1 ,\ldots ,P_n$ be elements of $\Sel (U)_{\alpha }$ in $H^1 _f (G_{K,T} ,V)$ spanning $H$ (recall that this is the image of $\mathbb{Q}_p [\Sel (U)_{\alpha }]$ in $H^1 _f (G_{K,T} ,V)\otimes H^1 _f (G_{K,T} ,V^* \otimes W)$), and such that $P_1 ,\ldots ,P_r $ span the image of $\Sel (U)_{\alpha }$ in $H^1 _f (G_{K,T} ,V)$ under $\pi _*$.
Suppose $m_{ijk},m_{ik}\in \Q _p $ satisfy
\begin{align}
(0 ,\tau _* \pi _* P_j \otimes \pi _* P_i ) & =\sum m_{ijk}(\pi _* P_k ,\tau _* \pi _* P_k \otimes \pi _* P_k ) \label{consts} \\
(\pi _* P_i  ,0) & =\sum m_{ik}(\pi _* P_k ,\tau _* \pi _* P_k \otimes \pi _* P_k ) \notag \\
\end{align}
Let $z\in X(K_{\mathfrak{p}})$. Then if $j_{\mathfrak{p}} (z)$ is in the image of $\loc _{\mathfrak{p}} (\Sel (U)_{\alpha })$, there are $\lambda _1 ,\ldots ,\lambda _r $ such that
$$
 \kappa _p (z)=\sum \lambda _i \loc _p \pi _* P_i \\
$$
and for all $\varphi \in H^1 _s (G_{K,T} ,W^* (1))$,
$$
h_{\mathfrak{p},\varphi }(j_{\mathfrak{p}} (z))+\sum _{v\in T_0 }h_{\mathfrak{p},\varphi }(\alpha _v )=\sum \lambda _i m_{ik}h _{\varphi }(P_k )+\sum \lambda _i \lambda _j m_{ijk}h_{\varphi }(P_k )
$$
since if $j_p (z)$ comes from some $P\in \Sel (U)_\alpha $, then we must have
$$
\pi _{1*} (A(b)^{(P)})\otimes \pi _{2*}(A(b)^{(P)})=\sum _i \lambda _i \pi _{*}(P_i )\otimes ([IA(b)]+\sum _j \lambda _j (P_j )),
$$
in $H^1 _f (G_{K,T} ,V)\otimes H^1 _f (G_{K,T} ,V^* \otimes W) $,
which is equal to the 
class of 
$$
\sum \lambda _i m_{ik}\pi _{1*}(A(b)^{(P_k  )})\otimes \pi _{2*}(A(b)^{(P_k  ) })+\sum \lambda _i \lambda _j m_{ijk}\pi _{1*}(A(b)^{(P_k  ) })\otimes \pi _{2*}(A(b)^{(P_k  ) })
$$
by assumption.  This gives the following explicit version of Lemma \ref{lemma13}.
\begin{Proposition}\label{exactformula1}
Suppose the kernel of $\Div ^0 (X(\mathbb{Q}))/P \otimes \mathbb{Q}_p \to J(K_{\mathfrak{p}} )$ has rank $k_1 $, and that the codimension of $H^1 _f (G_{K,T} ,W)$ in $H^1 _f (G_{\mathfrak{p}}, W)$ is $k_2 $. Let $k=k_2 -k_1 $. Then 
$$
X(K_{\mathfrak{p}})_{\alpha }=\cap _{1\leq i\leq k}\{ R_{ \lambda_1 , \ldots , \lambda_{k_1 }}(\mathcal{F}_z (\lambda _1 ,\ldots ,\lambda _{k_1 }) =0 \},
$$
where 
$$
\mathcal{F}_z (\lambda _1 ,\ldots ,\lambda _{k_1 })= h_{p,\varphi }(j_p (z))+\sum _{v\in T_0 }h_{p,\varphi }(\alpha _v )-\sum \lambda _i m_{ik}h_{\varphi }(P_k )-\sum \lambda _i \lambda _j m_{ijk}h_{\varphi }(P_k ),
$$
and$m_{ijk}$ and $m_{ik}$ are as in equation \eqref{consts}.
\end{Proposition}

In particular, if the Mordell--Weil rank of the Jacobian of $X$ is less than or equal to $g$, and the map 
$\Div ^0 (X(K))/P \otimes \mathbb{Q}_p \to J(K)\otimes \Q _p $ is injective, then 
$$
X(K_{\mathfrak{p}} )_{\alpha }=\cap _i \{h_{\mathfrak{p},\varphi _i}(z)+\sum h_{v,\varphi _i }(\alpha _v )-h_{v,\varphi _i }(S\circ T(j_{\mathfrak{p}} (z))) =0 \} ,
$$
where $\varphi _1 , \ldots ,\varphi _r $  is a basis for $H^1 _s (G_{K,T} ,\overline{\wedge ^2 V}^* (1))$.
\subsection{Equivariant height pairings}
For the Manin--Demjanenko type results in the next section, it will be crucial to consider the subset of height functions which are equivariant with respect to extra endomorphisms of $J$. 

\begin{Definition}\label{gamma_action}
Let $\gamma \in H^0 (G_{K,T} ,GL(V))$. Then $\gamma $ acts on $M_{f,T_0 } (G_{K,T};\mathbb{Q}_p ,V,W)$ by sending $(M,(M_i ),(\psi )_i )$ to $(M,(M_i ),(\psi _i ' ))$ where $\psi _1 ' =\psi _1 \circ \gamma $ and for $i=0,2$ $\psi _i ' =\psi _i $. We denote this action by $\gamma ^* $.
\end{Definition}
Let $\red _{F^\bullet }$ denote the quotient map $W_{\dr }\to W_{\dr }/F^0 $. The splitting $s$ induces sections of $W_{\dr }\to W_{\dr }/F^0 $ and 
$V^* \otimes W\to (V^* \otimes W)/F^0 $ as follows: the isomorphism $V_{\dr }\simeq \gr _F V_{\dr }$ induces an isomorphism $\gr _F (V_{\dr }\otimes V_{\dr })$, and hence together with the surjection $V_{\dr }\otimes V_{\dr }\to W_{\dr }$, we get an isomorphism $W_{\dr }\simeq \gr _F W_{\dr }$. We denote the induced section of $V^* \otimes W \to (V^* \otimes W)/F^0 $ by $s$.
\begin{Lemma}\label{equivtheight}We have the following:
\begin{enumerate}
\item For $v\neq p$, and $\gamma $ in $R^\times $, we have 
$h_v (M)=h_v (\gamma ^* M)$.
\item
The map 
$
M\mapsto \widetilde{h}_{\mathfrak{p}} (M) -\widetilde{h}_{\mathfrak{p}} (\gamma M).
$
factors as 
$$
M_f (G_{\mathfrak{p}} ;\mathbb{Q}_p ,V,W) \to V_\dr /F^0 \times (V_{\dr }^* \otimes W_{\dr })/F^0 \stackrel{\overline{h}_\gamma }{\longrightarrow}W_{\dr }/F^0
$$
where $\overline{h}_\gamma $ is the bilinear map
$$
(v,w)\mapsto \red_{F^\bullet }(s(w) \circ s)(v)-\red _{F^\bullet }(\gamma ^* s(w) ) \circ s(\gamma _* v).
$$
\item If $\gamma $ commutes with the splitting of the Hodge filtration, then $h_{\mathfrak{p}} (M)=h_{\mathfrak{p}} (\gamma M)$.
\end{enumerate}
\end{Lemma}
\begin{proof}
First note that for $v\neq p$, $h_v (M)=h_v (\gamma M)$, since the definition of $h_v $ does not depend on a choice of isomorphism $M_2 /W \simeq V$.
For the second claim, note that by definition of $\widetilde{h}_{\mathfrak{p}} $ we have
$$
h_{\mathfrak{p}} (M)-h_{\mathfrak{p}} (\gamma M)=t(M_1 ,M_2 )-t(\gamma ^* M_1 ,\gamma _* M_2 ),
$$
as required.
For the last part, note that if $\gamma $ commutes with $s$ we have 
$$
\red_{F^\bullet }(\gamma ^* s(w) ) \circ s\circ \gamma _* )=\red _{F^\bullet }((\gamma ^* s(w) ) \circ \gamma _* \circ s )=\red _{F^\bullet }(s(w)\circ s)
$$
which by the above implies $h_{\mathfrak{p}} (M)=h_{\mathfrak{p}} (\gamma M)$.
\end{proof}
As a result, the height $h$ is $R$-equivariant if and only if for all $\gamma $ in $R^\times $, $s(w) \circ s =(\gamma ^* s(w) ) \circ s\circ \gamma _*$ modulo $F^0 W$.

\section{Generalised heights on hyperelliptic curves}

In this section we prove Theorem \ref{thm1} and the finiteness part of Theorem \ref{thm2}, using equivariant heights. In brief, the previous section explained how generalised heights provided non-trivial quadratic relations between $\widetilde{h}(A(b,z))$ and $\kappa (z-b)$. To prove finiteness of $X(K_{\mathfrak{p}})_U$, one would like to find non-trivial polynomial relations between $\widetilde{h}(A(b,z))$ and $\loc _{\mathfrak{p}}(\kappa (z-b))$. In general, the obstruction to doing this lives in $H^1 _f (G_{K,T},V)\otimes H^1 _f (G_{K,T},V^* \otimes W)$, in some sense. The idea of using equivariant heights on hyperelliptic curves is to try and replace this with a smaller obstruction space.
\begin{Definition}
Define the \textit{hyperelliptic subspace} of $H^1 _f (G_{K,T} ,V)\otimes H^1 _f (G_{K,T} ,V^* \otimes W)$ to be the image of $H^1 _f (G_{K,T} ,V)^{\otimes 2}$ under the map $1\otimes \tau _* $, where $\tau $ is as defined in \ref{defn:tau}.
Define the \textit{hyperelliptic subspace} of $M_{f,T_0 } (G_{K,T} ;\mathbb{Q}_p,V,W)$, denoted $M_{f,T_0 } ^h (G_{K,T} ;\mathbb{Q}_p ,V,W)$, to be the subvariety of classes whose associated $H^1 _f (G_{K,T} ,V^* \otimes W)$ class is in the image of $\tau _* $.
\end{Definition}
The reason for the name is that, by Lemma \ref{hyperellipticsplitting}, the image of the Selmer variety of a hyperelliptic curve lies in the hyperelliptic subspace.
\begin{Lemma}\label{hyperheight}
Let $X$ be a hyperelliptic curve, $b$ a rational point and $U$ any quotient of $U_2 (b)$. Then the natural map
\[
H^1 _{f,T_0 } (G_{K,T} ,U) \to M_{f,T_0 } (G_{K,T};\mathbb{Q}_p ,V ,W); 
P \mapsto A(b)^{(P)}
\]
lands in the hyperelliptic subspace.
\end{Lemma}
One may straightforwardly extend this to equivariant heights.
\begin{Lemma}
Suppose $s$ is an $R$-equivariant splitting. Then 
\begin{enumerate}
\item The generalised height 
$$
h:M_{f,T_0 } (G_{K,T} ;\mathbb{Q}_p ,V,W)\to \mathbb{Q}_p
$$
factors through $H^1  (G_{K,T} ,V)\otimes _R H^1 _f (G_{K,T} ,V^* \otimes W)$.
\item The generalised height function, restricted to the hyperelliptic subspace, factors through
$H^1 _f (G_{K,T} ,V)\otimes _{R,\tau }H^1 _f (G_{K,T} ,V)$.
\end{enumerate}
\end{Lemma}
\begin{proof}
By Lemma \ref{pairing} and Lemma \ref{hyperheight}, we only need to prove $R$-equivariance. To prove this, it will be enough to prove that, for all $\gamma \in R^\times $, $h(\gamma ^* E_1 ,E_2 )=h(E_1 ,\gamma ^* E_2 )$. It is enough to prove this locally, i.e. to prove that for all mixed extensions $M$,
\[
h_v (M)=h_v (\gamma ^* M).
\]
This follows from Lemma \ref{equivtheight}. \\
\end{proof}

We now explain the application to finiteness of Chabauty--Kim sets.
\begin{Proposition}\label{maintheorem}
Let $X$ be a hyperelliptic curve. Let $R=\End _{K} ^0 (J)$. Suppose 
\begin{align*}
& \dim (H^1 _f (G_{\mathfrak{p}} ,W)/\loc _{\mathfrak{p}} H^1 _{f} (G_{K,T} ,W))-\dim (H^1 _f (G_{K,T} ,V)\otimes _R H^1 _f (G_{K,T} ,V)) \\
& \quad +\dim (H^1 _f (G_{\mathfrak{p}} ,V)\otimes _R H^1 _f (G_{\mathfrak{p}} ,V))>0.
\end{align*}
Then $X(K_{\mathfrak{p}} )_U $ is finite.
\end{Proposition}\begin{Remark}
Note that, given \cite[Lemma 3.2]{balakrishnan:2016}, this result is only new when 
\begin{align*}
& \dim (H^1 _f (G_{K,T} ,V)\otimes _R H^1 _f (G_{K,T} ,V))-\dim (H^1 _f (G_{\mathfrak{p}} ,V)\otimes _R H^1 _f (G_{\mathfrak{p}} ,V)) \\
& \quad <  \dim H^1 _f (G_{K,T} ,V)-\dim H^1 _f (G_{\mathfrak{p}} ,V).
\end{align*}
This can only happen when there are simple abelian varieties which occur as isogeny factors of $J$ with multiplicity greater than 1 (see the example below).
\end{Remark}
\begin{proof}
By \cite[Theorem 1]{kim:2009}, it is enough to prove that the localisation map 
$$
\Sel (U)\to H^1 _f (G_{\mathfrak{p}} ,U)
$$
is not dominant. Writing $\Sel (U)$ as a disjoint union of $\Sel (U)_{\alpha }$, for $\alpha $ a collection of local conditions, we reduce to proving that, for all $\alpha $, the localisation map
\[
\Sel (U)_{\alpha }\to H^1 _f (G_{\mathfrak{p}},U)
\]
is not dominant.
Let $$r=\dim (H^1 _f (G_{\mathfrak{p}} ,W)/\loc _{\mathfrak{p}} H^1 _f (G_{K,T} ,W))+\dim (H^1 _f (G_{\mathfrak{p}} ,V)\otimes _R H^1 _f (G_{\mathfrak{p}} ,V)).$$
We show that the codimension of 
$$
(\loc _{\mathfrak{p}} ,\pi _{1*}\otimes \pi _{2*}) :\Sel (U)_{\alpha }\to H^1 _f (G_p ,U)\times H^1 _f (G_{K,T} ,V)\otimes _R H^1 _f (G_{K,T} ,V)
$$
is greater than $r$, which proves the non-dominance of the localisation map by projecting.  We first choose a (vector space) section $t$ of the map
$$
\mathbb{Q}_p [M_{f,T_0 } (G_{K,T} ;\mathbb{Q}_p ,V,W)^h ]\to H^1 _f (G_{K,T} ,V)\otimes _R H^1 _f (G_{K,T} ,V).
$$
Define a map 
$$
H^1 _f (G_{\mathfrak{p}} ,U)\times H^1 _f (G_{K,T} ,V)\otimes _R H^1 _f (G_{K,T} ,V)\to H^1 _s (G_{K,T},W^* (1))
$$
by sending $(c,d)$ to $\widetilde{h}_{\mathfrak{p}} (c)+\sum _{v\in T_0 }\widetilde{h}_v (\alpha _v )-\widetilde{h}(t(d))$.
Then by equation \eqref{eqn:multilin}, the composite map 
$$
\Sel (U)_{\alpha }\to H^1 _f (G_{\mathfrak{p}} ,U)\times H^1 _f (G_{K,T} ,V)\otimes _R H^1 _f (G_{K,T} ,V)\to H^1 _s (G_{K,T},W^* (1))
$$
is identically zero. Similarly the composite map
\begin{align*}
\Sel (U)_{\alpha } & \to H^1 _f (G_{\mathfrak{p}} ,U)\times H^1 _f (G_{K,T} ,V)\otimes _R H^1 _f (G_{K,T} ,V) \\ 
& \stackrel{\pi _* \otimes \pi _* -\loc _{\mathfrak{p}} \otimes \loc _{\mathfrak{p}} }{\longrightarrow } H^1 _f (G_{\mathfrak{p}} ,V)\otimes _R H^1 _f (G_{\mathfrak{p}} ,V)
\end{align*}
is identically zero.
\end{proof}
\begin{Lemma} We have the following:
\begin{enumerate}
\item There is an $R$-equivariant pre-height.
\item The set of $R$-equivariant pre-heights is a $\Hom _{R\otimes \mathbb{Q}_p }(V/F^0 ,F^0 V)$-torsor.
\end{enumerate}
\end{Lemma}
\begin{proof}
By Lemma \ref{equivtheight}, $R$-equivariant pre-heights correspond to $R$-equivariant splittings of the Hodge filtration. By functoriality, $F^0 V$ is an $R\otimes \mathbb{Q}_p $-submodule of $V$. Since $R$ is semisimple, we deduce the existence of an $R$-equivariant splitting.  
\end{proof}

We now consider the setup of Theorem \ref{thm1}: $K=\mathbb{Q}$ or an imaginary quadratic field, the curve $X/K$ is hyperelliptic, with Jacobian $J$ isogenous to $A^d \times B$. Hence $R=\Mat _d (\mathbb{Q})$ is naturally a (non-unital) subalgebra of $\End ^0 (J)$. Let $V_A =T_p (A)\otimes \mathbb{Q}_p $ and $V_B =T_p (B)\otimes \mathbb{Q}_p $. Then $V\simeq V_A ^{\oplus d}\oplus V_B $. To apply Proposition \ref{maintheorem}, note that 
\[ H^1 _f (G_{K,T} ,V_A )^{\oplus d}\otimes _R H^1 _f (G_{K,T} ,V_A )^{\oplus d} \simeq H^1 _f (G_{K,T} ,V_A )\otimes _{\Q _p }H^1 _f (G_{K,T} ,V_A).
\]
We are now ready to give the proof of Theorem \ref{thm1}.
\begin{proof}[Proof of Theorem \ref{thm1}]
Let $\overline{\wedge ^2 (V_A ^{\oplus d} ) }$ be the quotient of $\wedge ^2 (V_A ^d ) $ by the image of $$\ker (\wedge ^2 V\to \overline{\wedge ^2 V})$$ under the projection from 
$\wedge ^2 V$ to $\wedge ^2 (V_A ^d ) $. 
First we prove that there is a quotient $W$ of $\overline{\wedge ^2 V}$ such that the quotient map factors through $\overline{\wedge ^2 (V_A ^{\oplus d} ) }$ and such that 
$$
\codim (\loc _{\mathfrak{p}} :H^1 _f (G_{K,T} ,W)\to H^1 _f (G_{\mathfrak{p}} ,W))=\rho _f (A)d+d(d-1)e(A)/2-1.
$$
We have $\wedge ^2 (V_A ^d )\simeq (\wedge ^2 V_A )^d \oplus (V_A ^{\otimes 2})^{d(d-1)/2}$.
Since $\NS (A_{\overline{K}})\otimes \mathbb{Q}_p (1)$ is a direct summand of $\wedge ^2 V_A$ and $\End ^0 (A_{\overline{K}})\otimes \mathbb{Q}_p (1)$ is a direct summand of $V_A ^{\otimes 2}$, we have a Galois-equivariant surjection
$$
\wedge ^2 V\to (\NS (A_{\overline{K}})\otimes \mathbb{Q}_p (1) ))^d \oplus (\End ^0 (A_{\overline{K}}\otimes \mathbb{Q}_p (1)))^{d(d-1)/2}.
$$
First suppose $K=\mathbb{Q}$. We take $W$ to be the quotient of $\overline{\wedge ^2 V}$ corresponding to  $(\NS (A_{\overline{\mathbb{Q}}})\otimes \mathbb{Q}_p (1) ))^d$ $\oplus (\End ^0 (A_{\overline{\mathbb{Q}}}\otimes \mathbb{Q}_p (1)))^{d(d-1)/2}$.
 Then it is enough to prove 
$$
\dim H^1 _f (G_{\mathfrak{p}},\NS (A_{\overline{\mathbb{Q}}})\otimes \mathbb{Q}_p (1)) -\dim H^1 _f (G_{\mathbb{Q},T} ,\NS (A_{\overline{\mathbb{Q}}})\otimes \mathbb{Q}_p (1))\geq \rho _f (A)
$$
and 
$$
\dim H^1 _f (G_{\mathfrak{p}},\End ^0 (A_{\overline{\mathbb{Q}}})\otimes \mathbb{Q}_p (1)) -\dim H^1 _f (G_{\mathbb{Q},T} ,\End ^0 (A_{\overline{\mathbb{Q}}})\otimes \mathbb{Q}_p (1))\geq e (A),
$$
which follows from Lemma \ref{artin-tate-bound}. If $K$ is imaginary quadratic, we have a surjection
$$
\wedge ^2 V\to (\NS (A_K )\otimes \mathbb{Q}_p (1) ))^d \oplus (\End ^0 (A_{K}\otimes \mathbb{Q}_p (1)))^{d(d-1)/2},
$$
and we take $W$ to be the corresponding quotient of $\overline{\wedge ^2 V}$.
The result now follows from the fact that $H^1 (G_{K,T},\mathbb{Q}_p (1))=0$.

We are now ready to complete the proof of Theorem \ref{thm1}. First suppose $\rho _f (A)d+d(d-1)e(A)/2-1 >d(r-\dim (A))$. By Lemma \ref{hyperellipticsplitting}, we have a Galois-stable quotient of $U_2$ which is an extension
$$
1\to W\to U \to V_A ^d \to 1
$$
and we have 
$$
\dim H^1 _f (G_{\mathfrak{p}},U) -\dim H^1 _f (G_{K,T} ,U) \geq \rho _f (A)d+d(d-1)e(A)/2-1 -d(r-\dim (A)).
$$
Finally, if $\rho _f (A)d+d(d-1)e(A)/2-1 >r^2 -\dim (A)^2$, then we use Proposition \ref{maintheorem}.
We take $R$, as above, to be $\Mat _d (\mathbb{Q} )$, acting trivially on $B$ and in the obvious way on $A^d $. Then 
$$
\rk (J(K)\otimes \mathbb{Q})\otimes _R (J(K)\otimes \mathbb{Q}) =r^2 \quad\textrm{and}
$$
$$
\dim H^1 _f (G_{\mathfrak{p}},V)\otimes _{R\otimes \mathbb{Q}_p }\dim H^1 _f (G_{\mathfrak{p}},V)=(\dim A)^2.
$$
\end{proof}
\subsection{An example}\label{example_shaska}
Given the restrictive hypotheses of  Theorem \ref{thm1}, it is perhaps worth demonstrating the existence of a hyperelliptic curve satisfying them which does not satisfy the Chabauty--Coleman bound. We use work of Paulhus \cite[Table 2]{paulhus:2013} and Shaska \cite[$\S 4$]{shaska:2004} on a family of hyperelliptic curves $X_t $ defined over $K=\mathbb{Q}(i)$ with Jacobian isogenous to $E_t ^3 \times A_t $. Let $X_t$ denote the genus 5 curve 
$$
y^2 = x^{12} - tx^{10} - 33x^8 + 2tx^6 - 33x^4 - tx^2 + 1.
$$
For all but a finite number of $t$, we have a subgroup of $\Aut X_t $ isomorphic to $A_4 $, generated by the automorphisms of order 2 and 3, respectively:
$$ \tau :(x,y)\mapsto (-x,y), \qquad \sigma :(x,y) \mapsto \left(  \frac{x-i}{x+i} ,\frac{y}{(x+i)^6}\right).
$$
Together with the hyperelliptic automorphism, this means that all but a finite number of curves in the family has $\mathbb{Z}/2\mathbb{Z}\times A_4 $ as a subgroup of its automorphism group.
The normalisation of the quotient of $X$ by $\sigma $ is the genus 1 curve 
$$
C: y^2 =x^4+ (-t + 12i)x^2 + ((2i - 2)t + 20i + 20)x + 2it + 21,
$$
which has Jacobian 
$$
E_t : y^2 =x^3-\frac{1}{4}(3t^2 - 70it - 411)x-\frac{1}{4}(t^3 - 30it^2 - 317t + 1180i).
$$
Fix a prime $p$ with $(p)= \mathfrak{p}\overline{\mathfrak{p}}$ in $K$ and such that $X_t$ has good reduction at $\mathfrak{p}$ and $\overline{\mathfrak{p}}$.
\begin{Corollary}
For all $t$ such that $\rk _{K}E_t \leq 2$, and $\mathfrak{p}$ as above,
$
X(K_{\mathfrak{p}})_2 
$ is finite.
\end{Corollary}
\begin{proof}
Let $V_E :=T_p (E_t )\otimes \mathbb{Q}_p $, $V_A :=T_p (A)\otimes \mathbb{Q}_p $.
We have an isomorphism
$$
\wedge ^2 V \simeq \wedge ^2 (V_E ^{\oplus 3})\oplus V_E \otimes V_A \oplus \wedge ^2 V_A .
$$
Let $\overline{\wedge ^2 V}\to \mathbb{Q}_p (1)^{\oplus 4}$ be the composite
$
\overline{\wedge ^2 V}\to (\wedge ^2 (V_E ^{\oplus 3}))/\mathbb{Q}_p (1)\to \mathbb{Q}_p (1)^{\oplus 4}.
$
Let $U$ be the corresponding quotient of $U_2 $. The result follows from Proposition \ref{maintheorem}.
\end{proof}
\begin{Remark}
Note that the dimension of the Selmer variety equals that of $H^1 _f (G_{\mathfrak{p}} ,U)$, so the multiplicities of isogeny factors are really used in an essential way.
\end{Remark}
An explicit example of a value of $t$ for which $E_t$ has rank 2 is $t=1$:
the elliptic curve $y^2 = x^3 + (35/2i+102)x + (-575/2i+79)$ has two independent points $P_1 = (4i - 3, 14i + 4), P_2 = (-12i + 1, 11i + 9)$. Using Sage \cite{sage:2017}, we verified linear independence (and a lower bound of 2 for the rank) by computing that the associated regulator of height pairings is approximately 6.501, and in particular, is nonzero.
An upper bound of 2 on the rank was found by using Magma \cite{magma:1997} to compute the rank of the 2-Selmer group to be 2.

\subsection{The Kulesz--Matera--Schost family}\label{subsec:KMS}
Here we return to the family of genus 2 curves mentioned in the introduction. We show that for this family, one can use equivariant heights to prove stronger finiteness results than the ones above.
Recall that $X$ is a hyperelliptic  curve of the form $y^2 =x^6 +ax^4 +ax^2 +1$, and let $E$ be the elliptic curve $y^2 =x^3 +ax^2 +ax+1$. We assume $E$ has rank 2. Define $V_E $ to be $H^1 (E_{\overline{\mathbb{Q}}},\mathbb{Q}_p )^* $. The morphisms $f_1 ,f_2 $ from $X$ to $E$ induce an isomorphism
and hence an isomorphism $V\simeq V_E \oplus V_E $, which induces a Galois-stable quotient $U$ of $U_2 $ with $W$ taken to be $\Sym ^2 V_E$, via the map 
\begin{equation}\label{eqn:explicit_wedge}
\wedge ^2 V \to \Sym ^2 V_E ;
(v_1 ,v_2 )\wedge (v_3 ,v_4 ) \mapsto v_1 v_4 -v_2 v_3.
\end{equation}
The aim of this subsection is to prove the following lemma:
\begin{Lemma}\label{lemma20}
The localisation map $\loc _{\mathfrak{p}} :\Sel (U)\to H^1 _f (G_{\mathfrak{p}},U)$ is not dense.
\end{Lemma}
In fact we will prove an explicit form of this.
The deep result underlying this non-density is the fact that $H^1 _f (G_{K,T} ,\Sym ^2 V_E )=0$. In the case when $K=\mathbb{Q}$, $p\geq 5$, and the map \[
\rho _{E[p]}:\Gal (\overline{\mathbb{Q}}|\mathbb{Q})\to \Aut (E[p])
\] 
is surjective, this is due to Flach \cite{flach:1992}. In general, the only known proof is via a Galois deformation argument, following Taylor--Wiles and Kisin. Namely, again using Fontaine--Perrin-Riou's Euler characteristic formula, we know that 
\[
\dim H^1 _f (G_{K,T},\Sym ^2 V_E )=\dim H^1 _f (G_{K,T},\ad ^0 V_E ).
\] 
Under the assumptions above, it is known that $H^1 _g (G_{K,T},\ad ^0 V_E )=0$ (see Allen \cite[Theorem A]{allen:2016} for a more general result).

Let $R:= \Mat _2 (\Q _p )$. Then $V$ has the structure of an $R$-module via the isomorphism $V\simeq V_E \oplus V_E $.
Let $\Gamma :\mathbb{Q}_p [H^1 _{f,T_0 } (G_{K,T} ,U)]\to \wedge ^2 H^1 _f (G_{K,T} ,V_E )$
be the composite map
\begin{align*}
& \mathbb{Q}_p [H^1 _{f,T_0 } (G_{K,T} ,U) ] \to H^1 _f (G_{K,T} ,V)\otimes _R H^1 _f (G_{K,T} ,V) \\
\stackrel{\simeq }{\longrightarrow } & H^1 _f (G_{K,T} ,V_E )\otimes _{\mathbb{Q}_p }H^1 _f (G_{K,T} ,V_E )\to \wedge ^2 H^1 _f (G_{K,T} ,V_E ),
\end{align*}
where the first map sends $P$ to the $H^1 _f (G_{K,T} ,V)\otimes _R H^1 _f (G_{K,T} ,V)$ class of $A(b)^{(P)}$, the second is the isomorphism
\begin{equation}\label{matrixmovearound}
H^1 _f (G_{K,T} ,V_E ^{\oplus 2})\otimes _{\Mat _2 (\Q _p )}H^1 _f (G_{K,T} ,V_E ^{\oplus 2}) \stackrel{\simeq }{\longrightarrow }H^1 _f (G_{K,T} ,V_E)\otimes _{\Q _p }H^1 _f (G_{K,T} ,V_E ),
\end{equation}
and the third is the usual projection of the tensor square onto the alternating product.
By Lemma \ref{hyperellipticsplitting}, and equation \ref{eqn:explicit_wedge} when $P=P(b,z)$, $\Gamma (P)$ is given by 
\[
f_{1*}\kappa (z-b)\wedge f_{2*}\kappa (z+b-D)-f_{2*}\kappa (z-b)f_{1*}(z+b-D), 
\]
which is equal to $2\kappa _E (f_1 (z)-O)\wedge \kappa _E (f_2 (z))$.
\begin{Definition}
Given $[E_1 ],[E_2 ]$ in $H^1 _f (G_{K,T} ,V)$, define $[E_1 ,E_2 ]\in M_{f,T_0 } (G_{K,T} ;\mathbb{Q}_p ,V,W)$ to be the quotient of $E_1 \otimes E_2 $ by $\wedge ^2 V\subset V\otimes V\subset E_1 \otimes E_2 $, viewed as a mixed extension with graded pieces $\mathbb{Q}_p ,V$ and $W$ via the isomorphism $V\simeq V_E ^{\oplus 2}$.
\end{Definition}
\begin{Lemma}\label{lemma21}Let $\mathfrak{p}$ be a prime above $p$.
\begin{enumerate}
\item 
The mixed extension $[E_1 ,E_2 ]$ lies in the hyperelliptic subspace, and its  image  in $H^1 _f (G_{K,T} ,V_E )^{\otimes 2}$ is given by $[E_1 ]\otimes [E_2 ]+[E_2 ]\otimes [E_1 ]$.
\item 
Let $[E_1 ]=f_{2*}\pi _* P$ and $[E_2 ]=-f_{1*}\kappa (2b-D)$. 
Let $h$ be an $R$-equivariant height. Let $\alpha \in \prod _{v\in T_0 }H^1 (G_v ,U)$ be a collection of local conditions.
Then the map
\begin{align*}
\widetilde{h}': H^1 _f (G_{K,T} ,U) & \to H^1 _s (G_{K,T},\Sym ^2 (V)(1)) \\
P &\mapsto \widetilde{h}_{\mathfrak{p}} (A(b)^{(P)})+\sum _{v\in T_0 }\widetilde{h}_v (\alpha _v )-\frac{1}{2}\widetilde{h} ([E_1 ,E_2 ])
\end{align*}
factors through $\wedge ^2 H^1 _f (G_{K,T} ,V_E )$.
\end{enumerate}
\end{Lemma}
\begin{proof}
For the first part, the image of $[E_1 ,E_2 ]$ in $H^1 (G_{K,T} ,V)\otimes H^1 (G_{K,T} ,V^* \otimes W)$ is equal to 
$([E_1 ],[E_2 ])\otimes ([E_2 ],[E_1 ])$, hence the claim follows from the explicit description of the isomorphism \ref{matrixmovearound}.
For the second part, note by Lemma \ref{hyperellipticsplitting}, $A(b)^{(P)}$ is a mixed extension of $\pi _* P$ and $\tau _* (\kappa (b-D)+\pi _* P)$. So under the decomposition 
$$
H^1 _f (G_{K,T} ,V_E )\otimes H^1 _f (G_{K,T} ,V_E )=\wedge ^2 H^1 _f (G_{K,T} ,V_E )\oplus \Sym ^2 H^1 _f (G_{K,T} ,V_E ), 
$$
the image of $A(b)^{(P)}$ in $\Sym ^2 H^1 _f (G_{K,T} ,V_E )$ is given by 
\begin{align*}
 &(f_{1*}\pi _* P)( f_{2*}(\kappa (2b-D)+\pi _* P )-(f_{2*}\pi _* P )( f_{1*}(\kappa (2b-D)+\pi _* P) \\
= &(f_{1*}\pi _* P)f_{2*}\kappa (2b-D)-(f_{2*}\pi _* P)f_{1*}(\kappa (2b-D) \\
= & -(f_{2*}\pi _* P)(f_{1*}\kappa (2b-D)),
\end{align*}
since $f_{2*}\kappa (2b-D)$ is zero. 
The image of $[E_1 ,E_2 ]$ in $\Sym ^2 H^1 _f (G_{K,T},V_E )$ is given by 
\[
-(f_{2* }\pi _* P )(f_{1*}\kappa (2b-D))-(f_{1*}\kappa (b-D))(2f_{2* }\pi _* P ).
\]
Hence the class of $A(b)^{(P)}-\frac{1}{2}[E_1, E_2 ]$ in $H^1 _f (G_{K,T},V_E )^{\otimes 2}$ lies in $\wedge ^2 H^1 _f (G_{K,T},V_E )$.
\end{proof}

We are now ready to prove an explicit form of the non-dominance result for the localisation map.

\begin{Lemma}\label{lemma22}Let $\alpha \in \prod _{v\in T_0 }H^1 (G_v ,\Sym ^2 _E )$ be a collection of local conditions.

\begin{enumerate}
\item Let $$t:\wedge ^2 H^1 _f (G_{K,T} ,V_E) \to \mathbb{Q}_p [H^1 _f (G_{K,T} ,U)]$$
be a section of $\Gamma $. Let $w $ be a generator of $\wedge ^2 H^1 _f (G_{K,T} ,V_E )$. Let $P_0 $ be any element of $\Sel (U)_{\alpha }$. Then $\loc _{\mathfrak{p}} \Sel (U)_{\alpha }$ is contained within the kernel of
\begin{align*}
 H^1 _f (G_{\mathfrak{p}} ,U) &\to \wedge ^2 H^1 _f (G_{\mathfrak{p}} ,\Sym ^2 V_E ) \\
P&\mapsto (\widetilde{h}_{\mathfrak{p}} ' (P)-\widetilde{h}_{\mathfrak{p}} ' (P_0 )) \wedge \widetilde{h}_{\mathfrak{p}} ' (t(w)).
\end{align*}
\item Let $b$ and $z_0 $ be points of $X(K)$ satisfying $\Gamma (A(b,z_0 ))\neq 0$. Then $X(K_{\mathfrak{p}})_U$ is in the kernel of 
\begin{align*}
X(K_{\mathfrak{p}} ) & \to \wedge ^2 (W_{\dr }/F^0 ) \\
z & \mapsto \widetilde{h}_{\mathfrak{p}} ' (A(b,z))\wedge \widetilde{h}_{\mathfrak{p}} ' (A(b,z_0 )).
\end{align*}
\end{enumerate}
\end{Lemma}
Note that part (1) of Lemma \ref{lemma22} implies Lemma \ref{lemma20}, since by Lemma \ref{itsonto}, the map $\widetilde{h}_{\mathfrak{p}} :H^1 _f (G_{\mathfrak{p}} ,U)\to H^1 _f (G_{\mathfrak{p}}, W)$ is onto and hence the map in part (1) of the lemma is surjective.

\begin{proof}[Proof of Lemma \ref{lemma22}]
Choose a basis $e_1 ,e_2 $ of $H^1 _f (G_{\mathfrak{p}},\Sym ^2 V_E )$. Since we have $H^1 _f (G_{K,T} ,\Sym ^2 V_E )=0$, we can define cohomology classes $\chi _1 ,\chi _2 $ in $H^1 (G_{K,T} ,\ad ^0 V_E)$ which are crystalline at all primes above $p$ other than $\mathfrak{p}$, and such that the image of $\loc _{\mathfrak{p}}(\chi _i )$ in $H^1 (G_{\mathfrak{p}},\ad ^0 V_E )/H^1 _f (G_{\mathfrak{p}},\ad ^0 V_E )$ is isomorphic to $e_i ^* $ via Tate duality. Let $$
h=(h_{\chi _1 },h_{\chi _2 }):M_{f,T_0 } (G_{K,T} ;\mathbb{Q}_p ,V,\Sym ^2 V_E )\to \Sym ^2 V_{\dr} /F^0 
$$
be the corresponding sum of heights. Let $h'$ denote the map 
$$
P\mapsto h(P)-\frac{1}{2}h([E_1 ,E_2 ])
$$
as before.
Part (1) follows from Lemma \ref{lemma21}, since that implies that the image of $\Sel (U)$ in $H^1 _f (G_{\mathfrak{p}} ,\Sym^2 V_E )$ under $h'$ has dimension at most 1. For  part (2), by assumption, $\Gamma (A(b,z_0 ))$ is a generator of $\wedge ^2 H^1 _f (G_{K,T} ,V_E )$, hence the result follows from part (1).
\end{proof}

\section{Explicit local methods}\label{sec:hodgefiltration}

The goal of this section is to provide an explicit, algorithmic description of the composite map 
$$
X(K_{\mathfrak{p}} ) \stackrel{j^{\cry } }{\longrightarrow } U^{\dr }/F^0 \stackrel{\widetilde{h}_p }{\longrightarrow } W_{\dr }/F^0 
$$
which sends a $K_{\mathfrak{p}}$-point to the generalised pre-height of $A(b,z)$. As $p$ splits completely in $K|\mathbb{Q}$, via a choice of embedding $K\hookrightarrow \mathbb{Q}_p $, $K_{\mathfrak{p}}$ is isomorphic to $\mathbb{Q}_p $, and we henceforth write $\mathbb{Q}_p $ instead of $K_{\mathfrak{p}}$.
Describing the map $j^{\cry }$ explicitly amounts to giving an explicit description of the structure of $D_{\cry }(A(b,z))$ as a filtered $\phi $-module.
By Olsson's comparison theorem \cite[Theorem 1.4]{olsson:2011}, this may be reduced to computing the Hodge filtration and Frobenius action on a de Rham path space $A^{\dr }(b,z)$ (see \cite[\S 3]{kim:2005}). The specific relation is stated in section \ref{subsec:reln}.

The filtered $\phi $-module $A^{\dr }(b,z)$ will be the pullback of this filtered $F$-isocrystal $\mathcal{A}^{\dr}$ at the $\mathbb{Z}_p $-point $z$ (in particular we choose $Y$ such that the reduction mod $p$ of $X-Y$ does not contain the reduction mod $p$ of $b$ or $z$). In this section we describe how to carry this process out explicitly: i.e., how to explicitly compute the connection $\mathcal{A}^{\dr}$ on an affine open $Y_{\mathbb{Q}_p } \subset X_{\mathbb{Q}_p }$ and enrich it with the structure of a filtered $F$-isocrystal on a smooth model of $Y$ over $\mathbb{Z}_p $. For hyperelliptic curves, using Coleman integration, we have a simple a priori description of the Frobenius structure (see Section \ref{frobenius_structure}), and hence most of this section explains the definition and computation of the connection and how to compute the filtration structure it carries.

Roughly, there are two reasons why it easier to work with the affine $Y$ than with the projective $X$. Firstly, as every extension of vector bundles on $Y$ splits, the underlying vector bundle of any unipotent connection will admit a trivialisation. This makes it much easier to write down a unipotent connection. Secondly, the de Rham fundamental group is a free pro-unipotent group, which makes it easier to write down elements of the fundamental group, or its enveloping algebra.
\subsection{The universal connection}
First we recall some properties of the de Rham fundamental group and associated objects, as developed by Chen, Deligne, Hain and Wojtkowiak (see \cite{deligne:1989, wojtkowiak:1993, hain:1987}). 

Let $Y\subset X$ be a non-empty open subset of $X$, with $X-Y$ of order $r$. Define 
$
V_{\dr }(Y)=H^1 _{\dr }(Y)^* .
$
When $Y=X$, we will write this simply as $V_{\dr }$.
Denote by $\mathcal{C}^{\dr}(Y)$ the category of unipotent flat connections on $Y$, and $\mathcal{C}^{\dr }(X)$ the category of unipotent flat connections on $X$. Given a connection $\mathcal{V}$ and a $K$-vector space $W$, we shall often refer to $W\otimes \mathcal{V}$ as a connection, in the natural way: the $\mathcal{U}$-sections of the vector bundle are just $W\otimes \mathcal{V}(\mathcal{U})$, and 
the connection morphism is $1_{W}\otimes \nabla $. Alternatively, if $\pi :Y\to \spec (K)$ denotes the structure morphism, we 
can think of $W\otimes \mathcal{V}$ as being a tensor product of connections:
$$
W\otimes \mathcal{V}:=(\pi ^* W ,d)\otimes (\mathcal{V},\nabla ).
$$

Let $b$ be a $K$-point of $Y$. Then taking the fibre of the underlying bundle at $b$ defines a fibre functor 
$
b^* 
$
from $\mathcal{C}^{\dr }(X)$ to $K$-vector spaces,
giving $(\mathcal{C}^{\dr }(X),b^* )$ the structure of a neutral Tannakian category.
Define $\pi _1 ^{\dr }(X,b)$ to be the corresponding $K$-group scheme. This group is pro-unipotent and is the inverse limit of the $n$-step unipotent quotients $U_n ^{\dr } (b)=U_n ^{\dr } (X)(b)$.  Moreover,
$$
P_n ^{\dr }(X) (b,z)=P_n ^{\dr } (b,z)=\pi _1 ^{\dr }(X;b,z)\times _{\pi _1 ^{\dr }(X,b)}U_n ^{\dr }(b).
$$ 
For example we have 
$U_1 (b)=:V_{\dr }\simeq H^1 _{\dr }(X)^* $ and an sequences
\[
1 \to \overline{\wedge ^2 V}_{\dr }\to U_2 ^{\dr }\to V_{\dr }\to 1,
\]
where $\overline{\wedge ^2 V}_{\dr }:=\coker (H^2 _{\dr }(X)\stackrel{\cup ^* }{\longrightarrow }\wedge ^2 V_{\dr }$.
Similarly define $\pi _1 ^{\dr }(Y,b)$, $U_n ^{\dr } (Y)(b)$, $P_n ^{\dr } (Y)(b,z)$. Finally as in the \'etale case, define 
$A_n ^{\dr }(X)(b)$ to be the quotient of the universal enveloping algebra of $\pi _1 ^{\dr }(X,b)$ by the $(n+1)$th power of the kernel of the co-unit map. Then $A_n ^{\dr }(X)(b)$ is a faithful $U_n ^{\dr }(b)$-representation, and we define
$$A_n ^{\dr }(X)(b,z)=A_n ^{\dr }(X)(b)\times _{U_n (b)}P_n (b,z). $$
Similarly we define $\pi _1 ^{\dr }(Y,b):=\pi _1 (\mathcal{C}^{\dr }(Y),b)$, and define associated objects $\mathcal{A}_n ^{\dr }(Y),A_n ^{\dr }(b,z)$, etc.
\subsection{The relation to $A(b,z)$}\label{subsec:reln}
Recall that the main goal of this section is to compute the generalised pre-height $\widetilde{h}_{\mathfrak{p}}(A(b,z))$.
This is related to $A^{\dr }(b,z)$, via Olsson's theorem which gives the isomorphism \eqref{eqn:olsn}:
\[
D_{\cry }(A_n (b,z)) \simeq A_n ^{\dr }(b,z).
\]
On graded pieces, this is induced by the isomorphism 
\[
H^1 _{\dr }(X_{\mathbb{Q}_p })\simeq D_{\cry }(H^1 _{\acute{e}t}(\overline{X},\mathbb{Q}_p )).
\]
Hence if we define $W_{\dr }:=D_{\dr }(W)$, and 
\[
A^{\dr }(b,z):=A_2 ^{\dr }(b,z)/\Ker (\overline{\wedge ^2 V}_{\dr }\to W_{\dr }),
\]
then we obtain an isomorphism of filtered $\phi $-modules 
\[
D_{\cry }(A(b,z))\simeq A^{\dr }(b,z).
\]

\subsection{The Hodge filtration}
The vector spaces $A_n ^{\dr }(X)(b)$ have canonical Hodge filtrations, which we now explain.
\begin{Definition}
By a filtered connection $\mathcal{V}=(\mathcal{V},\nabla ,F^{\bullet })$ we shall mean a vector bundle $\mathcal{V}$ together with a flat connection $\nabla $ and a decreasing, exhaustive, separating filtration $(F^i \mathcal{V})$ by sub-bundles, satisfying the \textit{Griffiths transversality} condition
$$
\nabla (F^i \mathcal{V})\subset \Omega ^1 \otimes F^{i-1}\mathcal{V}
$$
for all $i$. We similarly define a filtered connection with log singularities. We sometimes write a filtered connection as $(\mathcal{V},F^\bullet )$ and sometimes 
simply as $\mathcal{V}$.
\end{Definition}
There are various characterisations of the Hodge filtration. The one which seems to be the most useful computationally is Hadian's characterisation of the canonical extension of $\mathcal{A}^{\dr }_n (Y)$. 
\begin{Definition}\label{canonical}
Given a unipotent connection $\mathcal{V}$ on $Y$, we shall denote by $\mathcal{V}^{\can} $ the canonical extension of $\mathcal{V}$ to a connection on $X$ with log singularities along $D$, which exists and is functorial in $\mathcal{V}$ by Deligne \cite{deligne:1970}.
\end{Definition}

\begin{Proposition}[Hadian {\cite[Proposition 3.3]{hadian:2011}}]
Let $\mathcal{E}$ and $\mathcal{F}$ be filtered connections on $X$ with logarithmic singularities along $D$. Then the group of isomorphism
classes of extensions of $\mathcal{E}$ by $\mathcal{F}$ (in the category of filtered flat connections on $X$ with 
logarithmic singularities along $D$) is isomorphic to the first hypercohomology group of the complex
$$
F^0 (\mathcal{E}^* \otimes \mathcal{F}) \stackrel{\nabla }{\longrightarrow }\Omega ^1 \otimes F^{-1} (\mathcal{E}^* \otimes \mathcal{F})
$$
where $\nabla $ denotes the associated connection on the internal Hom bundle $\mathcal{E}^* \otimes \mathcal{F}$.
\end{Proposition}
By computing these hypercohomology groups in the case $\mathcal{E}=\mathcal{A}_{n-1}^{\dr }(Y)^{\can }$ and $\mathcal{F}=V_{\dr }(Y) ^{\otimes n}\otimes \mathcal{O}_X$, Hadian
 proves the following lemma (note that in \cite{hadian:2011}, $(X,Y,\mathcal{A}_n ^{\dr }(Y)^{\can },V_{\dr }(Y))$ is written as $(C,X,\mathcal{P}_n ^{\dr },T_{\dr }))$.
\begin{Lemma}[Hadian {\cite[Lemma $2.2.6$]{hadian:2011}}]\label{hadianhodge}
 There exists a filtration $(F^i \mathcal{A}^{\dr}_n (Y)^{\can }$ of vector bundles such that \\
 (i) For all $n$, the sequence of connections
 \begin{equation}\nonumber
  0\to \mathcal{O}_X \otimes V_{\dr }(Y) ^{\otimes n } \to \mathcal{A}_n ^{\dr }(Y)^{\can } \to \mathcal{A}_n ^{\dr}(Y)^{\can } \to 0
 \end{equation}
respects the filtrations, where $\mathcal{O}_X \otimes V_{\dr }(Y) ^{\otimes n }$ is given the filtration induced by the Hodge filtration on $V_{\dr }(Y) ^{\otimes n}$. \\
(ii) For all $n$, the filtration $F^i $ satisfies Griffiths transversality, and hence gives $\mathcal{A}_n ^{\dr }(Y)^{\can }$ the structure of a filtered connection for all $n$. \\
(iii) The filtration $F^i $ is unique up to isomorphism of filtered connections.
\end{Lemma}

\begin{Remark}
It is easy to see that the analogous theorem for the bundle $\mathcal{A}_n ^{\dr} (Y)$ on $Y$ (when $Y\neq X$) is false: since the category of unipotent vector bundles on $Y$ is trivial, 
there will be many ways to lift the Hodge filtration on the graded pieces and satisfy Griffiths transversality.
Hence the content of computing the Hodge filtration on the $\mathcal{A}_n ^{\dr }(Y)$ is contained in computing its canonical extension to $X$.
\end{Remark}
\begin{Remark}
The statement of the lemma is somewhat weaker than the statement given in \cite{hadian:2011}. In loc. cit. the author states that the filtration is unique (without allowing
for automorphisms). This is deduced by inductively determining from the computation that the map
$$
\ext ^1 _{\dr }(\mathcal{A}_{n-1}^{\dr }(Y)^{\can },V_{\dr }(Y) ^{\otimes n}\otimes \mathcal{O}_X) \to \ext ^1 _{\dr ,\fil }((\mathcal{A}_{n-1}^{\dr }(Y)^{\can },F^\bullet \mathcal{A}_{n-1}^{\dr }(Y)^{\can }),(V_{\dr }(Y) ^{\otimes n}\otimes \mathcal{O}_X ,F^\bullet V_{\dr }(Y) ^{\otimes n}\otimes \mathcal{O}_X ))
$$
is injective. However, this only implies that there is a unique \textit{extension class} of filtered connections $[(\mathcal{A}_n ^{\dr }(Y) ^{\can } ,F^\bullet )]$ corresponding to the extension class $[\mathcal{A}_n ^{\dr }(Y)^{\can }  ]$. 
To obtain uniqueness of the filtration itself, one must rigidify by imposing conditions on the filtration at the basepoint (this is already true when $n=1$). Needless to say this distinction is not important in the context of Hadian's paper and does not affect the main results.
\end{Remark}
For our purposes, we will be interested in a mild generalisation of this, where instead of considering $\mathcal{A}_n ^{\dr}(Y)$ we consider sheaves coming from other quotients of the universal enveloping algebra. In the following corollary, we let $W$ be any filtered quotient of $V_{\dr }(Y) ^{\otimes n}$, and let $\mathcal{B}$ be the corresponding quotient of the connection $\mathcal{A}_n ^{\dr}(Y)$. 
Hence the map $
\mathcal{A}_n ^{\dr}(Y)\to \mathcal{A}_{n-1}^{\dr }(Y)
$
factors through 
$
\mathcal{A}_n ^{\dr }(Y) \to \mathcal{B},
$
and $\mathcal{B}$ is an extension
$$
0 \to W \otimes \mathcal{O}_Y \to \mathcal{B}\to \mathcal{A}_{n-1} ^{\dr }(Y).
$$
\begin{Corollary}\label{correspondingquotient}
There is a unique lift of the filtrations on $\mathcal{A}_{n-1}^{\dr }(Y)^{\can } $ and $W\otimes \mathcal{O}_X $ to a filtered connection structure on $\mathcal{B}^{\can } $ such that in the fibre 
at $b$, $1$ lies in $b^* F^0 \mathcal{B}$.
\end{Corollary}
\begin{proof}
The category of filtered $K$-vector spaces is semisimple, so $W$ admits a filtered complement
$
V_{\dr }(Y)^{\otimes n} \simeq W \oplus W' .
$
Hence 
$$
\ext ^1 _{\dr , \fil }(V_{\dr }(Y) ^{\otimes n}\otimes \mathcal{O}_X ,\mathcal{A}_{n-1}^{\dr}(Y)^{\can }) \simeq \ext ^1 _{\dr , \fil }(W\otimes \mathcal{O}_X ,\mathcal{A}_{n-1}^{\dr }(Y)^{\can })  
\oplus \ext ^1 _{ \dr , \fil }(W' \otimes \mathcal{O}_X ,\mathcal{A}_{n-1}^{\dr}(Y)^{\can }) 
$$
and 
$$
\ext ^1 _{\dr }(V_{\dr }(Y) ^{\otimes n}\otimes \mathcal{O}_X ,\mathcal{A}_{n-1}^{\dr }(Y)^{\can }) \simeq \ext ^1 _{\dr }(W\otimes \mathcal{O}_X ,\mathcal{A}_{n-1}^{\dr }(Y)^{\can })  
\oplus \ext ^1 _{\dr }(W' \otimes \mathcal{O}_X ,\mathcal{A}_{n-1}^{\dr }(Y)^{\can }).
$$
Therefore uniqueness of the lift of the filtration on $\mathcal{A}_{n-1}^{\dr }(Y)^{\can }$ to $\mathcal{A}_n ^{\dr }(Y) ^{\can }$ given conditions on $b^* \mathcal{A}_n ^{\dr }(Y)$ implies uniqueness 
of the lift of the filtration on $\mathcal{A}_{n-1}^{\dr }(Y){\can }$ to $\mathcal{B}^{\can }$ given conditions on $b^* \mathcal{B}$.
\end{proof}

To compute the Hodge filtration on $\mathcal{A}_n ^{\dr} (X)$ (i.e. to carry the above out for a projective curve), we may compute the Hodge filtration on the universal connection of an open affine $Y$, and 
then take the quotient to get the Hodge filtration on the universal connection on the projective curve $X$. This will be explained in more detail in the next section.

\subsection{Universal pointed objects}
To describe the universal connection $\mathcal{A}_n ^{\dr}(X)$,  we recall a ubiquitous construction in the study of unipotent fundamental groups (see e.g.  Andreatta, Iovita, and Kim \cite[\S 3.5]{andreatta:2015}).
Given a neutral Tannakian category $(\mathcal{C},\omega )$ over a field $K$, define the \textit{pointed category} to be the category whose objects are pairs $(V,v)$, where 
$v$ is an element of $\omega (V)$, and whose morphisms are $(V,v)\to (W,w)$,  i.e., morphisms $f:V\to W$ in $\mathcal{C}$ such that $\omega (f)(v)=w$. We say that an inverse limit of pointed objects $(V_n ,v_n )$ is \textit{universal} if for every pointed object $(W,w)$, there is an $n_0 >0$ such that for all $n>n_0$, there is a unique homomorphism $(V_n ,v_n )\to (W,w)$. Note that if $((V_n ,v_n ))$ and $((W_n ,w_n ))$ are two universal pointed objects, then there is a unique pro-morphism between them.

Now suppose that $(\mathcal{C},\omega )$ is a unipotent neutral Tannakian category over $K$ (i.e. one for which every object $V$ admits an inclusion of the trivial object). Without loss of generality we may simply assume that $\mathcal{C}$ is the category of representations of a pro-unipotent group $\mathcal{U}$ over $K$. Suppose furthermore that $\mathcal{C}$ has finite-dimensional ext groups (i.e., the abelianisation of $\mathcal{U}$ is finite-dimensional) . Let $A_\infty (\mathcal{C},\omega )$ denote the pro-universal enveloping algebra of the Lie algebra of $\mathcal{U}$. 

Let $I\subset A_\infty (\mathcal{C},\omega )$ denote the kernel of the co-unit map $A_\infty (\mathcal{C},\omega ) \to K$, and define $A_n (\mathcal{C},\omega )=A_\infty (\mathcal{C},\omega )/I^{n+1}$. Then each $A_n (\mathcal{C},\omega )$ is an object of $\mathcal{C}$, and $((A_n (\mathcal{C},\omega ),1))_n$ is a universal pointed object. In particular, we see that $(\mathcal{A}_n ^{\dr }(X),1)$ and $(\mathcal{A}_n ^{\dr }(Y),1)$ are universal $n$-unipotent objects in $(\mathcal{C}^{\dr }(X),b^* )$ and $(\mathcal{C}^{\dr }(Y),b^* )$ respectively (here $1$ denotes the identity element in the algebras $A_n ^{\dr }(X)(b)$ and $A_n ^{\dr }(Y)(b)$ respectively. Going the other way, this means that in order to compute the enveloping algebra, it is enough to construct a universal pointed object. In particular, to compute $\mathcal{A}_n ^{\dr }(X)$ or $\mathcal{A}_n ^{\dr }(Y)$, it is enough to compute a universal pointed object.
\begin{Definition}\label{defn:YsubX}
Let $Y\subset X$ be an affine open over $K$. For simplicity we assume that all the points of $X-Y$ are defined over $K$. Choose $\eta _0 ,\ldots ,\eta _{2g+r-2} \in H^0 (Y,\Omega ^1 )$ a set of differentials whose image in $H^1 _{\dr }(Y)$ forms a basis. We will henceforth assume that this basis is chosen such that $\eta _0 ,\ldots ,\eta _{g-1}$ is a basis of $H^0 (X,\Omega ^1 )$, and $\eta _g ,\ldots ,\eta _{2g-1}$ form a basis of $H^1 _{\dr }(X)$. Let $R=\oplus _{i \geq 0}V_{\dr }(Y)^{\otimes i}$ be the tensor algebra of $V_{\dr }(Y)$. Hence $R$ may also be thought of as the free associative $K$ algebra on $2g+r-1$ generators 
$T_0 ,\ldots ,T_{2g+r-2}$, where the $T_i $ are the dual basis to the $\eta _i $. 
Define $R_n $ to be the quotient of $R$ by the 2-sided ideal generated by $V_{\dr }(Y)^{\otimes (n+1)}$. Let $\mathcal{A}_n (Y) :=R_n \otimes \mathcal{O}_Y $ be the corresponding trivial vector bundle, and define a connection $\nabla _n $ on $\mathcal{A}_n (Y)$:
\[\nabla :  R_n \otimes \mathcal{O}_Y \to R_n \otimes \Omega ^1 _Y ;
 w\otimes 1 \mapsto -\sum _{i=0}^{2g+r-2} T_i w \otimes \eta _i.\]\end{Definition}
The following theorem of Kim says that $(\mathcal{A}_n (Y),1)$ is a universal pointed object in $(\mathcal{C}^{\dr }(Y),1)$, and hence
$
(\mathcal{A}_n (Y),1)\simeq (\mathcal{A}_n ^{\dr }(Y),1).
$
\begin{Theorem}[Kim {\cite[Lemma 3]{kim:2009}}]\label{universal_connection}
For every $n$-unipotent pointed connection $(\mathcal{V},v)$ there is a unique map $(\mathcal{A}_n (Y),1)\mapsto (\mathcal{V},v)$.
\end{Theorem}
We shall refer to the bundle isomorphism $\mathcal{A}_n ^{\dr }(Y)\simeq \oplus _{i=0}^n V_{\dr }(Y)^{\otimes i}\otimes \mathcal{O}_Y$, and the induced vector space isomorphism $A_n ^{\dr }(Y)(b,z)$ $\simeq \oplus _{i=0}^n V_{\dr }(Y)^{\otimes i}$ as the \textit{affine trivialisation} of $\mathcal{A}_n $ (relative to the basis $(\eta _i )$). 
\subsection{Computation of the Hodge filtration}\label{computehodge}
We now explain how to use this to algorithmically determine the Hodge filtration on the $\mathbb{Q}_p$-vector spaces $A(b,z)$. Unlike the computation of the Frobenius structure, this requires no particular ingenuity, as results of Kim and Hadian reduce the problem to elementary calculations in computational algebraic geometry. 

As in the \'etale case, $U_2 ^{\dr }$ is an extension of $V_{\dr }:= H^1 _{\dr }(X)^* $ by 
$
\overline{\wedge ^2 V}_{\dr }:=\coker ( H^2 _{\dr }(X)^*  \stackrel{\cup ^* }{\longrightarrow }  \wedge ^2 V_{\dr } ).
$
We take as input a filtered quotient $W_{\dr }$ of $\overline{\wedge ^2 V}_{\dr }$, giving a quotient $U_{\dr }$ of $U_2 ^{\dr }$ sitting in an exact sequence 
$$
1\to W_{\dr }\to U_{\dr }\to V_{\dr }\to 1.
$$
Recall that in practice we take $W_{\dr }:=D_{\dr }(W)$, where $W$ is as in section \ref{subsec:twist}.
We also fix an open affine $Y\subset X$ such that $D=X-Y$ has $|D(K)|=r$. 
\begin{Definition}
\ref{universal_connection}
We denote by $\tau $ the multiplication map
$$
V_{\dr }(Y) \otimes V_{\dr }(Y)\to W_{\dr }.
$$
Let $S_1 ,\ldots ,S_{d}$ be a basis of $W_{\dr }$, and define $\tau _{ijk}, 0\leq i,j \leq 2g+r-2,1\leq k\leq d,$ by 
$$
\tau (T_i \otimes T_j )=\sum _{k=1}^d \tau _{ijk}S_k .
$$
By definition this map factors through $V_{\dr }\otimes V_{\dr }$, and hence by our choice of basis differentials, $\tau _{ijk}$ is zero whenever $i$ or $j$ are greater than $2g-1$. Note that the condition that the map factors through $\overline{\wedge ^2 V }_{\dr }$ is equivalent to the equations
\begin{align}
\tau _{ijk}+\tau _{jik}=0, & 0\leq i,j\leq 2g-1,1\leq k \leq  d. \notag \\
\sum _{0\leq i<j\leq 2g-1 }[\eta _i ]\cup [\eta _j ]\tau _{ijk}=0, & 1\leq k\leq d. \label{eq:cupzero}
\end{align}
\end{Definition}
Corresponding to $W_{\dr }$, we have a quotient $\mathcal{A}^{\dr }(Y)$ of $\mathcal{A}_2 ^{\dr }(Y)$ which is an extension
$$
0 \to W_{\dr }\otimes \mathcal{O}_Y \to \mathcal{A}^{\dr }(Y)\to \mathcal{A}_1 ^{\dr }(Y)\to 0.
$$
By Theorem \ref{universal_connection}, the connection on $\mathcal{A}^{\dr }(Y)$ is given as follows:
$$1 \mapsto -\sum _{i=0}^{2g+r-2} \eta _i \otimes T_i, \qquad T_i  \mapsto +\sum _{j=0}^{2g+r-2}\sum _{k=1}^d \tau _{ijk}\eta _j \otimes S_k, \qquad S_k \mapsto 0.$$
The Hodge filtration on $\mathcal{A}^{\dr }(X)$ is computed in two stages:
\begin{enumerate}
\item Compute the maximal quotient $\mathcal{A}^{\dr }(X)|_Y$ of $\mathcal{A}^{\dr }(Y)$ whose canonical extension to $X$ defines a connection without singularities.
\item Compute the Hodge filtration on $\mathcal{A}^{\dr }(X)$.
\end{enumerate}
\subsubsection{Computing $\mathcal{A}^{\dr }(X)$}\label{following_properties}
\begin{Lemma}
The connection $\mathcal{A}^{\dr }(X)|_Y$ is the maximal quotient of $\mathcal{A}^{\dr }(Y)$ which extends to a connection on $X$ without log singularities.
\end{Lemma}
\begin{proof}
By definition $\mathcal{A}^{\dr }(X)|_Y$ extends to $X$. By Tannaka duality, the claim is equivalent to the saying that $A^{\dr }(b)$ is the maximal quotient of $A^{\dr }(Y)(b)$ for which the action of $\pi _1 ^{\dr }(Y,b)$ factors through $\pi _1 ^{\dr }(X,b)$.
 Passing to enveloping algebras, this is equivalent to the action of $A^{\dr }_2 (Y)(b)$ factoring through $A^{\dr }_2 (X)(b)$, which implies the Lemma.
\end{proof}
We deduce that $\mathcal{A}^{\dr }(X)|_Y $ is the unique quotient of $\mathcal{A}^{\dr }(Y)$ which extends to a connection on the whole of $X$ without log singularities and fits in a commutative diagram with exact rows
$$
\begin{tikzpicture}
\matrix (m) [matrix of math nodes, row sep=3em,
column sep=3em, text height=1.5ex, text depth=0.25ex]
{0 & W_{\dr }\otimes \mathcal{O}_Y & \mathcal{A}^{\dr }(Y) & \mathcal{A}_1 ^{\dr }(Y) & 0 \\
0 & W_{\dr }\otimes \mathcal{O}_Y & \mathcal{A}^{\dr }(X)|_Y & \mathcal{A}_1 ^{\dr }(X)|_Y & 0. \\};
\path[->]
(m-1-1) edge[auto] node[auto]{} (m-1-2)
(m-1-2) edge[auto] node[auto] {$=$} (m-2-2)
edge[auto] node[auto] {} (m-1-3)
(m-1-3) edge[auto] node[auto] {} (m-2-3)
edge[auto] node[auto] {} (m-1-4)
(m-1-4) edge[auto] node[auto] {} (m-2-4)
edge[auto] node[auto] {} (m-1-5)
(m-2-1) edge[auto] node[auto]{} (m-2-2)
(m-2-2) edge[auto] node[auto]{} (m-2-3)
(m-2-3) edge[auto] node[auto]{} (m-2-4)
(m-2-4) edge[auto] node[auto]{} (m-2-5);
\end{tikzpicture}
$$
Let $C(Y/X)\subset \langle \eta _0 ,\ldots ,\eta _{2g+r-1} \rangle $ be the subspace of $H^0 (Y,\Omega )$ spanned by the differentials $\eta _{2g},\ldots ,\eta _{2g+r-2}$. We will show that there are unique $\xi _1 ,\ldots ,\xi _d $ in $C(Y/X)$ such that 
\begin{equation}\label{eqn:AdR}
1  \mapsto -\sum _{i=0}^{2g-1} \eta _i T_i -\sum _{k=1}^d \xi _k S_k, \qquad
T_i  \mapsto +\sum _{0\leq j\leq 2g-1,1\leq k\leq d} \tau _{ijk}\eta _j \otimes S_k, \qquad
S_i  \mapsto 0  
\end{equation}
defines a flat connection on $X$, and give an algorithm for finding them. 
We solve for the $\xi _k $ by computing the canonical extension for a general choice of $\xi _k $, and working out the condition for this extension to have no singularities.

For each $x\in D(K)$, let $t_x \in K(X)$ be a parameter at $x$. Let $U_x$ be a Zariski neighbourhood of $x$ such that $t_x$ has no poles on $U_x $ and $U_x \cap D=x$.
To compute the canonical extension of $\mathcal{A}^{\dr }(Y)$, one has to find, for each $x\in (X-Y)(K)$, connections $(\mathcal{A}_x ,\nabla _x )$ on $U_x$, with log singularities along $x$, and isomorphisms of connections 
$$
\psi _x :\mathcal{A}_x |_{U_x \cap Y}\stackrel{\simeq }{\longrightarrow }\mathcal{A}^{\dr }|_{U_x \cap Y}.
$$
We take $\mathcal{A}_x $ to be a trivial bundle with sections $1_x ,T_{i,x} (0\leq i\leq 2g+r-2),S_{k,x} (1\leq k\leq d)$, and solve for isomorphisms $\psi _x $ of the form 
\begin{equation}\label{eqn:fix}
1_x \mapsto 1+\sum _{i=0}^{2g+r-2} f_{i,x}T_i+\sum _{k=1}^d g_{k,x}S_k, \qquad T_{i,x} \mapsto T_i +\sum _{k=1}^d  h_{ik,x}S_k, \qquad S_{k,x}  \mapsto S_k.
\end{equation}
Since $\psi _x $ is an isomorphism of connections we deduce that the connection $\nabla _x $ is given by 
\begin{align*}
1_x & \mapsto -\sum _{i=0}^{2g-1} (\eta _i -df_{i,x})T_i,x +\sum _{k=1}^d (dg_k -\sum _{i=0}^{2g-1} (df_i -\eta _i )h_{ik,x}-\sum _{j=0}^{2g-1} \tau _{ijk}f _i \eta _j ) \otimes S_{k,x}\\
T_{i,x} & \mapsto \sum _{k=1}^d (dh_{ik,x}+\sum _{j=0}^{2g-1} \tau _{ijk}\omega _j )S_k.
\end{align*}
Hence the condition on the $f_{i,x},g_{k,x}$ and $h_{ik,x}$ is exactly that this connection has no poles of order bigger than one on $K[\! [t_x ]\! ]$.
This reduces the problem of finding $f_{i,x},g_{k,x}$ and $h_{ik,x}$ by computing the $t_x $-adic expansion of the $\omega _i $ to sufficient accuracy. Specifically let $S$ be the section of 
$$
K(\! (t)\! )\to K(\! (t)\! )/t^{-1}K[\! [t]\! ]
$$
defined by sending the equivalence class of $\sum _i a_i t^i $ to $\sum _{i\leq -2}a_i t^i $.
Let $I$ be the formal integration function 
$$
\begin{array}{ccc}
 I: \oplus _{i<-1}K.t^{i} \to \oplus _{i<0}K.t^{i}; & & \sum a_i t^i \mapsto \sum \frac{a_i}{i+1}t^{i+1}.
\end{array}
$$ 
For a global function $f\in K(X)$ or differential $\omega \in \Omega ^1 _{K(X)|K}$, let $\loc _{x}(f)$ or $\loc _{x}(\omega )$ denote its image in $K(\! (t_x )\! )$ or $K(\! (t_x )\! )dt_x $ respectively. Then the $f_{i,x},g_{k,x}$ and $h_{ik,x}$ defined in \eqref{eqn:fix} satisfy
\begin{align*}
f_{i,x} &=I\circ S (\loc _x (\eta _i ) ),\; h_{ik,x} = \sum _{0\leq j\leq 2g-1} \tau _{ijk}f_{j,x} . \\
g_{k,x} &= -I\circ S\left(\sum _i (df_{i,x} -\loc _x (\eta _i ) )h_{ik,x}-\sum _{0\leq j\leq 2g-1} \tau _{ijk}f _{i,x} \loc _x (\eta _j ) -\loc _x (\xi _k ) \right).
\end{align*}
This determines the connection on $X$ with log singularities $\mathcal{A}^{\dr }(Y)^{\can }$. We now determine the connection without log singularities $(\mathcal{A}^{\dr }(X),d+\Lambda )$. Since we are looking for a quotient of $\mathcal{A}^{\dr}(Y)^{\can }$ of the form \eqref{eqn:AdR}, the condition that $d+\Lambda $ extends to a connection without log singularities is exactly the condition that one can choose $\xi _k $ such that, for all $x$,
\[
\res \left(\sum _{i=0}^{2g-1} (df_{i,x} -\loc _x (\eta _i ) h_{ik,x}-\sum _{j=0}^{2g-1} \tau _{ijk}f _{i,x} \loc _x (\eta _j ) -\loc _x (\xi _k ) \right) =0.
\]
By the exact sequence $$0\to C(Y/X)\stackrel{\oplus _x \res _x}{\longrightarrow }\oplus _{x\in (X-Y)(K)} K \stackrel{\sum }{\longrightarrow }K \to 0,$$
such $\xi _k $ exist if and only if 
\[
\sum _{x\in D(K)} \res \left(\sum _{i=0}^{2g-1} (df_{i,x} -\loc _x (\eta _i ) )h_{ik,x}+\sum _{j=0}^{2g-1} \tau _{ijk}f _{i,x} \loc _x (\eta _j ) \right) =0.
\]
Since $\sum _{x\in D(K)} \res _x (f_{i,x}\loc _x (\eta _j ))=[\eta _i ]\cup [\eta _j ]$ by Serre's cup product formula, we can solve for $\xi _k $ by \eqref{eq:cupzero}. Explicitly, the residue of $\xi _k$ is given by
\begin{equation}\label{eqn:explicit_residue}
dg_{k,x} -\sum _{i=0}^{2g-1} (df_{i,x} -\eta _i )h_{ik,x}-\sum _{j=0}^{2g-1} \tau _{ijk}f_{j,x} \eta _k .
\end{equation}

By inspection, in order to compute these functions in practice, one simply needs to determine constants 
$\beta (i,j,x)\in K$ $(0\leq i\leq 2g+r-1,-m\leq m)$  having the property that 
$$
\loc _x (\eta _i ) -\sum _{j=-m}^{m-1} \beta (i,j,x)t_x ^i dt_x \in t_x ^m K[t_x ]dt_x ,
$$
where $m$ is the maximum over all $i$ and $x$ of the order of the pole of $\eta _i $ at $x$.
\subsubsection{Computing the Hodge filtration}
To explain how to compute the Hodge filtration, we recall some elementary properties of differentials on curves. 
\begin{Lemma}
Suppose there is a function $g\in H^0 (Y,\mathcal{O}_Y )$ and constants $\mu _i $, $g\leq i<2g$, such that for all $x\in D(K)$, $g-\sum \mu _i f_{i,x}$ has no pole at $x$. Then $g$ is constant and all the $\mu _i $ are zero.
\end{Lemma}
\begin{proof}
For $g$ and $\mu _i $ as in the lemma, we have that $dg-\sum _{i=g}^{2g-1} \mu _i \eta _i $ has no poles, hence defines an element $\sum _{i=0}^{g-1}\mu _i \eta _i $ of $H^0 (X,\Omega ^1 )$. Since 
$[\eta _0 ],\ldots ,[\eta _{2g-1}]$ is a basis of $H^1 _{\dr }(X)$, the lemma follows.
\end{proof}
It follows that given any tuple $(w_x )_{x\in D(K)}\in \prod _{x\in D(K)}K(\! (t_x )\! )$, there is a unique choice of $g\in H^0 (Y,\mathcal{O})$ and $\mu _i \in K$ ($ g\leq i<2g$) such that $g(b)=0$ and for all $x$ in $D(K)$, $w_x -\loc  _x (g) -\sum _{i=g}^{2g-1} \mu _i f_{i,x}$  does not have a pole at $x$.
\begin{Definition}
As above, let $r$ denote the degree of $D$ over $K$, and let $m$ denote the maximum over all $x\in D(K)$ and $0\leq i<2g$ of the order of the pole of $\eta _i $ at $x$. Denote by $\Pi$ the $rm$-dimensional $K$-vector space $$\prod _x t_x ^{-m}K[t_x ]dt_x /K[t_x ]dt_x .$$ Define functions 
$
r:\Pi \to H^0 (Y, \mathcal{O})
$
and 
$
\underline{c}=(c_0 ,\ldots c_{2g+r-1}):\Pi \to K^{\oplus (2g+r)}
$
by the property that for all $\pi $ in $\Pi $,
\begin{equation}\label{defofc&r}
\pi \equiv \loc _x (r(\pi ) )+\sum c_i (\pi )\loc _x (\omega _i ) \mod \prod t_x ^{-1} K[t_x ^{-1}]
\end{equation}
and $r(\pi )(b)=0$.
\end{Definition}
By Lemma \ref{hadianhodge}, $F^0 \mathcal{A}^{\dr }$ is uniquely determined by the following properties:
\begin{itemize}
\item 
There is a commutative diagram of bundles
$$
\begin{tikzpicture}
\matrix (m) [matrix of math nodes, row sep=3em,
column sep=3em, text height=1.5ex, text depth=0.25ex]
{ F^0 \mathcal{I}^2 \mathcal{A}^{\dr }(X) & F^0 \mathcal{I}\mathcal{A}^{\dr }(X) & F^0 \mathcal{A}^{\dr }(X)  \\
W_{\dr }\otimes \mathcal{O}_X & \mathcal{IA}^{\dr }(X) & \mathcal{A}^{\dr }(X) \\ };
\path[right hook->]
(m-1-1) edge (m-1-2)
(m-1-1) edge (m-2-1)
(m-1-2) edge (m-1-3)
(m-2-1) edge (m-2-2)
(m-1-2) edge (m-2-2)
(m-1-3) edge (m-2-3)
(m-2-2) edge (m-2-3)
;
\end{tikzpicture} 
$$ 
where $\mathcal{IA}^{\dr }(X)$ is the kernel of the surjective map of connections \[
\mathcal{A}^{\dr }(X)\to (\mathcal{O}_X ,d).
\]
Passing to the associated map of gradeds defines an isomorphism 
$$
\gr F^0 \mathcal{A}^{\dr } \simeq \mathcal{O}_X \oplus F^0 V_{\dr }\otimes \mathcal{O}_X \oplus F^0 W_{\dr }\otimes \mathcal{O}_X .  
$$
\item In the fibre at $b$, $1\in A_{\dr }(b)$ is in the image of $b^* F^0 \mathcal{A}^{\dr }$.
\end{itemize}
An elementary calculation shows us that $H^0 (Y ,F^0 \mathcal{A}_1 ^{\dr }(X) )$ has basis of sections $1,T_g ,\ldots ,T_{2g-1}$. To compute $F^0 \mathcal{A}^{\dr }$, we need to lift these to determine the bundle $F^0 \mathcal{A}^{\dr }$. Suppose they lift to sections $1+\sum _{k=1}^d r_k ^H \otimes S_k$ and $T_i +\sum _{k=1}^d c_{ik}\otimes S_k $. Then by the above computation of the charts defining the bundle $\mathcal{A}^{\dr }$, we find
\begin{Lemma}\label{extM1}
The functions $r_k ^H$ are given by 
$r_k ^H =r((g_{k,x})_x ) $.
The functions $c_{ik}^H$ are constant and are given by
$c_{ik}^H =c_i ((g_{k,x})_x )$.
\end{Lemma}
\begin{proof}
We need to check that the sub-bundle of $\mathcal{A}^{\dr }|_Y $ spanned by $1+\sum _{k=1}^d r_k ^H S_k $, $T_i +\sum _{k=1}^d c_{ik}^H S_k$ $(g\leq i<2g)$, and $S_k$ $(d_0 \leq k\leq d)$ extends to a sub-bundle of $\mathcal{A}^{\dr }$. Via the charts $\psi _x $, the corresponding sections of $\mathcal{A}^{\dr }_x |_{U_x -x}$ are given by 
\begin{align*}
&\psi _x ^{-1}(1+\sum _{k=1}^d r_k ^H S_k )=1_x-\sum _{i=g}^{2g-1}f _{i,x} T_{i,x} +\sum _{k=1}^d ( r_k ^H - g_{k,x})S_{k,x} ,\\
&\psi _x ^{-1}(T_i +\sum _{k=1}^d c_{ik}^H S_k )=X_{i,x} +\sum _{k=1}^d (c_{ik}^H -\sum _{j=0}^{2g-1} \tau  _{ijk}f_{j,x})S_{k,x} ,\\
&\psi _x ^{-1}(S_k )=S_{k,x}.
\end{align*}
For this $\mathcal{O}(U_x -x)$-module to be the localisation of an $\mathcal{O}(U_x )$-module, it is sufficient that there are functions $\theta _{i,x} (g\leq i<2g)$ and $\chi _{k,x}$ ($1\leq k \leq d$) in $H^0 (U_x -x,\mathcal{O})$ such that 
$$
\psi _x ^{-1}(1+\sum _{k=1}^d r_k ^H S_k )+\sum _{i=0}^{2g-1} \theta _{i,x}(T_i +\sum _{k=1}^d  c^H _{ik}S_k )+\sum _{k=1}^d \chi _{k,x}\phi _x ^{-1}S_k \in H^0 (U_x ,\mathcal{A} ^{\dr }(X)).
$$
By examining $T_i$-coordinates, we find that $\theta _{i,x}\equiv f_{i,x} \mod H^0 (U_x ,\mathcal{O})$. For $k>d_0 $ we take $\theta _k = g_{k,x}$. Hence the only nontrivial condition on the $r_k ^H$ and $c_{ik}^H$ is that for $d-d_0 \leq k\leq d_0 $,
$$
g_{k,x}-\sum _{i=0}^{2g-1} c_{ik}^H f_{i,x}-r_k ^H \in H^0 (U_x ,\mathcal{O}),
$$
for all $x$,
which hold by definition of the functions $r$ and $\underline{c}$.
\end{proof}
\subsection{The universal connection of a hyperelliptic curve}\label{5.5}
In this subsection we use the hyperelliptic splitting to provide a simple description of the Hodge filtration on $A ^{\dr }(b,z)$ when $X$ is hyperelliptic. 
In general, given an automorphism $\sigma $ of $X$, fixing the point $b$, by the universal property of $\mathcal{A}_n ^{\dr }$, we obtain a unique morphism 
$
\mathcal{A}_n ^{\dr } \to \sigma  ^* \mathcal{A}_n ^{\dr }
$
sending $1$ to $\sigma ^* 1$. The connection $\sigma ^* \mathcal{A}_n ^{\dr }$ is in a natural way isomorphic to $\mathcal{A}_n ^{\dr }$. If $\sigma (Y)=Y$, then it will also be the case that $\sigma ^* \mathcal{A}_n ^{\dr }(Y)$ is isomorphic to $\mathcal{A}_n ^{\dr }(Y)$. In this case the connection structure on $\sigma ^* \mathcal{A}_n ^{\dr }$ is given by
$$
v \otimes 1 \mapsto -\sum _{i=0}^{2g+r-2} T_i v \otimes \sigma ^* \omega _i .
$$
Restricting to the fibre at $b$, we obtain an automorphism of the algebra $A_n ^{\dr }(Y)(b)$.
For example, suppose $X/\Q$ is hyperelliptic, given by 
$$
y^2 =f(x)=x^{2g+2}+a_{2g+1}x^{2g+1} +\cdots + a_0,
$$
$Y=X-\{ \infty ^{\pm }\}$, and $\omega _i =x^i dx /2y$. Then pulling back by the hyperelliptic involution $w$ sends $\mathcal{A}_n ^{\dr } (Y)$ to the connection $
v \otimes 1 \mapsto \sum _{i=0}^{2g+r-2} T_i v \otimes \omega _i .
$
Hence we deduce that with respect to this affine trivialisation, at any Weierstrass point $b$, the automorphism on the algebra $A_n ^{\dr }(Y)(b)$ induced by $w$ is simply given by $
T_i \mapsto -T_i .
$
\begin{Definition}
For an effective divisor $D$ on $X$ whose support has points $z_1 ,\ldots ,z_n $ in an algebraic closure, we let $D[1]$ denote the divisor $D+\sum _{i=1}^n z_i $.
\end{Definition}
\begin{Lemma}\label{boundthedivisor}We have the following:
\begin{enumerate}
\item The constants $c_{ik}^H$ are independent of base point.
\item Suppose $X$ is hyperelliptic, with defining equation as above, and the $\eta _i $ are taken to be a $K$-linear combination of the basis differentials $\omega _i =x^i dx/2y$. Then $c^H _{ik}(x)$ is zero for all $i$ and $k$, and $\xi _k =0$ for all $1\leq k \leq d$. 
\item Suppose $\eta _0 ,\ldots ,\eta _{2g-1}$ are differentials in $H^0 (X,\Omega (D))$, for some effective divisor $D$, and $\eta _0 ,\ldots ,\eta _{g-1}$ are a basis of $H^0 (X,\Omega )$. Then we have that for all $k<2g-1$, $r_k ^H \in H^0 (X,D[1])$.
\end{enumerate}
\end{Lemma}
\begin{proof}
For part (1), we use the characterisation of $c_{ik}$ from Lemma \ref{extM1}. By \eqref{defofc&r}, changing the basepoint $b$ changes $r$ by a constant, but does not alter the $r_i (\pi )$. \\
For part (2), it suffices to prove this after a finite extension of the base field, and by part (1) we may assume that $b$ is taken to be Weierstrass. As we did in the \'etale setting, we observe that $w$ then induces an automorphism of the bundle $\mathcal{A}|_Y$. With respect to the affine trivialisation of $\mathcal{A}$ at $b$, $w$ acts as -1 on the $V_{\dr }$ component, and acts as 1 on the $W_{\dr }$ component. By functoriality, the involution must respect $w$, and hence we conclude all the $r^H _{ik}$ must be zero.
Similarly, by the explicit description of $\xi _k $ given in equation \eqref{eqn:explicit_residue}, we see that the residue of $\xi _k$ is equal to the residue of a sum of differentials which are even with respect to the hyperelliptic involution, and hence zero.
For part (3), this follows from the defining property (see \eqref{defofc&r}) of the function $c$ used to define the $c_{ik} ^H $.
\end{proof}

We now explain how to carry out some of these calculations for a hyperelliptic curve. We consider $X/K$ given by
$
y^2 =f(x)=x^{2g+2}+a_{2g+1}x^{2g+1} +\cdots + a_0,
$
and let $Y=X-\{ \infty ^{\pm }\}$ and  $\omega _i =x^i dx /2y$. 

\subsection{Computing $(c,r) $ for even degree models of hyperelliptic curves}\label{subsec:hodge_hyp}
The set $\{\omega _0 ,\ldots ,\omega _{2g}\}$ forms a basis of $H^1 _{\dr }(Y)$, and the set $\{\omega _0 ,\ldots ,\omega _{g-1}\}$ forms a basis of 
$H^0 (X,\Omega ^1 )$. In general $\omega _0, \ldots ,\omega _{2g-1}$ will not form a basis of $H^1 _{\dr }(X)$, so we take $\eta _0 ,\ldots ,\eta _{2g}$ in the $K$-span of $\omega _0 ,\ldots ,\omega _{2g}$ forming a basis of $H^1 _{\dr }(Y)$ such that $\eta _0 ,\ldots ,\eta _{g-1}$ form a basis of $H^0 (X,\Omega ^1 )$ and $\eta _0 ,\ldots ,\eta _{2g-1}$ form a basis of $H^1 _{\dr }(X)$. Let $W_{\dr }$ be any filtered quotient of $\overline{\wedge ^2 V}_{\dr }$. By truncating the power series expansion of $x^{2g+2} \sqrt{f(x^{-1})}$, we find polynomials $f_{i,\infty ^\pm }$ in $uK[u]$ such that 
$$
\omega _i -df_i \in u^{-1}K[u]du.
$$
Similarly we find the functions $g_{i,\infty ^\pm }$ and $h_{i,\infty ^{\pm }} $.

In the notation of the previous section, $r=2$, $m=g$, and for $x=\infty ^{\pm }$, we may take the uniformiser $t_x$ to be $u:=x^{-1}$. The function 
$$
(c,r ) :(u^{-g}K[u]/K[u] )\times (u^{-g}K[u]/K[u] )\to H^0 (X,\mathcal{O}(g\infty ))\times K^g
$$
is given as follows: let $s(x)=\sum _{i=1 }^g s_i u^{-i}$ be a representative of an element of $u^{-g}K[u]/K[u] $. For any polynomial $s(x)$, we have $c(s(x),s(x))=s(x)$ and $r(s(x))=0$.
 Define $B=(B_{ij})$ by
$\loc _{\infty ^+ }(\omega _{i+g})-\sum B_{ij}u^j du \in K[u]du .$
Then $(c(s(x),-s(x))=0$ and $\underline{r}(s(x),-s(x))=B^{-1}(\underline{s})$, where $\underline{s}:=(s_1 ,\ldots ,s_g )$.

\subsection{Frobenius structure on the universal connection of a hyperelliptic curve}\label{frobenius_structure}
In order to complete the description of the filtered $\phi $-module structure, we need to describe the Frobenius action on the fibres $A^{\dr }(b,z)$ of the connection $\mathcal{A}^{\dr }$. Although it will not be needed in this paper, for completeness we briefly outline how this computation might be carried out for a general curve. Let $X_{\mathbb{F}_p }$ be the special fibre of a smooth model of $X$ over $\mathbb{Z}_p $, and let $\phi $ be an overconvergent lift of the absolute Frobenius morphism to some wide open subspace in the rigid analytification of $X_{\mathbb{Q}_p}$. The analytifications of the pointed connections $(\mathcal{A}_n ,1)$ may be viewed as universal pointed objects $(\mathcal{A}_n ^{\dagger },1)$ in the category of unipotent isocrystals on $X_{\mathbb{F}_p }$. The action of Frobenius on the category of unipotent isocrystals induces a Frobenius structure on $\mathcal{A}_n ^\dagger $, and one may reduce the problem of computing the action of Frobenius on $A^{\dr }(b,z)$ to that of computing this Frobenius structure.

For a hyperelliptic curve, we use the hyperelliptic splitting principle to determine the filtered $\phi $-module $A^{\dr }(b,z)$ when $b=z$ is a Weierstrass point. This gives a characterisation of the $\phi$-module structure of $A^{\dr }(b,z)$ for general $b$ and $z$ in terms of Coleman integrals.
\begin{Lemma}\label{phi-bit} 
\begin{enumerate}
\item Let $X$ be a hyperelliptic curve, and $\eta _i $ as in section \ref{subsec:hodge_hyp}. With respect to the affine trivialisation, the unipotent $\phi $-equivariant isomorphism
\[
\Q _p \oplus V_{\dr }\oplus W_{\dr }\stackrel{\simeq }{\longrightarrow }A^{\dr }(b,z)
\]
is given by 
\[
\left( \begin{array}{ccc} 1 & 0 & 0 \\
\sum _{i=0}^{2g-1} T_i \int ^z _b \eta _i & 1 & 0 \\
\sum _{1\leq k\leq d}\left( \sum _{0\leq i,j\leq 2g-1}\tau _{ijk}\int ^z _b \eta _i \eta _j \right) S_k & \sum _{0\leq i,j\leq 2g-1,1\leq k\leq d}-\tau _{ijk}T_i ^* \otimes S_k \int ^z _{w(b)}\eta _j & 1 \\
\end{array} \right) ,
\]
modulo $F^0 \Hom (V_{\dr },W_{\dr })$. 
\item For general smooth projective $X$, there are constants $c ^\phi _{ik}$, independent of $z$, such that the $\phi $-equivariant isomorphism is given by 
\[
\left( \begin{array}{ccc} 1 & 0 & 0 \\
\sum _{i=0}^{2g-1} T_i \int ^z _b \eta _i & 1 & 0 \\
\sum _{1\leq k\leq d}\left( \int ^z _b \xi _k +\sum _{0\leq i,j\leq 2g-1}\tau _{ijk}\int ^z _b \eta _i \eta _j \right) S_k & \sum _{0\leq i<2g,1\leq k\leq d}( c ^\phi _{ik}-\sum _{0\leq j<2g}\tau _{ijk} \int ^z _{b}\eta _j  )T_i ^* \otimes S_k & 1 \\
\end{array} \right).
\]
\end{enumerate}
\end{Lemma}
\begin{proof}
We compute the isomorphism as the composite of $\phi $-equivariant isomorphisms
\[
\Q _p \oplus V_{\dr }\oplus W_{\dr }\stackrel{\simeq }{\longrightarrow}A^{\dr }(b)\stackrel{\simeq }{\longrightarrow }A^{\dr }(b,z).
\]
We compute the latter isomorphism first. By definition, such an isomorphism is given by iterated Coleman integrals, as in \cite[Corollary 3.3]{besser:2002}. More precisely, 
for all $z_1 ,z_2 ,z_3 $ in $Y(\mathbb{Q}_p )$, the unipotent $\phi $-equivariant isomorphism
$$
A^{\dr }(z_1 ,z_2 )\stackrel{\simeq }{\longrightarrow }A^{\dr }(z_1 ,z_3 )
$$
is given by 
\begin{align*}
1 & \mapsto 1+\sum _{i=0}^{2g-1} \int ^{z_3 }_{z_2 }\eta _i \otimes T_i +\sum _{k=1}^d (\int ^z _b \xi _k +\sum _{0\leq i<j\leq 2g-1} \tau _{ijk}\int ^{z_3 }_{z_2 }(\eta _i \eta _j -\eta _j \eta _i ))S_k \\
T_i & \mapsto T_i -\sum _{0\leq j\leq 2g+r-2,1\leq k\leq d} \tau _{ijk}\int ^{z_3 }_{z_2 }\eta _j \otimes S_k , \\
S_k &\mapsto S_k .
\end{align*}
This proves part (2). For part (1), we compute the other isomorphism. By Lemmas \ref{leftrighttwist} and \ref{hyperellipticsplitting}, we know that, modulo $F^0 \Hom (V_{\dr },W_{\dr })$, the $\phi$-equivariant splitting is given by $T_i -\mapsto \sum _{j,k}\tau _{ijk} \int _{2b-D}\eta _j \otimes S_k .$
Again, by the definition of Coleman integration we have 
$$
t(A_1 ^{\dr } (b,z) ,IA^{\dr }(b,z) )=\sum _{0\leq i<g }\int ^z _b \eta _i T_i -\sum \tau _{ijk}\left(\int ^z _b \eta _i \right)\left(\int _{z+b-D}\eta _j \right) \otimes S_k
$$
and for $i<g$, we have
$
\int ^z _{w(b)}\eta _i =\int ^z _b \eta _i +\int _{2b-D}\eta _i .
$
\end{proof}
\begin{Lemma}\label{preheightisgivenby}
\begin{enumerate}
\item
Let $X$ be a hyperelliptic curve, and $\eta _i $ a basis of $H^1 _{\dr }(Y)$ as in section \ref{subsec:hodge_hyp}. Then 
the generalised pre-height of $A(b,z)$ is given by 
\begin{align*}
\widetilde{h}_{\mathfrak{p}}(A(b,z)) & = \sum _k \left( -r^H _k (z)+\sum _{0\leq i<j<2g}\tau _{ijk}\int ^{z } _{b} (\eta _i \eta _j -\eta _j \eta _i -(\int ^z _b \eta _i )(\int ^b _{w(b)}\eta _j )-(\int ^z _b \eta _j )(\int ^b _{w(b)}\eta _i )) \right) S_k \\
\end{align*}
\item Let $X$ be a general smooth projective curve. $Y\subset X$ and $( \eta _i )$ be as in section \ref{defn:YsubX} and $\xi $ as in section \ref{computehodge}. Then
\begin{align*}
\widetilde{h}_{\mathfrak{p}}(A(b,z)) & = \sum _k \left( -r^H _k (z)+\int ^z _b \xi _k -\sum _{0\leq i<g}c^\phi _{ik}\int ^z _b \eta _i +\sum _{0\leq i<j <2g}\int ^z _b (\eta _i \eta _j -\eta _j \eta _i )-\sum _{g\leq i<2g}c^H _{ik}\int ^z _b \eta _i \right) S_k
\end{align*}
\end{enumerate}
\end{Lemma}
\begin{proof}
We compute the local pre-height using Lemma \ref{explicit_formula}. The $\phi $-equivariant isomorphism from $\Q _p \oplus V_{\dr }\oplus W_{\dr }$ is computed in Lemma \ref{phi-bit}. A filtration preserving isomorphism is computed in section \ref{computehodge} in the general case and \ref{subsec:hodge_hyp} in the hyperelliptic case. If we write the $F^0 \backslash U(\Q _p ,V_{\dr },W_{\dr })$ class of $A^{\dr }(b,z)$ as 
\[
\left(
\begin{array}{ccc}
1 & 0 & 0 \\
\alpha & 1 & 0 \\
\gamma & \beta & 1 \\
\end{array}
\right)
\]
then, in the general case, we have 
\begin{align*}
\alpha & = \sum _{i=0}^{2g-1} T_i \int ^z _b \eta _i , \qquad \beta = \sum _{0\leq i\leq 2g-1,1\leq k\leq d}(-c ^H _{ik}-\sum _{0\leq j<2g}\tau _{ijk}) T_i ^* \otimes S_k \int ^z _{w(b)}\eta _j \\
\gamma & = \sum _{1\leq k\leq d}\left(\int ^z _b \xi _k -r^H _k (z) +\sum _{0\leq i<2g}-c^H _{ik}\int ^z _b \eta _i +\sum _{0<j<2g}\tau _{ijk}\int ^z _b \eta _i \eta _j  \right) S_k \\
\end{align*}
In the hyperelliptic case, the formula for $\alpha $ is the same, and 
\begin{align*}
\beta & =  -\sum _{0\leq i,j\leq 2g-1,1\leq k\leq d}\tau _{ijk}T_i ^* \otimes S_k \int ^z _{w(b)}\eta _j \\
\gamma & = \sum _{1\leq k\leq d}\left(\int ^z _b \xi _k -r^H _k (z) +\sum _{0\leq i,j\leq 2g-1}\tau _{ijk}\int ^z _b \eta _i \eta _j  \right) S_k . \\
\end{align*}
The Lemma now follows from Lemma \ref{explicit_formula}.
\end{proof}
A corollary of Lemmas \ref{boundthedivisor} and \ref{preheightisgivenby} is the following explicit general formula. Let $\eta _0 ,\ldots ,\eta _{2g-1}$ be  differentials of the second kind in $H^0 (X,\Omega ^1 (D))$ for some effective divisor $D$. Let $|D|_{\mathbb{F}_p }$ denote the reduction mod $p$ of the support of $D$, and let $\mathcal{W}\subset X_{\mathbb{Q}_p }$ denote the tube of $p$-adic points which are not congruent to $|D|$ mod $p$. 
\begin{Proposition}\label{prop5}
Suppose $m:=\rho _f (J)-1>r-g$. Then  there exist constants $a_{ijk},b_{ijk},c_{ijk},b_{ik},c_{ik}$, rational functions $s_k \in H^0 (X,\mathcal{O}(D[1]))$, and differentials of the third kind $\xi _k $, such that
\begin{align*}
X(\mathbb{Q}_p )_U \cap \mathcal{W} & \subset \{ z\in \mathcal{W}  : R _{T_1 ,\ldots ,T_{r-g} }(F_1,\ldots  ,F_{m} )=0 \},\\ 
F_k (T_1 ,\ldots ,T_m ,z) & =\sum _{1\leq i,j\leq r-g}a_{ijk}T_i T_j +\sum _{1\leq i\leq r-g,0\leq j<g}b_{ijk}T_i \int ^z _b \eta _k +\sum _{1\leq i\leq r-g}b_{ik}T_i \\
&\qquad +\sum _{0\leq i,j<2g}c_{ijk}\int ^z _b \eta _i \eta _j +\sum _{0\leq i < 2g}c_{ik}\int ^z _b \eta _i +\int ^z _b \xi _k +s_k.
\end{align*}
\end{Proposition}
\begin{proof}
By Proposition \ref{exactformula1}, the set $X(\mathbb{Q}_p )_2 $ is contained in the intersection of the zeroes of $F_k $.  
\label{phi-bit}
Using the identity $\int \omega _i \omega _j +\omega _j \omega _i =\int \omega _i \int \omega _j $, we can write the formula for the generalised pre-height as 
$c_{ij}\int _b ^z \omega _i \omega _j  
+\sum c_i \int ^z _b \omega _i +\int ^z _b \eta + s. $
\end{proof}

\section{Computing $X(K_{\mathfrak{p}})_U$}\label{sec:algorithms}

\subsection{Theorem \ref{thm2}, general case}
We now return to the setting of Section \ref{subsec:KMS}. $X$ is a curve of the form
\[
y^2 =x^6 +ax^4+ ax^2 +1
\]
with $a\in K_0 $, where $K_0 $ is $\Q$ or a real quadratic field, and the base field $K$ is a totally real extension of $K_0$.

Let $T_{0,V}$ denote the set of primes of potential type V reduction. At each $v$ in $V_0$ we choose an ordering of the two components of the special fibre of the stable model of $X$ over $\mathcal{O}_{K_w}$. Over such an extension, the dual graph of a minimal regular model is then a ``line'', i.e., a graph with vertex set $\{ v_0 ,\ldots ,v_n \}$ and edge set $\{ e_0 ,\ldots ,e_{n-1} \}$ where $e_i $ is an edge from $v_i $ to $v_{i+1}$. Define $\pi _v :X(K_v )\to \mathbb{Q}$ to be the map sending a point $x$ to $i/n$, where $v_i $ is the unique vertex containing the reduction of $x$ (note that the ratio $i/n$ is independent of the choice of extension $K_w $). Finally, if $\alpha $ is a function from $T_{0,V}$ to $\mathbb{Q}$, we let $X(K)_{\alpha }$ denote the set of rational points for which $\pi _v (x)=\alpha (v)$ for all $v\in T_{0,V}$. The theorem involves the following hypothesis.
\\
\textit{Hypothesis (H): For all $v\in T_0 $ of potential type V reduction, the map
$$
j_v : X(K_v )\to H^1 (G_v ,U)
$$
factors as $X(K_v )\to \mathbb{Q}_p \to H^1 (G_v ,U)$, where the first map sends 
$z$ to $\pi _v (z)-\pi _v (b)$ and the second map is a vector space homomorphism.} 
\\
In future work of the second author and Alex Betts it will be shown, using the methods of Oda \cite{oda:1995}, that all $X$ in the family satisfy Hypothesis (H).

\begin{Theorem2*}Let $K_0$ be $\mathbb{Q}$ or a real quadratic field. Let $K|K_0 $ be a totally real extension. Let $X/K_0 $ be a genus 2 curve in the family $y^2 =x^6 +ax^4 +ax^2 +1$ whose Jacobian has Mordell--Weil rank 4 over $K$. Let $b\in X(K)$ denote the point $(0,1)$. Assume $X$ satisfies Hypothesis (H) (see below), and that there is a prime $p$ of $\mathbb{Q}$ such that 
\begin{itemize}
\item The prime $p$ splits completely in $K|\mathbb{Q}.$
\item The curve $X$ has good reduction at all primes above $p$, and the action of $G_K$ on $E[p]$ is absolutely irreducible.
\item If $E$ has complex multiplication by a CM extension $L$, then $L$ is not contained in $K(\mu  _p )$.

\end{itemize} 
Then there exist constants $\lambda _v, \mu _v \in \mathbb{Q}_p  , v\in T_{0,V}$ with the following property:
Suppose $z_0 $ is a point in $X(K)$ such that $f_1 (z_0 )\wedge f_2 (z_0 )$ is of infinite order in $\wedge ^2 E(K)$. Then for all $\alpha :T_{0,V}\to \mathbb{Q}$, $X(K)_{\alpha }$ is contained in the finite set 
of $z$ in $X(K_{\mathfrak{p}})$ satisfying $G(z) = 0$, where
\begin{itemize}
\item
\[
G(z) = \det \left( \begin{array}{cc}
F_1 (z)+\sum _{v\in T_{0,V}}\lambda _v (\alpha (v)-\pi _v (b))  & F_2 (z_0 )+\sum _{v\in T_{0,V}}\mu _v (\pi _v (z_0 )-\pi _v (b)) \\
F_1 (z_0 )+\sum _{v\in T_{0,V}}\lambda _v (\pi _v (z_0 )-\pi _v (b))  & F_2 (z)+\sum _{v\in T_{0,V}}\mu _v (\alpha (v)-\pi _v (b)) \\
\end{array} \right) ,
\]
\item 
\[
F_1 (z) = \int ^z _b (\omega _0 \omega _1 -\omega _1 \omega _0 )+\frac{1}{2}\int ^z _b \omega_0 \int ^b _{w(b)}\omega _1,
\]
\item 
\[
F_2 (z) =2\int ^z _{b }(-\omega _0 \omega _3 +a \omega _1 \omega _2 +2\omega _1 \omega _4 )-\frac{1}{2}x(z) -\int ^z _{b }\omega _0 \int ^{b }_{w(b)}\omega _3.
\]
\end{itemize}
\end{Theorem2*}

\subsection{Computing $(\underline{c}^H,\underline{r}^H)$ for the Kulesz--Matera--Schost family}
To complete the proof of Theorem \ref{thm2}, by Lemma \ref{lemma22}, it will be enough to show that, with respect to a suitable basis of $W_{\dr }/F^0 $, we have 
$$
\widetilde{h}_p (A(b,z))-\frac{1}{2}\widetilde{h}_p ([E_1 ,E_2 ])=(F_1 (z),-F_2 (z)).
$$
We shall prove this by explicitly determining the functions $f_{i,x},g_{i,x}, h_{i,x},  \underline{c}^H$ and constants  $\underline{r}^H$ from Section \ref{sec:hodgefiltration}.

Let $X$ be a hyperelliptic curve of the form $y^2 =x^6 +ax^4 +ax^2 +1$. Denote by $\{ \infty ^+,\infty ^- \}$ the points at infinity with respect to this model. Suppose $b$ is a rational point of $X$ and $U(b)$ is the quotient of the fundamental group defined in Section 1. Recall the maps $f_1 $ and $f_2 $ from the introduction.
The set $\{\omega _0 ,\ldots \omega _{4}\}$ forms a basis of $H^1 _{\dr }(Y)$ and a basis of $H^1 _{\dr }(X)$ is given by $\{\eta _0 =\omega _0 ,\eta _1 =\omega _1 ,\eta _2 =a\omega _2 +2 \omega _4 , \eta _2 =\omega _3 \}$. Let $T_0 ,T_1 ,T_2 ,T_3 $ be the corresponding dual basis.
For the quotient $\mathcal{A}^{\dr }(X)$, we find that all the $\xi _k $ are zero, so that 
$$1 \mapsto -\sum _{i=0}^3 \eta _i \otimes T_i, \qquad T_j \mapsto -\sum _{0\leq j<3,1\leq k \leq 3} \eta _i \otimes (\tau _{ijk} S_k ), \qquad S_k \mapsto 0$$
extends to a connection on $X$.

Let $\omega _E = dx/2y$ denote the canonical Weierstrass differential on $E$. Let $T_{E,1}$ and $T_{E,2}$ denote the basis of $H^1 _{\dr }(E-O)$ dual to $[\omega _E ],[x\omega_E]$.
The set $\{S_0 =T_{E,0} T_{E,0} ,S_1 =T_{E,0}T_{E,1},S_2 =T_{E,1} T_{E,1}\}$ forms a basis of $W_{\dr }$, and the set $\{S_0 ,S_1\}$ forms a basis of $W_{\dr }/F^0 $. Since the map $\tau $ factors through $\wedge ^2 V_{\dr }$, it is enough to specify its values 
on the elements $T_i \wedge T_j $. These may be calculated by observing that
\[
f_1 ^* [\omega _E ]=[\eta _1 ],\qquad f_{2*}[\omega _E ]=[\eta _0 ], \qquad f_1 ^* [x\omega _E ]=[\eta _3 ], \qquad f_2 ^* [x\omega _E ]=[\eta _2 ].
\]
Hence we deduce by equation \eqref{eqn:explicit_wedge} that
$$\tau (T_0 \wedge T_1 )= - S_0, \qquad \tau (T_0 \wedge T_3 ) =-\tau (T_1 \wedge T_2 )  = -S_1,\qquad \tau (T _2 \wedge T_3 )  = -S_2.$$
With respect to these bases, we find that $
\underline{c}^H=0 $ and $ \underline{r}^H=(0,\frac{1}{2}x(z)-\frac{1}{2}x(b))$.

\subsection{Local constants at primes of bad reduction}\label{local_justification}
We now explain how to compute local pre-heights at primes away from $p$, under the assumption of Hypothesis (H). First we explain why a non-trivial contribution at $v\in T_0 $ can only arise when $v$ is a prime of potential type $V$ reduction.

\begin{Lemma}
Suppose $X$ has potential good reduction at $v$. Then $j_v$ is trivial.
\end{Lemma}
\begin{proof}
Recall from \cite[I.5.8]{serregc:1997} that, given a profinite group $G$, closed normal subgroup $H$, and $G$-group $A$, we get an exact sequence of pointed sets
\[
H^1 (G/H,A^H ) \to H^1 (G,A)\stackrel{\res }{\longrightarrow } H^1 (H,A).
\]
Applying this when $G=G_v $, $H=G_w$ is the Galois group of a finite extension $L_w$ of $K_v$ over which $X$ acquires good reduction. The commutative diagram
$$
\begin{tikzpicture}
\matrix (m) [matrix of math nodes, row sep=3em,
column sep=3em, text height=1.5ex, text depth=0.25ex]
{X(K_v) & H^1 (G_v ,U)  \\
 X(L_w) & H^1 (G_w,U) \\ };
\path[->]
(m-1-1) edge[auto] node[auto] {$j_v$} (m-1-2)
edge[auto] node[auto] {} (m-2-1)
(m-1-2) edge[auto] node[auto] {$\res $} (m-2-2)
(m-2-1) edge[auto] node[auto] { } (m-2-2);
\end{tikzpicture} $$
implies that the composite $\res \circ j_v$ is trivial. Hence to prove the Lemma it is enough to show that $U^{G_w}$ is trivial, which may be seen from the fact that $V$ and $\Sym ^2 V_E$ are pure of weight $-1$ and $-2$ respectively over $L$.
\end{proof}

\begin{Lemma}
Suppose $E$ does not have potential good reduction at $v$. Then $j_v$ is trivial.
\end{Lemma}
\begin{proof}
It will be enough to prove that $H^1 (G_v ,U)$ is trivial. This is done using the exact sequence of pointed sets
\[
H^1 (G_v ,\Sym ^2 V_E )\to H^1 (G_v ,U)\to H^1 (G_v ,V).
\]
Recall from Lemma \ref{locallemma} that $H^1 (G_v ,V)=0$. Hence it is enough to show that $H^1 (G_v ,\Sym ^2 V_E )=0$. This is well known (see e.g. \cite[Lemma 2.10]{flach:1992}) but we recall the proof for the sake of completeness. Let $L_w$ be a finite extension of $K_v$ over which $E$ acquires semi-stable reduction. Then $\res _{G_w}V_E$ is a nontrivial extension of $\mathbb{Q}_p $ by $\mathbb{Q}_p (1)$. Hence $\Sym ^2 V_E $ is an extension
$$
0\to \mathbb{Q}_p (2)\to \Sym ^2 V_E \to (\Sym ^2 V_E )/\mathbb{Q}_p (2)\to 0,
$$
and $(\Sym ^2 V_E )/\mathbb{Q}_p (2)$ is a nontrivial extension of $\mathbb{Q}_p $ by $\mathbb{Q}_p (1)$. Then we have $H^1 (G_v ,\Sym ^2 V_E )=0$ as in Lemma \ref{locallemma}, and $H^1 (G_v ,\mathbb{Q}_p (2))$ $=0$ for weight reasons.
\end{proof}

The only remaining case is where $E$ has potential good reduction but $X$ does not, which implies that $X$ has potential type $V$ reduction. Again using injectivity of the restriction map we can recover $j_{2,v}$ from its image in $H^1 (G_{L},\Sym ^2 V_E)$ which is determined (up to a scalar) by Hypothesis (H).

\subsection{Completion of proof}
We now explain how to use this explicit description of generalised heights on $X$ to prove Theorem \ref{thm2}.
This gives the following proposition.
\begin{Proposition}
With respect to the basis $S_0 ,S_1 $ of $(\Sym ^2 V_E ^{\dr })/F^0 $, the local heights  $\widetilde{h}_{\mathfrak{p}}(A(b,z))$ and $\widetilde{h}_{\mathfrak{p}}([E_1 ,E_2 ])$ are given by 
\[
\widetilde{h}_{\mathfrak{p}}(A(b,z))=  (F_1 (z)+\frac{1}{2}\int ^z _b \omega _0 \int ^z _{w(b)}\omega _1 )S_0 -F_2 (z) S_1 , \qquad
\widetilde{h}_{\mathfrak{p}}([E_1 ,E_2 ])  = \frac{1}{2}\int ^z _b \omega _0 \int ^z _{w(b)}\omega _1 S_0 .
\]
\end{Proposition}
\begin{proof}
We have an isomorphism $H^1 _f (G_{\mathfrak{p}} ,V_E )\simeq \mathbb{Q}_p .T_{E,0} $ using the basis of $H^1 _{\dr }(E)$ above, and given extensions $[E_1 ]=\lambda _1 .T_{E,0} $ and $[E_2 ]=\lambda _2 .T_{E,0}$. Then the class of $[E_1 ,E_2 ]$ in $F^0 \backslash U(\Q _p ,V_{\dr },W_{\dr })$ is given by 
\[
\left( \begin{array}{ccc}
1 & 0 & 0 \\
\lambda _1 T_{E,0}+\lambda _2 T_{E,1} & 0 & 0 \\
\lambda _1 \lambda _2  S_0 & \lambda _2 T_{E,0}^* \otimes S_0 + \lambda _1 T_{E,0}^* \otimes S_0 & 1 \\
\end{array} \right).
\]
Hence the pre-height of $[E_1 ,E_2 ]$ is given by $-\lambda _1 \lambda _2 $, and the result follows from Lemma \ref{lemma21}. 
%
\end{proof}
 Hence we find 
$$
\widetilde{h}_{\mathfrak{p}}(A(b,z))-\widetilde{h}_{\mathfrak{p}}([E_1 ,E_2 ])=(F_1 (z),F_2 (z)).
$$

\subsection{Examples}
\subsubsection{Example 1: $K=\mathbb{Q}$, $a=31$, $p=3$}\label{Ex1}
The curve $E_{31}$ has rank 2 over $\mathbb{Q}$. To determine the local constants, we first need to find the primes of potential type $V$ reduction.
$X$ has potential good reduction at all primes away from $2$ and $7$, which are both of potential type V reduction. 
\begin{enumerate}
\item $v=7$: in this case we find that all $\mathbb{Q}_7$-points reduce to a common component of the minimal regular model over $\mathbb{Z}_7$. Hence they reduce to a common component of the stable model of $X$ over a finite extension of $\mathbb{Q}_7$, and so by (H) the contribution at 7 is zero. 
\item $v=2$: we observe $H^1 (G_{\mathbb{Q}_2 },\Sym ^2 V_E )=0$, which implies that $H^1 (G_{\mathbb{Q}_2 },U )=0$. One way to see this is to note that for $H^1 (\mathbb{Q}_2 ,\Sym ^2 V_E )$ to be nonzero, it is necessarily the case that $\Hom _{G_{\mathbb{Q}_2 }}(T_3 E,T_3 E)$ has rank bigger than 1, which means that the action of inertia at 2 must factor through an abelian subgroup of GL$_2 (\mathbb{F}_3$). This does not happen at $a=31$, because $E$ does not acquire good reduction over any $(\mathbb{Z}/2 )^2$ or degree 3 extension of $\mathbb{Q}_2 $.  
\end{enumerate}
Hence our equation for rational points simplifies to $F_1(z)F_2(z_0) = F_1(z_0)F_2(z)$. The set of solutions is tabulated below. We find $X(\Q _3 )_U$ appears to contain 8 non-rational points.
\begin{center}
 \begin{tabular}{|c |  r  |c | c |}
    \hline
$\overline{z}\in X(\F_3)$ & $x(z) \in \Z_p$ & $z \in X(\Q)$ \\
  \hline
  & $O(3^{7})$ & $(0,\pm 1)$ \\ 
 $\overline{(0, \pm 1)}$ & $2 \cdot 3 + 2 \cdot 3^{3} + 2 \cdot 3^{5} + O(3^{7})$ &  \\
 & $3 + 2 \cdot 3^{2} + 2 \cdot 3^{4} + 2 \cdot 3^{6} + O(3^{7})$ &  \\
\hline 
 & $1 + O(3^{7})$ & $(1,\pm 8)$\\ 
 $\overline{(1, \pm 2)}$ & $1 + 2 \cdot 3 + O(3^{7})$ & $(7, \pm 440) $ \\ 
 & $1 + 3 + 2 \cdot 3^{3} + 3^{4} + 2 \cdot 3^{5} +O(3^{7})$ & $(\frac{1}{7}, \pm\frac{440}{343})$  \\
 \hline
 & $2 + 2 \cdot 3^{2} + 2 \cdot 3^{3} + 2 \cdot 3^{4} + 2 \cdot 3^{5} + 2 \cdot 3^{6}  + O(3^{7})$ & $(-7,\pm 440)$ \\ 
 $\overline{(2, \pm 2)}$ & $2 + 3 + 2 \cdot 3^{2} + 3^{4} + 2 \cdot 3^{6} + O(3^{7})$ & $(-\frac{1}{7}, \pm \frac{440}{343})$ \\ 
 & $2 + 2 \cdot 3 + 2 \cdot 3^{2} + 2 \cdot 3^{3} + 2 \cdot 3^{4} + 2 \cdot 3^{5} + 2 \cdot 3^{6}  +O(3^{7})$ & $(-1, \pm 8)$  \\
 \hline 
 & $2 \cdot 3^{-1} + 1 + 2 \cdot 3 + 2 \cdot 3^{2} + 2 \cdot 3^{3} + 2 \cdot 3^{4}+ O(3^{7})$ &   \\
 $\overline{\infty^{\pm}}$ & $3^{-1} + 1 + 2 \cdot 3^{5} + 2 \cdot 3^{6} + O(3^{7})$ &   \\
 & $\infty^{\pm}$ & $\infty^{\pm}$  \\
  \hline
\end{tabular}
\end{center}

\subsubsection{Example 2: $K=\mathbb{Q}(\sqrt{3})$, $a=19$, $p = 11,$ $\mathfrak{p}=(2\sqrt{3}+1)$}
Note that in the previous example, all  identities which the functions $F_i$ satisfy on rational points are also implied by the functional equations the $F_i $ satisfy with respect to the automorphisms of the curve. Numerical experiments suggest that it is rare for the formula in Theorem \ref{thm2} to produce nontrivial identities that the $F_i $ satisfies on rational points, as when a curve has many rational points relative to its Mordell--Weil rank, it typically has many potential type V primes. 

However, there are instances where the theorem produces nontrivial identities between the values of $F_i (z)$ on rational points. When $a=19$ and $K=\mathbb{Q}(\sqrt{3})$, we find that $E(K)$ has rank 2, the prime above 2 is the only potential type V prime and the set $X(\mathbb{Q}(\sqrt{3}))$ has at least 28 points, coming from the $\Aut (X)$-orbits of $(0,1)$ together with the points
 $z_1 =(\sqrt{3},16)$, $z_2 =(-\sqrt{3}+2,-24\sqrt{3}+40)$ and $z_3 =(-39\sqrt{3}/71 + 98/71 , -2736216\sqrt{3}/357911 + 5551000/357911 )$. 

Local constants at $v$ above 2: one can compute a semistable model of $X$ over the totally ramified extension $L$ of $\mathbb{Q}_2 $ cut out by the polynomial
$$ 
F(t)=t^8 + 32t^7 + 448t^6 + 3584t^5 + 16096t^4 + 28160t^3 - 18432t^2 - 6912
$$
by first computing a smooth model of $E$ over $L$, for example as described in \cite[\S 10.2.3]{liu:2002}. 
Let $\beta $ be a root of $F$ in $L$, and define $\gamma := \frac{1}{3}(\beta ^2 +8\beta ). $
Using this model, we can show that the regular semistable model of $X$ over $L(\sqrt{3})$ has 9 irreducible components, and that the map
$
\pi _v :X(K_v )\to \{ a/8 :0\leq a\leq 8\}
$
is given by 
$$
z\mapsto \left\{ \begin{array}{cc}
0 & v(x(z)-1)>1,v(x(z)^2 +1-\gamma )\geq 3  \\
(3-v(x(z)^2 +1-\gamma ))/2 & v(x(z)-1)>1, 2\leq v(x(z)^2 +1-\gamma )\leq 3 \\
1/2 & v(x(z)^2 +1-\gamma )\leq 2 \\
(v(x(z)^2 +1-\gamma )-1)/2 & v(x(z)-1)=1, 2\leq v(x(z)^2 +1-\gamma )\leq 3 \\
1 & v(x(z)-1)=1,v(x(z)^2 +1-\gamma )\geq 3  \\
\end{array} \right.
$$
where the valuation $v$ is normalised so that $v(2)=1$.
For example this tells us that $\pi _v (z_0 )=1/2$. For $z_1 ,z_2 $ and $z_3 $, we are in case 4, since
$$
v(x(z_1 )^2 -1+\gamma )=5/2,
v(x(z_2 )^2 -1+\gamma )=11/4,
v(x(z_3 )^2 -1+\gamma )=11/4.
$$
Hence $\pi _v (z_1 )=3/4$ and $\pi _v (z_2 )=\pi _v (z_3 )=7/8$. 
We find that the divisor $$3[z_2 ]+[z_3 ]-6[z_1 ]$$ maps to zero in $\wedge ^2 E(K)\otimes \mathbb{Q}$ and in $H^1 (G_v ,\Sym ^2 V_E )$. Working at  a prime above 11, we compute that
$$
3F_1 (z_2 ) +F_1 (z_3 )-6F_1 (z_1 )=
3F_2 (z_2 ) +F_2 (z_3 )-6F_2 (z_1 ) =O(11^{18}).
$$


\section*{Acknowledgements}We are indebted to Minhyong Kim and Jan Vonk for helpful discussions, encouragement, and suggestions on the material in this paper. We thank the anonymous referees for several valuable comments on an earlier version of this manuscript, and Jan Steffen M\"uller and Jan Tuitman for numerous helpful discussions about $p$-adic heights. Part of this paper builds on the thesis of the second author, who is very grateful to his examiners Victor Flynn and Guido Kings for several suggestions which have improved the present work. The first author was supported by NSF grant DMS-1702196 and the Clare Boothe Luce Professorship (Henry Luce Foundation).  The second author was supported by the EPSRC during his thesis and by NWO/DIAMANT grant number 613.009.031.

\end{document}